\documentclass[a4paper,10pt]{amsart}  
\usepackage[latin1]{inputenc} 
\usepackage{amssymb}
\usepackage{colordvi}
\usepackage{graphicx,tikz}
\usetikzlibrary{patterns,calc}
\usepackage{vmargin}
\usepackage{todonotes}
\usepackage{hyperref}
\usepackage{overpic,graphics}
\usepackage{algpseudocode}
\usepackage{enumerate}
\usepackage{multirow}
\usepackage{mathtools}

\setmargrb{1in}{1in}{1in}{1in} 
\newtheorem{theorem}{Theorem}[section]
\newtheorem{proposition}[theorem]{Proposition}
\newtheorem{lemma}[theorem]{Lemma}
\newtheorem{corollary}[theorem]{Corollary}

\theoremstyle{definition}
\newtheorem{definition}[theorem]{Definition}

\newtheorem{algo}[theorem]{Algorithm}

\theoremstyle{definition}
\newtheorem{example}[theorem]{Example}
\newtheorem{remark}[theorem]{Remark}

\newtheorem{proofpart}{Part}
\newtheorem{proofcase}{Case}
\makeatletter
\@addtoreset{proofpart}{theorem}
\@addtoreset{proofcase}{theorem}
\makeatother

\newcommand{\NN}{\mathbb{N}}

\newcommand{\PP}{\mathbb{P}}

\newcommand{\HH}{\mathcal{H}\mathcal{H}}

\newcommand{\reverse}[1]{\overleftarrow{#1}} 

\DeclareMathOperator{\Cat}{Cat} 
\DeclareMathOperator{\inv}{inv} 
\DeclareMathOperator{\spe}{\sigma} 
\DeclareMathOperator{\card}{\#} 
\DeclareMathOperator{\horiz}{horiz} 

\newcommand{\wole}{\preccurlyeq} 
\newcommand{\woless}{\prec}
\newcommand{\wocover}{\prec\mathrel{\mkern-5mu}\mathrel{\cdot}}
\newcommand{\tc}[1]{#1^{\mathsf{tc}}} 
\newcommand{\meet}{\wedge} 
\newcommand{\join}{\vee} 
\newcommand{\maxs}{\Sigma} 
\newcommand{\pidown}{\pi_\downarrow}
\newcommand{\piup}{\pi_\uparrow}
\newcommand{\epidown}{\equiv_{\pidown}}
\newcommand{\epiup}{\equiv_{\piup}}

\newcommand{\vT}{{\mathcal T}} 

\newcommand{\streestovpaths}{{\varphi}} 
\newcommand{\streestovtrees}{{\widetilde \varphi}} 

\newcommand{\ie}{\textit{i.e.}~} 

\definecolor{darkblue}{rgb}{0,0,0.7} 
\newcommand{\darkblue}{\color{darkblue}} 
\newcommand{\defn}[1]{\emph{\darkblue #1}} 

\usepackage{todonotes}

\graphicspath{{figures/}}

\title[The $s$-weak order and $s$-permutahedra I]{The $s$-weak order and $s$-permutahedra I: \\
combinatorics and lattice structure}
\date{\today}

\author[C.~Ceballos]{Cesar Ceballos$^{\diamond}$}
\address[C.~Ceballos]{Institute of Geometry, TU Graz, Graz, Austria}
\email{\href{mailto:cesar.ceballos@tugraz.at}{\texttt{cesar.ceballos@tugraz.at}}}
\urladdr{http://www.geometrie.tugraz.at/ceballos/}

\author[V.~Pons]{Viviane Pons$^{\star}$}
\address[V.~Pons]{CNRS, IRL CRM Montr\'eal, Canada -- Universit\'e Paris-Saclay, CNRS,
Laboratoire Interdisciplinaire des Sciences du Num\'erique, Orsay, France.}
\email{\href{mailto:viviane.pons@lisn.upsaclay.fr}{\texttt{viviane.pons@lisn.upsaclay.fr}}}
\urladdr{\url{https://www.lri.fr/~pons/}}

\thanks{This project was supported by the 
ANR-FWF International Cooperation Project PAGCAP, funded by
the ANR Project ANR-21-CE48-0020 and
the FWF Project I 5788. 
Cesar Ceballos was also supported by the Austrian Science Fund FWF, Project P 33278.
}


\makeatletter
\def\l@section{\@tocline{1}{5pt}{0pc}{}{}}
\makeatother
\let\oldtocpart=\tocpart
\renewcommand{\tocpart}[2]{\bf\large\oldtocpart{#1}{#2}}
\let\oldtocsection=\tocsection
\renewcommand{\tocsection}[2]{\bf\oldtocsection{#1}{#2}}

\begin{document}

\maketitle

\begin{abstract}
This is the first contribution of a sequence of papers introducing the notions of $s$-weak order and $s$-permutahedra, certain discrete objects that are indexed by a sequence of non-negative integers~$s$. 
In this first paper, we concentrate purely on the combinatorics and lattice structure of the $s$-weak order, a partial order on certain decreasing trees which generalizes the classical weak order on permutations.
In particular, we show that the $s$-weak order is a semidistributive and congruence uniform lattice, generalizing known results for the classical weak order on permutations.

Restricting the $s$-weak order to certain trees gives rise to the $s$-Tamari lattice, a sublattice which generalizes the classical Tamari lattice.
We show that the $s$-Tamari lattice can be obtained as a quotient lattice of the~\mbox{$s$-weak}~order when~$s$ has no zeros, 
and show that the $s$-Tamari lattices (for arbitrary~$s$) are isomorphic to the~$\nu$-Tamari~lattices of Pr\'eville-Ratelle and Viennot.

The underlying geometric structure of the $s$-weak order will be studied in a sequel of this paper, where we introduce the notions of $s$-permutahedra and $s$-associahedra. 
\end{abstract}

\tableofcontents

\section*{Introduction}

The weak order is a partial order on the set of permutations of $[n]$, which is defined as the inclusion order of their corresponding inversion sets.
This partial order turns out to be a lattice, whose Hasse diagram can be realized geometrically as the edge graph of a polytope called the permutahedron.
Restricting the weak order to the set of $231$-avoiding permutations gives rise to the Tamari lattice, a well-studied lattice whose 
Hasse diagram can be realized as the edge graph of another polytope called the associahedron~\cite{hoissen12associahedra}.

\begin{figure}[htbp]
\begin{center}
\begin{tabular}{cc}
\includegraphics[height= 8cm]{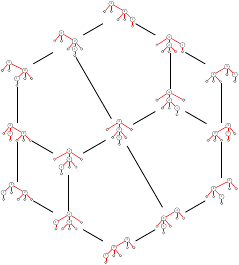}&
\qquad
\includegraphics[height= 8cm]{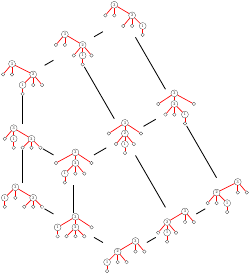}
\end{tabular}
\caption{The $s$-weak order and the $s$-Tamari lattice for $s=(0,2,2)$.
Their Hasse diagrams are the edge graphs of the $s$-permutahedron and the $s$-associahedron.}
\label{fig_s022_perm_asso}
\end{center}
\end{figure}

The Tamari lattice and the associahedron have been central objects in the combinatorial and discrete geometry communities for the past decades. See for example the collection of papers gathered in~\cite{hoissen12associahedra} or the recent survey on Loday's associahedron~\cite{PSZ23}. Many interesting connections have been made with the weak order, related in particular to Hopf algebraic structures~\cite{LodayRonco02_permutations,LodayRonco98_binarytrees,AguiarSottile_StructureLodayRonco,HNT05} and geometric realizations~\cite{HL07,HLT11}. The Tamari lattice is a lattice quotient of the weak order that can be realized as a generalized permutahedron, and it was proven in~\cite{PS19} 
that every lattice quotient of the weak order satisfies this property. 
The Tamari lattice also inspired the study of many more general objects such as Cambrian lattices~\cite{Reading-cambrianLattices} and the Permutree lattices~\cite{PP18}, which in turn connect to other subjects such as cluster algebras~\cite{FominZelevinsky-ClusterAlgebrasI,FominZelevinsky-ClusterAlgebrasII}. 

New specific interests in the Tamari lattice arose in connection to representation theory and generalizations of the theory of diagonal harmonics. For instance, the number of intervals in the Tamari lattice~\cite{chapoton_tamari_intervals_2005} is conjectured to be equal to the dimension of the alternating component of an $\mathfrak S_n$-module in the study of trivariate diagonal harmonics~\cite{haiman_conjectures_1994,BPR12}.
This motivated the introduction of the $m$-Tamari lattice in~\cite{BPR12}, whose number of intervals was conjectured to be the dimension of the alternating component of an $\mathfrak S_n$-module in higher trivariate diagonal harmonics. A formula for their enumeration and connections to representation theory can be found in~\cite{bousquet_representation_2013,bousquet_number_2011}.
Later on, the more general concept of the $\nu$-Tamari lattice was introduced in~\cite{PrevilleRatelleViennot} (which includes the $m$-Tamari and classical Tamari lattices). The connection to diagonal harmonics and representation theory is still unclear at this time, but several further investigations have been made, see for instance~\cite{HopfDreams_2022}. 
Other examples include the connections with the enumeration of non-separable planar maps in~\cite{fang_enumeration_2017}, and the tropical geometric realizations of $\nu$-Tamari lattices as the edge graph of a polytopal complex called the $\nu$-associahedron in~\cite{ceballos_vtamarivtrees_2018}. The interest for these new combinatorial and geometrical structures along with the numerous relations between the weak order and the Tamari lattice lead to a natural question: {\bf what is the equivalent of the weak order for the $\nu$-Tamari lattices?} In the case of $m$-Tamari, the question was briefly approached in an extended abstract~\cite{pons_lattice_2015}. The Hopf algebraic aspects were later studied in~\cite{novelli_hopf_2014}. 

This article is the first contribution of a sequence of papers which aims to answer this question by generalizing and completing the work started in~\cite{pons_lattice_2015}, while connecting it to interesting new geometrical structures in relation to the $\nu$-Tamari lattice and the $\nu$-associahedron. We define the \emph{$s$-weak order} which is indexed by a weak composition $s=(s(1),s(2),\dots, s(n))$. The case were $s = (1,1,\dots,1)$ is the classical weak order. In general, the $s$-weak order is a partial order on a set of certain trees which we call \emph{$s$-decreasing trees}, 
and is defined as the inclusion order of their \emph{$s$-tree inversions}.
We show that the $s$-weak order has the structure of a lattice (Theorem~\ref{thm:lattice}), and give a complete characterization of 
its cover relations in terms of a simple combinatorial rule on the trees (Theorem~\ref{thm:cover-relations}).
We also study its lattice properties and show that it is semidistributive and congruence \mbox{uniform (Theorem~\ref{thm:hh})}. 
 
Restricting the $s$-weak order to certain trees, which play the role of $231$-avoiding permutations, 
gives rise to a sublattice which we call the \emph{$s$-Tamari lattice} (Theorem~\ref{thm:stam-sublattice}).
We show that if $s$ contains no zeros, then the $s$-Tamari lattice can also be obtained as a quotient lattice of the $s$-weak order (Theorem~\ref{thm_sTam_quotientLattice}).
We also give a simple characterization of the cover relations of the $s$-Tamari lattice (Theorem~\ref{thm_sTamari_covering_relations}),
and show that it is isomorphic to a well chosen \emph{$\nu$-Tamari lattice} (Theorem~\ref{thm_sTam_vTam_paths}).

Some of the connections between the Tamari lattice and the theory of diagonal harmonics mentioned above were originally discovered in connection to the study of Garcia--Haiman $n!$-conjecture~\cite{garsia_graded_1993,garsia_natural_1996}, which was later proved by Haiman in~\cite{haiman_hilbert_2001}. Their result implies, in particular, that the dimension of certain Garcia--Haiman space is equal to $n!$. 
On the other hand, the number of elements of the $s$-weak order is equal to the \emph{$s$-factorial number} 
$s!:=1\cdot (s(n)+1) \cdot (s(n-1)+s(n)+1)\cdot \dots \cdot (s(2)+\dots+s(n)+1)$.
For instance, if $s=(m,m,\dots,m)$ consists of $n$ copies of $m$, then $s!=\prod_{i=0}^{n-1}(1+im)$. 
It would be interesting to investigate whether these $s$-factorial numbers appear as the dimensions of interesting~$S_n$ modules in representation theory. 
We leave this as an open problem for the interested reader.

As in the classical case, the $s$-weak order and the $s$-Tamari lattice possess beautiful underlying geometric structures, which are apparent in Figure~\ref{fig_s022_perm_asso} and Figure~\ref{fig:3d-example}. 
These structures, which we call the \emph{$s$-permutahedron} and the \emph{$s$-associahedron}, will be introduced and studied in a sequel of this work. 
This first paper may be regarded as the combinatorial foundations of the theory, while the second part will have a more geometric flavour. 

The case where $s=(m,\dots, m)$ corresponds to the afore mentioned paper~\cite{pons_lattice_2015}, where it was called the ``metasylvester lattice''. Note that \cite{pons_lattice_2015} has only been published as an extended abstract so many of the proofs, even for the $m$ case, are actually novel to this paper. Furthermore, note that this article is the first part of a long version of an extended abstract of our work, which was presented at the FPSAC 2019 conference in Ljubljana and is available in~\cite{FPSAC2019}. A blog post about the story behind the talk appears in~\cite{BlogViv}.

Beside the paper, we provide the SageMath~\cite{SageMath2022} code used to reach the results as well as a demo SageMath worksheet with computations of all examples in this paper~\cite{SageDemoI}.

\begin{figure}[ht]
\begin{tabular}{cc}
\includegraphics[height=7.5cm]{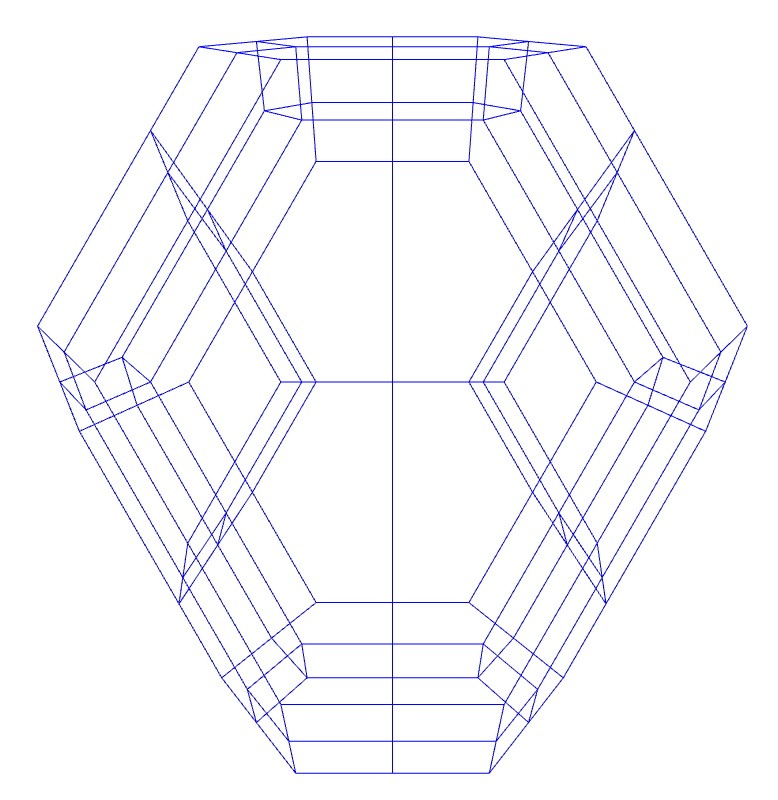}
&
\includegraphics[height=7.5cm]{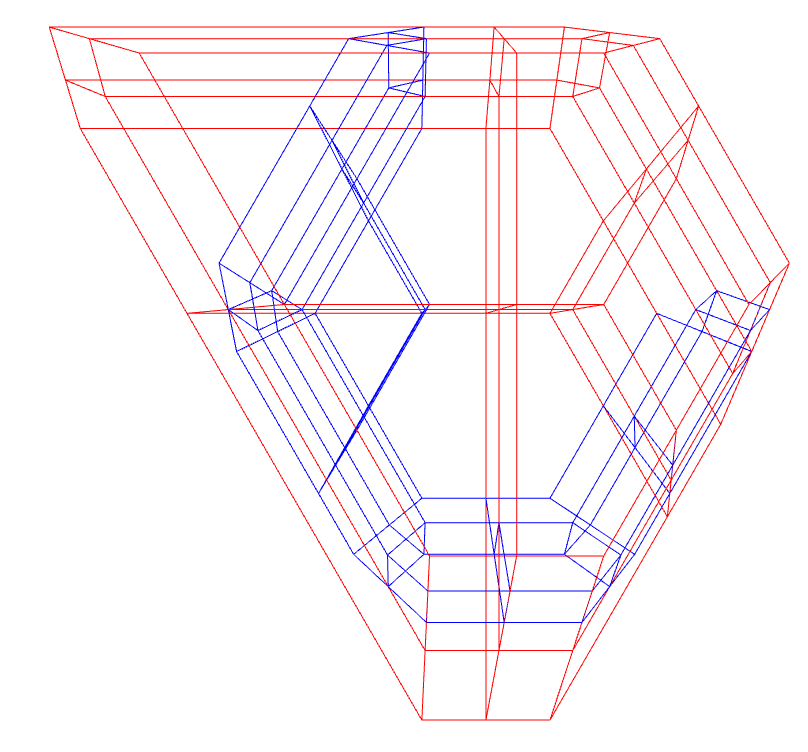}
\end{tabular}
\caption{Geometric realizations of the 3-dimensional $s$-permutahedron and \mbox{$s$-associahedron}, for $s=(0,2,2,2)$. 
Their edge graphs realize the Hasse diagrams of the $s$-weak order and the $s$-Tamari lattice, respectively.
}
\label{fig:3d-example}
\end{figure}

\section{The $s$-weak order}

We denote by $\PP$ the set of positive integers, by $\NN$ the set of nonnegative integers and for any 
$n \in \PP$ let $[n]:=\{1,2,\dots,n\}$. A \defn{weak composition} is a finite sequence 
$\mu=(\mu(1),\mu(2),\dots, \mu(n))$ of numbers~\mbox{$\mu(i) \in \NN$}. For a weak composition $\mu$, we 
define its \defn{weight} $|\mu|:=\sum_i \mu(i)$ and its \defn{length} $\ell(\mu):=n$. For example 
for $\mu=(2,0,3,4,0,1)$ we have that $\ell(\mu)=6$ and $|\mu|=10$.
A \defn{composition} is a weak composition~$\mu$ such that $\mu(i)\ne 0$ for all $i\in [\ell]$, in 
other words, a composition is a finite sequence of entries in~$\PP$. 

\subsection{$s$-permutations}
Let $s=(s(1),s(2),\dots,s(n))$ be a composition (\emph{without zero entries}). 
an $s$-permutation is any permutation of the word $1^{s(1)}2^{s(2)}\dots n^{s(n)}$. We say that an $s$-permutation is $121$-avoiding (respectively $132$-avoiding) if it does not contain any subsequence of the form~$a\dots b \dots a$ for $a<b$ (respectivelty $a\dots c \dots b$ for $a<b<c$). 
We will often refer to $s$ as the \defn{{\bf s}ignature}. This is the reason why we use the letter $s$.

The $s$-permutations appear in the work on Hopf algebras of Novelli and Thibon~\cite{novelli_hopf_2014}, who study generalizations of the Hopf algebra of permutations for ``$m$-permutations" of $1^{m}2^{m}\dots n^{m}$, and of the Hopf algebra on planar binary trees for $(m+1)$-ary trees. The authors introduce a notion of metasylvester congruence on permutations that allows them to obtain Hopf algebras based on decreasing trees. The element representatives of the congruence classes are $121$-avoiding $s$-permutations for $s=(m,m,\dots,m)$, see~\cite[Proposition~3.10]{novelli_hopf_2014}. These 121-avoiding $m$-permutations and associated lattices called metasylvester lattices were also discovered and studied by Pons in~\cite{pons_lattice_2015}. One interesting property of this lattice is that the $m$-Tamari lattice can be obtained both as a sublattice and as a quotient lattice of it~\cite[Theorem~4.7]{pons_lattice_2015}. 

In this section we generalize the results in~\cite{pons_lattice_2015} for general signatures $s$. 
This gives a natural notion of ``rational" metasylvester lattices in the case where the entries of $s+\bf 1$  count the numbers of boxes in each of the rows of an~$a\times b$ rectangle that are crossed by its main diagonal,
when $a\geq b$ are relatively prime. 
It turns out that the number of $132$-avoiding permutations in this case, see e.g~\cite{ceballos_signature_2018}, coincides with the rational Catalan number 
\[
\Cat(a,b) = \frac{1}{a+b} {a+b \choose a}.
\]
The rational Catalan elements also have a natural lattice structure described by Pr\'eville-Ratelle and Viennot in~\cite{PrevilleRatelleViennot}, which we refer to as the \emph{rational Tamari lattice}. More generally, they introduce a \emph{$v$-Tamari lattice} for arbitrary lattice paths $v$ consisting of east and north steps in the plane. 
As we will see in Section~\ref{sec_sTamarivTamari},
our signature $s$ can be thought as the composition that counts the number of east steps of $v$ in each of the rows  starting from top to bottom. 
   
In this set up, it is natural to study the case where the signature $s$ has some zero entries, corresponding for example to the case when $v$ has two consecutive north steps. In this case, there is a row with zero number of east steps that will contribute a zero to the signature. 
If the signature $s$ has zero entries, 
it is difficult to define well behaved notions of 121- and 132-avoiding $s$-permutations. 
In this case, these notions 
can be described more naturally in the context of $s$-decreasing labeled trees.

\subsection{$s$-decreasing trees}
Let $s=(s(1),s(2),\dots,s(n))$ be a weak composition (\emph{with possible zero entries}). A \defn{$s$-decreasing tree} is a tree with $n$ internal nodes labeled from $1$ to $n$, such that node $i$ has $s(i)+1$ children and all its descendants have smaller labels. It is straightforward to check that the number of $s$-decreasing trees is equal to the \defn{$s$-factorial number} defined by 
$$1\cdot (s(n)+1) \cdot (s(n-1)+s(n)+1)\cdot \dots \cdot (s(2)+\dots+s(n)+1).$$ 
If $s$ has no entries equal to zero, there is a natural bijection between $s$-decreasing trees and $s$-permutations, which is given by labeling the ``caverns" of the tree with the label of its node and reading them in \emph{preorder}. An example of this procedure is illustrated in Figure~\ref{fig:s-perm_from_tree}. We write $\spe(T)$ the 121-avoiding $s$-permutation corresponding to the $s$-decreasing tree $T$. In particular, it is clear that $\spe(T)$ avoids $121$: the numbers in between two occurrences of a letter $a$ correspond to some descendants of the node labeled by $a$ and are by definition smaller than $a$. 

\begin{figure}[htbp]
\begin{center}

\begin{tabular}{c}
{ \newcommand{\nodea}{\node[draw,circle] (a) {$4$}
;}\newcommand{\nodeb}{\node[draw,circle] (b) {$3$}
;}\newcommand{\nodec}{\node[draw,circle] (c) {$ $}
;}\newcommand{\noded}{\node[draw,circle] (d) {$1$}
;}\newcommand{\nodee}{\node[draw,circle] (e) {$ $}
;}\newcommand{\nodef}{\node[draw,circle] (f) {$ $}
;}\newcommand{\nodeg}{\node[draw,circle] (g) {$ $}
;}\newcommand{\nodeh}{\node[draw,circle] (h) {$ $}
;}\newcommand{\nodei}{\node[draw,circle] (i) {$2$}
;}\newcommand{\nodej}{\node[draw,circle] (j) {$ $}
;}\newcommand{\nodeba}{\node[draw,circle] (ba) {$ $}
;}\newcommand{\cavb}{\node[red] {$3$}
;}\newcommand{\cavd}{\node[red] {$1$}
;}\newcommand{\cava}{\node[red] {$4$}
;}\newcommand{\cavi}{\node[red] {$2$}
;}\begin{tikzpicture}[auto]
\matrix[column sep=.3cm, row sep=.3cm,ampersand replacement=\&]{
         \&         \&         \&         \&         \& \nodea  \&         \&         \&         \\ 
         \&         \& \nodeb  \&         \&    \cava     \& \nodeh  \&  \cava       \& \nodei  \&         \\ 
 \nodec  \&    \cavb     \& \noded  \&    \cavb     \& \nodeg  \&         \& \nodej  \&    \cavi     \& \nodeba \\ 
         \& \nodee  \&   \cavd      \& \nodef  \&         \&         \&         \&         \&         \\
};

\path[ultra thick, red] (d) edge (e) edge (f)
	(b) edge (c) edge (d) edge (g)
	(i) edge (j) edge (ba)
	(a) edge (b) edge (h) edge (i);
\end{tikzpicture}} \\
$\spe(T) = 313442$
\end{tabular}

\caption{Bijection from $s$-decreasing trees to $121$-avoiding $s$-permutations, with $s = (1,1,2,2)$}
\label{fig:s-perm_from_tree}
\end{center}
\end{figure}
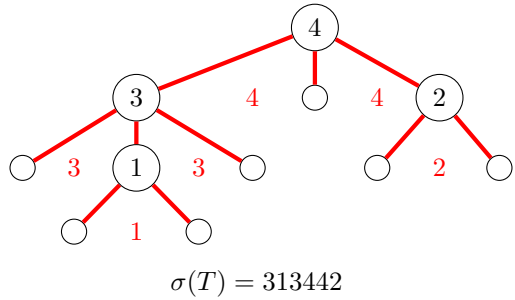

Note that each node is given a unique label. For practicality, we then write \emph{the node $y$} instead of \emph{the node labeled by $y$}. Our goal now is to define a partial order on $s$-descreasing trees. Let us first introduce some useful notations.

\begin{definition}
\label{def:tree-notations}
Let $T$ be an $s$-decreasing tree and $y \in T$ with $m = s(y)$, we write $T_0^y, \dots, T_{m}^y$ the children subtrees from left to right of the node $y$. 
Note that in this definition, a \defn{child} always means the full subtree and not just a node.
\begin{enumerate}
\item $T_0^y$ is the \defn{left child} of $y$. If $s(y) > 0$, then $T_0^y$ is a \defn{strict left child} of $y$;
\item $T_{m}^y$ is the \defn{right child} of $y$. If $s(y) > 0$, then $T_{m}^y$ is a \defn{strict right child} of $y$.
\item $T_i^y$ for $i \neq 0$ and $i \neq m$ is a \defn{middle child} of $y$.
\item If $s(y) = 0$ then $T_0^y$ is the \defn{unique child} of $y$.
\end{enumerate}

Note that a unique child is also both a left and a right child but not a strict left child nor a strict right child. In addition, we say that a node $x$ is
\begin{itemize}
\item \defn{right} (resp. \defn{left}) of a node $y$ if $x$ is not a descendant of $y$ and if their first common ancestor $z$ is such that $x \in T_j^z$ and $y \in T_i^z$ with $j > i$ (resp. $j < i$);
\item \defn{weakly right} (resp. \defn{weakly left}) of a subtree $T_i^y$ if $x \in T_j^y$ with $j \geq i$ (resp. $j \leq i$) or if $x$ is right (resp. left) of $y$.
\item \defn{strongly right} (resp. \defn{strongly left}) of a a subtree $I_i^y$ if $x \in T_j^y$ with $j > i$ (resp. $j < i$) or $x$ is right (resp. left) of $y$.
\end{itemize} 
\end{definition}

An example is given on Figure~\ref{fig:ex-dtree}. We have that $3, 2, $ and $1$ are right of $4$ and also strongly right of $T_0^5$ and weakly right of $T_2^5$. 

\begin{remark}
If $x < y$, then there are only 3 options for $x$: being left of $y$, being a descendant of $y$, being right of $x$. Similarly, for $0 \leq i \leq s(y)$, if $x$ is not weakly right of $T_i^y$ (resp. weakly left), then $x$ is strongly left (strongly right). 
\end{remark}

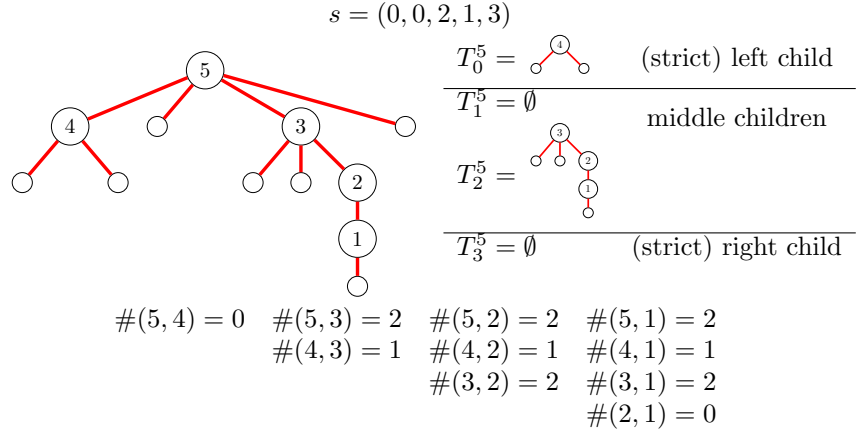
\begin{figure}[ht]
$s = (0,0,2,1,3)$

\begin{tabular}{clc}
\multirow{4}{*}{\scalebox{.8}{{ \newcommand{\nodea}{\node[draw,circle] (a) {$5$}
;}\newcommand{\nodeb}{\node[draw,circle] (b) {$4$}
;}\newcommand{\nodec}{\node[draw,circle] (c) {$ $}
;}\newcommand{\noded}{\node[draw,circle] (d) {$ $}
;}\newcommand{\nodee}{\node[draw,circle] (e) {$ $}
;}\newcommand{\nodef}{\node[draw,circle] (f) {$3$}
;}\newcommand{\nodeg}{\node[draw,circle] (g) {$ $}
;}\newcommand{\nodeh}{\node[draw,circle] (h) {$ $}
;}\newcommand{\nodei}{\node[draw,circle] (i) {$2$}
;}\newcommand{\nodej}{\node[draw,circle] (j) {$1$}
;}\newcommand{\nodeba}{\node[draw,circle] (ba) {$ $}
;}\newcommand{\nodebb}{\node[draw,circle] (bb) {$ $}
;}\begin{tikzpicture}[auto]
\matrix[column sep=.3cm, row sep=.3cm,ampersand replacement=\&]{
         \&         \&         \&         \& \nodea  \&         \&         \&         \&         \\ 
         \& \nodeb  \&         \& \nodee  \&         \&         \& \nodef  \&         \& \nodebb \\ 
 \nodec  \&         \& \noded  \&         \&         \& \nodeg  \& \nodeh  \& \nodei  \&         \\ 
         \&         \&         \&         \&         \&         \&         \& \nodej  \&         \\ 
         \&         \&         \&         \&         \&         \&         \& \nodeba \&         \\
};

\path[ultra thick, red] (b) edge (c) edge (d)
	(j) edge (ba)
	(i) edge (j)
	(f) edge (g) edge (h) edge (i)
	(a) edge (b) edge (e) edge (f) edge (bb);
\end{tikzpicture}}
}}&
$T_0^5 = \begin{aligned}\scalebox{.4}{{ \newcommand{\nodea}{\node[draw,circle] (a) {$4$}
;}\newcommand{\nodeb}{\node[draw,circle] (b) {$ $}
;}\newcommand{\nodec}{\node[draw,circle] (c) {$ $}
;}\begin{tikzpicture}[auto]
\matrix[column sep=.3cm, row sep=.3cm,ampersand replacement=\&]{
         \& \nodea  \&         \\ 
 \nodeb  \&         \& \nodec  \\
};

\path[ultra thick, red] (a) edge (b) edge (c);
\end{tikzpicture}}}\end{aligned} $ & (strict) left child \\ \cline{2-3}
 & $T_1^5 = \emptyset $ & \multirow{2}{*}{middle children} \\
 & $T_2^5 = \begin{aligned}\scalebox{.4}{{ \newcommand{\nodea}{\node[draw,circle] (a) {$3$}
;}\newcommand{\nodeb}{\node[draw,circle] (b) {$ $}
;}\newcommand{\nodec}{\node[draw,circle] (c) {$ $}
;}\newcommand{\noded}{\node[draw,circle] (d) {$2$}
;}\newcommand{\nodee}{\node[draw,circle] (e) {$1$}
;}\newcommand{\nodef}{\node[draw,circle] (f) {$ $}
;}\begin{tikzpicture}[auto]
\matrix[column sep=.3cm, row sep=.3cm,ampersand replacement=\&]{
         \& \nodea  \&         \\ 
 \nodeb  \& \nodec  \& \noded  \\ 
         \&         \& \nodee  \\ 
         \&         \& \nodef  \\
};

\path[ultra thick, red] (e) edge (f)
	(d) edge (e)
	(a) edge (b) edge (c) edge (d);
\end{tikzpicture}}
}\end{aligned}$ & \\ \cline{2-3}
 & $T_3^5 = \emptyset $ & (strict) right child \\
\end{tabular}

\vspace{.5cm}
\begin{tabular}{llll}
$\card(5,4) = 0$ & $\card(5,3) = 2$ & $\card(5,2) = 2$ & $\card(5,1) = 2$ \\
                 & $\card(4,3) = 1$ & $\card(4,2) = 1$ & $\card(4,1) = 1$ \\
                 &                  & $\card(3,2) = 2$ & $\card(3,1) = 2$ \\
                 &                  &                  & $\card(2,1) = 0$
\end{tabular}
\caption{an $s$-decreasing tree and its tree-inversions}
\label{fig:ex-dtree}
\end{figure}

We can now introduce a fundamental notion on $s$-decreasing trees. 

\begin{definition}[Tree-inversions]
\label{def:tree-inversions}
For all $y > x$ in an $s$-decreasing tree $T$, we define the \defn{cardinality} $\card_T(y,x)$ of $(y,x)$ in $T$ by 
\begin{enumerate}
\item if $s(y) = 0$ or if $x$ is weakly left of $T_0^y$, then $\card_T(y,x) = 0$;
\label{def-cond:tree-inversion0}
\item if $x \in T_i^y$ with $T_i^y$ a middle child of $y$, then $\card_T(y,x) = i$;
\label{def-cond:tree-inversionm}
\item if $x$ is weakly right of $T_{s(y)}^y$, then $\card_T(y,x) = s(y)$.  
\label{def-cond:tree-inversionr}
\end{enumerate}

This covers all positions of $x$ relatively to $y$. If $\card_T(y,x) > 0$, we say that $(y,x)$ is a \defn{tree-inversion} of $T$ and we write $\inv(T)$ the multi-set of tree-inversions counted with their cardinalities. When the context is clear, we omit $T$ and simply write $\card(y,x)$.
\end{definition}

See for example all the tree-inversions corresponding to a given tree on Figure~\ref{fig:ex-dtree}. Note that in the case where $s(y) =0$ and $x$ is right of $y$, both conditions \eqref{def-cond:tree-inversion0} and \eqref{def-cond:tree-inversionr} apply and are consistent with $\card_T(y,x) = 0$. In particular, there is no tree-inversion $(y,x)$ if $s(y) = 0$ even if $x$ is right of $y$. If $s$ is a composition, this definition coincides with the one given in \cite{pons_lattice_2015}. In particular, when $s = (1,1,\dots, 1)$ then the $s$-decreasing trees are decreasing binary trees, in bijection with permutations. In this case, the tree-inversions of $T$ are exactly the inversions of $\spe(T)$ taken in terms of values. More generally, we consider tree-inversions a natural analogue of classical inversions of permutations. As such, the sets of tree-inversions can be characterized, following some rules very similar to the ones of classical permutations inversions sets.

\begin{definition}
\label{def:multi-inversion-set}
A \defn{multi inversion set} $I$ on $1, \dots, n$ is a multi set of inversions $(y,x)$ with $1 \leq x < y \leq n$. We write $\card_I(y,x)$ the number of occurrences of $(y,x)$ in $I$. If there is no occurrence of $(y,x)$ in $I$ we write $(y,x) \notin I$ or equivalently $\card_I(y,x) = 0$.

Given two multi inversion sets $I$ and $J$, we say that $I$ is \defn{included} in $J$ (resp. strictly included) and write $I \subseteq J$ (resp. $I \subset J$) if $\card_I(y,x) \leq \card_J(y,x)$ (resp. $\card_I(y,x) < \card_J(y,x)$) for all $1 \leq x < y \leq n$.

Given a weak composition $s$ with $\ell(s) = n$, we write $\maxs_s$ the \defn{maximal $s$-inversion set} defined by  $\card_{\maxs_s}(y,x) = s(y)$ for all $1 \leq x < y \leq n$.
\end{definition}

\begin{definition}
\label{def:tree-inversion-set}
Let $s$ be a weak composition with $\ell(s) = n$ and $I \subseteq \maxs_s$ a multi inversion set. Then $I$ is said to be a \defn{$s$-tree-inversion set} if it satisfies these two following rules.
\begin{itemize}
\item \defn{Transitivity}: if $a < b < c$ with $\card(c,b) = i$, then $\card(b,a) = 0$ or $\card(c,a) \geq i$.
\item \defn{Planarity}: if $a < b < c$ with $\card(c,a) = i$, then $\card(b,a) = s(b)$ or $\card(c,b) \geq i$.
\end{itemize}
\end{definition}

The transitivity and planarity conditions can be easily understood from the geometric drawing of the tree, as illustrated in Figure~\ref{fig:transitivity_planarity}. The transitivity is about the two possible positions of $a$ with respect to~$b$ and $c$, while planarity is about the two possible positions of $b$ with respect to $a$ and $c$. 
In particular, one can check that $\maxs_s$ is an $s$-tree-inversion set. Note that when $\card(c,a) = 0$, the planarity
is always satisfied on $a < b < c$. When $s = (1,1,\dots,1)$, these two rules correspond to the characterisation of a permutation inversion set. Transitivity is then classical transitivy of inversions: if $(c,b) \in \inv(\sigma)$ and $(b,a) \in \inv(\sigma)$ then $(c,a) \in \inv(\sigma)$. And planarity is the rule that says that when $(c,a) \in \inv(\sigma)$ then
either $(c,b) \in \inv(\sigma)$ or $(b,a) \in \inv(\sigma)$, depending on where is the value $b$ in the permutation. In general, we have the following proposition.

\begin{figure}[htbp]
\begin{center}
\includegraphics[width = 0.5\textwidth]{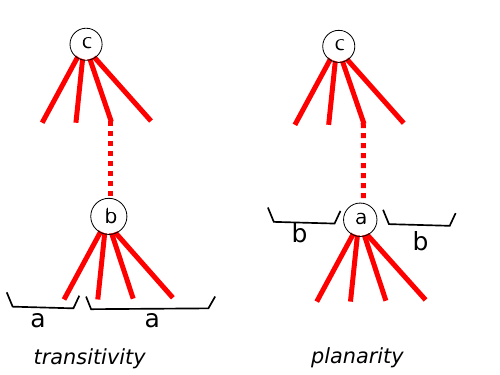}
\caption{Illustration of the transitivity and planarity conditions on $s$-tree inversion sets.}
\label{fig:transitivity_planarity}
\end{center}
\end{figure}

\begin{proposition}
\label{prop:tree-inversions-bij}
For $s$ a given weak composition, 
the multiset of tree-inversions of any $s$-decreasing tree is an $s$-tree-inversion set. 
Moreover, this gives a bijection between $s$-decreasing trees and $s$-tree-inversions sets. 
\end{proposition}

We will use the following Lemma which is immediate from Definition \ref{def:tree-inversions} by following rules \eqref{def-cond:tree-inversionm} or~\eqref{def-cond:tree-inversionr}.

\begin{lemma}
\label{lem:tree-inversion-trivial}
For $T$ an $s$-decreasing tree of some weak composition $s$ and $i > 0$, then $\card_T(y,x) \geq i$ if and only if $s(y) \geq i$ and $x$ is weakly right of $T_i^y$.
\end{lemma}

\begin{proof}[Proof of Proposition \ref{prop:tree-inversions-bij}]
First, let $T$ be an $s$-decreasing tree, $I$ its multiset of tree-inversions and let us check that $I$ is an $s$-tree-inversion set. By definition, it is clearly included in $\maxs_s$.
\begin{itemize}
\item Transitivity: if $\card(c,b) = 0$, there is nothing to prove. We suppose $\card(c,b)  = i > 0$ which means that $b$ is weakly right of $T_i^c$. Now either $a$ is left of $b$ which implies $\card(b,a) = 0$, or $a$ is a descendant of $b$, or $a$ is right of $b$. It can be checked that the last two cases imply that $a$ is also weakly right of $T_i^c$ which means $\card(c,a) \geq i$.
\item Planarity: suppose that $\card(c,a) = i > 0$ which means that $a$ is weakly right of $T_i^c$. Now let us look at the different possibilities for $b$. Either $b$ is weakly right of $T_i^c$ which implies $\card(c,b) \geq i$ by Lemma \ref{lem:tree-inversion-trivial}. If not, $b$ is then strongly left of $T_i^c$. In particular, $b$ is left of $a$ and then $a$ is weakly right of $T_{s(b)}^b$ which gives  $\card(b,a) = s(b)$.
\end{itemize}

The application that takes an $s$-decreasing tree and sends it to its tree-inversion set is injective. Indeed take $T$ and $T'$ such that $\inv(T) = \inv(T')$. Being decreasing trees, both their roots are labeled with~$n$. Let $a \in [n-1]$. The node $a$ belongs to some subtree $T_i^n$ of $T$ and ${T'}_j^n$ of $T'$. By definition of tree-inversions, we have $i = \card_T(n,a) = \card_{T'}(n,a) = j$. This means that for all $0 \leq i \leq s(n)$, the labeling set of the subtree $T_i^n$ is equal to the labeling set of the subtree ${T'}_i^n$. In particular, by the decreasing rule, both subtrees have the same root. We can then apply this same reasoning recursively to show that $T = T'$.

To prove that the construction is surjective, we now take $I$ an $s$-tree-inversion set and construct an $s$-decreasing tree $T$ for which $I$ is the set of tree-inversions. We use the following recursive algorithm that we call on $C = [n]$.

\begin{algo}[ConstructTree]
\label{algo:tree-inversions}

~

Given a weak composition $s$ and an $s$-tree-inversion set $I$.
\begin{algorithmic}
\Require $C \subseteq [n]$
\Ensure a decreasing tree $T$ labeled by $C$
\If{$C$ is empty}
    \State return the empty tree
\EndIf
\State $c \gets \max(C)$ 
\State $C_0 \gets \emptyset, \dots C_{s(c)} \gets \emptyset$
\For{$a$ in $C \setminus c$}
	\State add $a$ in $C_i$ where $i = \card_I(c,a)$
\EndFor
\State return the tree with root $c$ and subtrees $ConstructTree(C_0), \dots ConstructTree(C_{s(c)})$
\end{algorithmic}
\end{algo}

Basically, the algorithm partitions $[n-1]$ into $s(n)+1$ subsets which will give recursively the $s(n)+1$ subtrees $T_0^n, \dots T_{s(n)}^n$. It is clear that the result is an $s$-decreasing tree (at each step, we label the root with the local maximum and the number of subtrees we create is given by $s$). We have to check that the tree-inversions of $T$ are indeed given by $I$. Let $x < y$. We distinguish three cases which corresponds to the three possible positions of $x$ relatively to $y$.

\begin{proofcase} \textbf{$x$ is left of $y$.} In this case, we have $\card_T(y,x) = 0$. Let $z$ be the first common ancestor of $x$ and $y$ (closest to $x$ and $y$). We have $x \in T_i^z$ and $y \in T_j^z$ with $i < j$. We have constructed $T$ from the algorithm, in particular the subsets of the trees $T_i^z$ and $T_j^z$ have been constructed when the algorithm was running on the subset of the tree rooted in $z$. This gives that $\card_I(z,x) = i$ and  $\card_I(z,y) =j$. The transitivity rule of Definition \ref{def:tree-inversion-set} then says that if $\card_I(y,x) > 0$, then $i = \card_I(z,x) \geq \card_I(z,y) = j$ which contradicts the fact that $i < j$. We conclude that in this case $\card_I(y,x) = 0 = \card_T(y,x)$.
\end{proofcase}

\begin{proofcase} \textbf{$x$ is a descendant of $y$.} The node $x$ belongs so some subtree $T_i^y$ and so we have $\card_T(y,x) = i$. By the construction of $T$, we also have $\card_I(y,x) = i$. Otherwise, $x$ would have been added in a different subset.
\end{proofcase}

\begin{proofcase} \textbf{$x$ is right of $y$.} In this case, we have $\card_T(y,x) = s(y)$.  Similarly to case 1, we take $z$ to be the first common ancestor to $x$ and $y$ and we have $x \in T_j^z$ and $y \in T_i^z$ with $i < j$. In particular, $\card_I(z,x) = j$ and $x < y < z$. We follow the planarity rule of Definition \ref{def:tree-inversion-set}. Because $\card_I(z,y) = i < j$, then $\card_I(y,x) = s(b) = \card_T(y,x)$.
\end{proofcase}

In all three cases, $\card_I(y,x) = \card_T(y,x)$. We conclude that Algorithm~\ref{algo:tree-inversions} is indeed the inverse construction of an $s$-decreasing tree from its tree-inversions which proves the proposition.
\end{proof}

\subsection{The $s$-weak order or the lattice of $s$-decreasing trees}

Tree-inversions are an analogue of permutation inversions: this motivates the following definition of an analogue of the weak order on $s$-decreasing trees.

\begin{definition}
Let $R$ and $T$ be two $s$-decreasing trees of a same weak composition $s$ with $\ell(s) = n$. We say that $R \wole T$ if 
\begin{equation}
\inv(R) \subseteq \inv(T)
\end{equation}
using the inclusion of multi inversion sets from Definition \ref{def:multi-inversion-set}. We call the relation $\wole$ the \defn{$s$-weak order}.
\end{definition}

\begin{lemma}
For a given weak composition $s$, the relation $\wole$ is a partial order on $s$-decreasing trees.
\end{lemma}

\begin{proof}
It is immediate from the definition that $\wole$ is both reflexive and transitive. It is also clear that if $R \wole T$ and $T \wole R$, then $\inv(T) = \inv(R)$ (the cardinality is the same for all the tree-inversions). By Proposition \ref{prop:tree-inversions-bij} then $T = R$.
\end{proof}

\begin{figure}[htbp]
\begin{center}
\begin{tabular}{cc}
\includegraphics[height=8cm]{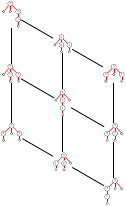} &
\includegraphics[height=8cm]{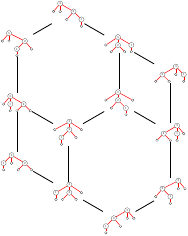} \\
$s = (0,0,2)$ & $s = (0,1,2)$ \\
\includegraphics[height=8cm]{sweaklattice_s022.pdf} &
\includegraphics[height=8cm]{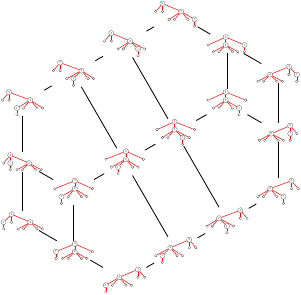} \\
$s = (0,2,2)$ & $s = (0,3,2)$
\end{tabular}

\caption{Examples of $s$-weak lattices.}
\label{fig:sweak_lattice}
\end{center}

\end{figure}

\begin{proposition}
\label{prop:symmetry}
The $s$-weak order is symmetric, \ie there is an isomorphism $\phi$ from $s$-decreasing trees to $s$-decreasing trees such that $T \wole R$ if and only if $\phi(R)  \wole \phi(T)$. 
\end{proposition}

\begin{proof}
This is immediate with $\phi$ being the involution that the reflects the tree along a vertical line (swaps horizontally). 
For an $s$-decreasing tree $T$, we have $\card_{\phi(T)}(b,a) = s(b) - \card_T(b,a)$.
Therefore, $T \wole R$ if and only if $\phi(R)  \wole \phi(T)$.
\end{proof}

You can see different examples on Figure~\ref{fig:sweak_lattice}. Note that when $s = (1,1, \dots, 1)$, this is actually the classical weak order on permutations which is known to be a lattice. In the case where $s =(m, m, \dots, m)$, this is the \emph{metasylvester lattice} defined and proven to be a lattice in \cite{pons_lattice_2015}. 
The goal of this section is to prove that the $s$-weak order is a lattice for general $s$.  

In the case where $s$ is a strict composition (i.e. $s$ contains no zeros), it is possible to see the $s$-weak order is a suborder of the weak order on permutations. Using the natural bijection between $s$-decreasing trees and $s$-permutations (see Figure~\ref{fig:s-perm_from_tree}), one can see that it is actually a \emph{sublattice} of some ideal of the weak order (slightly generalizing the results of \cite{pons_lattice_2015}). The more interesting case is when~$s$ contains some zero values. Then, there is no natural definition of $s$-permutations, and so, no easy bijection that would allow us to see the $s$-weak order as a sublattice of the weak order. This is why we use a more general approach here.

We will re-enact the proof that the weak order is a lattice in a more general setting. Indeed, we use the characterization of $s$-decreasing trees through their $s$-tree-inversion sets. The set of multi inversion sets is endowed with a natural poset structure through inclusion of multisets. 
We will use the notions of \emph{union} and \emph{transitive closure} in order to prove that the tree-inversion sets induce a lattice.

\begin{definition}
\label{def:union-tree-inversions}
The \defn{union} of two multi inversion sets $I$ and $J$ is given by 
\begin{equation}
\card_{I \cup J} (y,x) = \max \left( \card_{I}(y,x), \card_{J}(y,x) \right)
\end{equation}
for all $x < y$. 
\end{definition}

\begin{remark}
\label{rem:union-tree-inversions}
Note that the definition has been chosen so that $I \cup J$ is the smallest multi inversion set in terms of inclusion such that $I \subseteq I \cup J$ and $J \subseteq I \cup J$.
\end{remark}

\begin{definition}
\label{def:tc-tree-inversions}
 Let $I$ be a multi inversion set. A \defn{transitivity path} between two values $c > a$ is a list of $k \geq 2$ values $c = b_1 > b_2 > \dots > b_k = a$ such that $\card_I(b_i, b_{i+1}) > 0$ for all $1 \leq i < k$. Note that if $\card_I(c,a) > 0$, then $(c,a)$ itself is a transitivity path. We write $\tc{I}$ the \defn{transitive closure} of $I$ and define for all $a < c$, $\card_{\tc{I}}(c,a) = v$ with $v$ the maximal value of $\card_I(b_1,b_2)$ for $c = b_1 > \dots > b_k = a$ a transitivity path between $c$ and $a$.
\end{definition}

\begin{example}
Let $I$ be the following multi inversion set:

\begin{tabular}{lll}
$\card_I(4,3) = 3$ & $\card_I(4,2) = 2$ & $\card_I(4,1) = 0$ \\
                   & $\card_I(3,2) = 1$ & $\card_I(3,1) = 0$ \\
                   &                    & $\card_I(2,1) = 2$ 
\end{tabular}

Then $\tc{I}$ is given by

\begin{tabular}{lll}
$\card_{\tc{I}}(4,3) = 3$ & $\card_{\tc{I}}(4,2) = 3$ & $\card_{\tc{I}}(4,1) = 3$ \\
                          & $\card_{\tc{I}}(3,2) = 1$ & $\card_{\tc{I}}(3,1) = 1$ \\
                          &                           & $\card_{\tc{I}}(2,1) = 2$ 
\end{tabular}

Indeed, 
\begin{itemize}
\item there are two transitivity paths between $4$ and $1$: $ 4 > 3 > 2 > 1$ with $\card_I(4,3) = 3$, and $4 > 2 > 1$ with $\card_I(4,2) = 2$. This gives $\card_{\tc{I}}(4,1) = 3$.
\item There are two transitivity paths between $4$ and $2$: $4 > 3 > 2$ with $\card_I(4,3) = 3$, and $4 > 2$ with $\card_I(4,2) = 2$. This gives $\card_{\tc{I}}(4,2) = 3$.
\item There is one transitivity path between $3$ and $1$: $3 > 2 > 1$ with $\card_I(3,2) = 1$. This gives $\card_{\tc{I}}(3,1) = 1$.
\end{itemize}

\end{example}

\begin{lemma}
\label{prop:tc-transitive}
Let $I$ be a multi inversion set, then $\tc{I}$ is \emph{transitive} in the sense of Definition \ref{def:tree-inversion-set}. Moreover, for any transitive multi inversion set $J$ such that $I \subseteq J$, we have $\tc{I} \subseteq J$.
\end{lemma}

\begin{proof}
We first prove that $\tc{I}$ is transitive. Let us suppose that there is $c > b > a$ in $\tc{I}$ such that $\card_{\tc{I}}(c,b) = v$ and $\card_{\tc{I}}(b,a) > 0$. By definition of $\tc{I}$, $\card_{\tc{I}}(c,b) = v$ implies that there is a transitivity path $c = b_1 > \dots > b_k = b$ between $c$ and $b$ with $\card{I}(b_1, b_2) = v$ and $\card_I(b_i, b_{i+1}) > 0$ for all $i$ such that $1 \leq i < k$. Likewise, the fact that $\card_{\tc{I}}(b,a) > 0$ implies that there is a transitivity path $b = b_k > b_{k+1} > \dots > b_{k'} = a$ between $b$ and $a$ with $\card_I(b_i, b_{i+1}) > 0$ for all $i$ such that $k \leq i < k'$. This makes $ c = b_1 >  \dots > b_k > \dots > b_{k'} = a$ a transitivity path between $c$ and $a$ in $I$ and so $\card_{\tc{I}}(c,a) \geq \card_I(b_1,b_2) = v$.  

Now, suppose that $J$ is a transitive multi inversion set such that $I \subseteq J$ and $\tc{I} \not\subseteq J$. So there is $c > a$ such that $\card_J(c,a) < \card_{\tc{I}}(c,a)$. We choose such an occurrence for which the distance $c - a$ is minimal. We have in particular that $\card_{\tc{I}}(c,a) = v > 0$ and so, by definition, there is a transitivity path $c = b_1 > \dots > b_k = a$ in $I$ with $\card_I(b_1, b_2) = v$ and $\card_I(b_i, b_{i+1}) > 0$ for all $1 \leq i < k$. The fact that $I \subseteq J$ implies that $\card_J(b_1, b_2) \geq v$ and $\card_J(b_i, b_{i+1}) > 0$. Now, if $k=2$, we have $\card_J(c,a) \geq v = \card_{\tc{I}}(c,a)$ which contradicts our initial assumption on the inversion $(c,a)$. So $k > 2$ and we look at the transitivity path $c = b_1 > \dots > b_{k-1}$ which is well defined. By construction, $\card_{\tc{I}}(c,b_{k-1}) \geq v$. Also the distance $c - b_{k-1} < c - a$ and so by minimality of $c - a$, we get $\card_J(c,b_{k-1}) \geq \card_{\tc{I}}(c,b_{k-1}) \geq v$ and because $I \subseteq J$, we also have $\card_J(b_{k-1}, a) > 0$. By transitivity of $J$, we then get $\card_J(c,a) \geq v$ which is again a contradiction. In conclusion, there is no such inversion $(c,a)$ and so $\tc{I} \subseteq J$.
\end{proof}

In other words, the notion of transitive closure that we have defined on multi inversion sets is a analogue of the classical transitive closure on binary relations: if $I$ is a multi inversion set, then $\tc{I}$ is the smallest transitive multi inversion set (in terms of inclusion) that contains $I$. 

\begin{remark}
\label{rem:transitive-lattice}
From Remark \ref{rem:union-tree-inversions} and Lemma \ref{prop:tc-transitive}, we get that the set of transitive multi inversion sets actually form a join semi-lattice. If $I$ and $J$ are two transitive multi inversion sets, then $I \join J = \tc{ \left( I \cup J \right) }$ is the smallest multi inversion set in terms of inclusion that is both transitive and includes $I$ and $J$. We now want to prove that the $s$-weak order is a sub-lattice. 
\end{remark}

The following two lemmas show that the planarity condition is stable under union and transitive closure. 

\begin{lemma}
\label{lem:union-planarity}
Let $I,J$ be two multi inversion sets which satisfy the planarity condition of Definition~\ref{def:tree-inversion-set}, then $K=I\cup J$ also satisfies the planarity condition.
\end{lemma}

\begin{proof}
Suppose that we have $a < b < c$ with $\card_K(c,a) = i$, then we have $\max \left( \card_I(c,a), \card_J(c,a) \right) = i$. We suppose without loss of generality that $\card_I(c,a) = i$. Then, either $\card_I(b,a) = s(b)$ or $\card_I(c,b) \geq i$. The cardinality of any inversion in $K$ is greater than or equal to its cardinality in $I$ and so the condition is preserved in $K$.
\end{proof}

\begin{lemma}
\label{prop:tc-planarity}
Let $K$ be a multi inversion set which satisfies the planarity condition of Definition~\ref{def:tree-inversion-set}, then $\tc{K}$ also satisfies the planarity condition. 
\end{lemma}

\begin{proof}
Let $K$ be a multi inversion set which satisfies the planarity condition. Suppose that there is $a < b < c$ such that $\card_{\tc{K}}(c,a) = v > 0$. Then, there is $c = b_1 > b_2 > \dots > b_k = a$, a transitivity path in~$K$ with $\card_K(b_1, b_2) = v$ and $\card_K(b_i, b_{i+1}) > 0$ for all $1 \leq i < k$. If there is $j$ such that $b = b_j$, then, there is a transitivity path in $K$ between $c$ and $b$ and so $\card_{\tc{K}}(c,b) \geq v$. Otherwise, there is $j$ such that $b_j > b > b_{j+1}$ and $\card_K(b_j, b_{j+1}) = v' > 0$. We know that $K$ satisfies the planarity rule. If $\card_K(b_j,b) \geq v'$, then $c = b_1 > \dots > b_j > b$ is a transitivity path in $K$ between $c$ and $b$ and we have $\card_{\tc{K}}(c,b) \geq v$. Otherwise, $\card_K(b,b_{j+1}) = s(b)$, then either $s(b) = 0$ or there is a transitivity path $b > b_{j+1} > \dots > b_k = a$ between $b$ and $a$. In both cases $\card_{\tc{K}}(b,a) = s(b)$. So $\tc{K}$ satisfies the planarity condition.
\end{proof}

\begin{proposition}
\label{prop:tc-tree-inversion}
Given a weak composition $s$, let $I$ and $J$ be two $s$-tree-inversion sets. Then $\tc{(I \cup J)}$ is the smallest $s$-tree-inversion set containing $I$ and $J$.
\end{proposition}

\begin{proof}
First, neither the union nor the transitive closure can increase the cardinality of a tree-inversion $(y,x)$ up to more than $s(y)$ and so $\tc{(I \cup J)}$ is still included in or equal to the maximal element $\maxs_s$ defined in Definition \ref{def:multi-inversion-set}. We set $K = I \cup J$ and we need to prove that $\tc{K}$ is the smallest $s$-tree-inversion set containing $K$.
By Lemma \ref{prop:tc-transitive}, we know that $\tc{K}$ is the smallest transitive multi inversion set containing~$K$. Lemmas~\ref{lem:union-planarity} and~\ref{prop:tc-planarity} imply that $\tc{K}$ satisfies the planarity condition.
\end{proof}

This leads us to the main result of this Section.

\begin{theorem}
\label{thm:lattice}
Given a weak composition $s$, the $s$-weak order on $s$-decreasing trees is a lattice. The join of two $s$-decreasing trees $T$ and $R$ is given by
\begin{equation}
\inv(T \join R) = \tc{ \left( \inv(T) \cup \inv(R) \right) }.
\end{equation}
\end{theorem}

\begin{proof}
From Remark~\ref{rem:transitive-lattice} and Proposition~\ref{prop:tc-tree-inversion}, we get that the $s$-weak order is a join-sublattice of the join semilattice of transitive multi inversion sets. By Proposition~\ref{prop:symmetry}, it is also symmetric and so it is meet-semilattice, hence a lattice.
\end{proof}

\begin{example}
Figure~\ref{fig:join} gives an example of computing the join of two $s$-decreasing trees with $s = (0,2,2)$. In this case, if $K = \inv(T) \cup \inv(R)$, we obtain 

\begin{tabular}{ll}
$\card_K(3,2) = 2$ & $\card_K(3,1) = 0$ \\
                   & $\card_K(2,1) = 1$ 
\end{tabular}

by taking the maximum of the cardinalities in $T$ and $R$. The multiset $K$ is not transitive. Indeed, we have $3 > 2 > 1$ with $\card_K(3,2) = 2$ and $\card_K(2,1) = 1$ but $\card_K(3,1) = 0 < 2$. By taking the transitive closure, we get $\card_{T \join R}(3,1) = 2$. You can check on Figure~\ref{fig:sweak_lattice} that the tree $T \join R$ is indeed the join of $T$ and $R$ in the lattice.

\end{example}

\begin{figure}[ht]
\begin{tabular}{ccc}
$T$ & $\begin{aligned}\scalebox{.8}{{ \newcommand{\nodea}{\node[draw,circle] (a) {$3$}
;}\newcommand{\nodeb}{\node[draw,circle] (b) {$1$}
;}\newcommand{\nodec}{\node[draw,circle] (c) {$ $}
;}\newcommand{\noded}{\node[draw,circle] (d) {$ $}
;}\newcommand{\nodee}{\node[draw,circle] (e) {$2$}
;}\newcommand{\nodef}{\node[draw,circle] (f) {$ $}
;}\newcommand{\nodeg}{\node[draw,circle] (g) {$ $}
;}\newcommand{\nodeh}{\node[draw,circle] (h) {$ $}
;}\begin{tikzpicture}[auto]
\matrix[column sep=.3cm, row sep=.3cm,ampersand replacement=\&]{
         \& \nodea  \&         \&         \&         \\ 
 \nodeb  \& \noded  \&         \& \nodee  \&         \\ 
 \nodec  \&         \& \nodef  \& \nodeg  \& \nodeh  \\
};

\path[ultra thick, red] (b) edge (c)
	(e) edge (f) edge (g) edge (h)
	(a) edge (b) edge (d) edge (e);
\end{tikzpicture}}}\end{aligned}$ &
\begin{tabular}{ll}
$\card_T(3,2) = 2$ & $\card_T(3,1) = 0$ \\
                   & $\card_T(2,1) = 0$ 
\end{tabular} \\ \hline
$R$ & $\begin{aligned}\scalebox{.8}{{ \newcommand{\nodea}{\node[draw,circle] (a) {$3$}
;}\newcommand{\nodeb}{\node[draw,circle] (b) {$2$}
;}\newcommand{\nodec}{\node[draw,circle] (c) {$ $}
;}\newcommand{\noded}{\node[draw,circle] (d) {$1$}
;}\newcommand{\nodee}{\node[draw,circle] (e) {$ $}
;}\newcommand{\nodef}{\node[draw,circle] (f) {$ $}
;}\newcommand{\nodeg}{\node[draw,circle] (g) {$ $}
;}\newcommand{\nodeh}{\node[draw,circle] (h) {$ $}
;}\begin{tikzpicture}[auto]
\matrix[column sep=.3cm, row sep=.3cm,ampersand replacement=\&]{
         \&         \&         \& \nodea  \&         \\ 
         \& \nodeb  \&         \& \nodeg  \& \nodeh  \\ 
 \nodec  \& \noded  \& \nodef  \&         \&         \\ 
         \& \nodee  \&         \&         \&         \\
};

\path[ultra thick, red] (d) edge (e)
	(b) edge (c) edge (d) edge (f)
	(a) edge (b) edge (g) edge (h);
\end{tikzpicture}}}\end{aligned}$ &
\begin{tabular}{ll}
$\card_R(3,2) = 0$ & $\card_R(3,1) = 0$ \\
                   & $\card_R(2,1) = 1$ 
\end{tabular} \\ \hline
$T \join R$ & $\begin{aligned}\scalebox{.8}{{ \newcommand{\nodea}{\node[draw,circle] (a) {$3$}
;}\newcommand{\nodeb}{\node[draw,circle] (b) {$ $}
;}\newcommand{\nodec}{\node[draw,circle] (c) {$ $}
;}\newcommand{\noded}{\node[draw,circle] (d) {$2$}
;}\newcommand{\nodee}{\node[draw,circle] (e) {$ $}
;}\newcommand{\nodef}{\node[draw,circle] (f) {$1$}
;}\newcommand{\nodeg}{\node[draw,circle] (g) {$ $}
;}\newcommand{\nodeh}{\node[draw,circle] (h) {$ $}
;}\begin{tikzpicture}[auto]
\matrix[column sep=.3cm, row sep=.3cm,ampersand replacement=\&]{
         \& \nodea  \&         \&         \&         \\ 
 \nodeb  \& \nodec  \&         \& \noded  \&         \\ 
         \&         \& \nodee  \& \nodef  \& \nodeh  \\ 
         \&         \&         \& \nodeg  \&         \\
};

\path[ultra thick, red] (f) edge (g)
	(d) edge (e) edge (f) edge (h)
	(a) edge (b) edge (c) edge (d);
\end{tikzpicture}}}\end{aligned}$ &
\begin{tabular}{ll}
$\card_{T \join R}(3,2) = 2$ & $\card_{T \join R}(3,1) = 2$ \\
                   & $\card_{T \join R}(2,1) = 1$ 
\end{tabular} \\ 
\end{tabular}
\caption{The computation of $T \join R$ in the $s$-weak lattice}
\label{fig:join}
\end{figure}
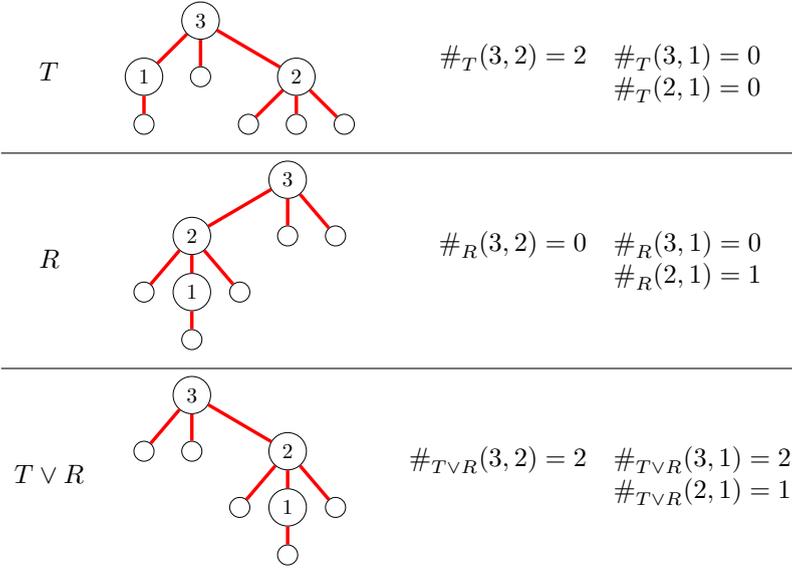

\begin{remark}
\label{rem:meet}
We could also prove directly that the $s$-weak order is a meet semilattice. By symmetry, we just define the \defn{tree-versions} of an $s$-decreasing tree $T$ by $\card_T(x,y) = s(y) - \card_T(y,x)$ for all $x < y$ in~$T$. This gives a multi set of \emph{versions} which follows some very similar rules as the multi set of \emph{inversions}. Then the definition and proofs used for the join semi-lattice also apply and we can define the meet of two $s$-decreasing trees by making the transitive closure of the union of their tree-version sets.

This is also equivalent to taking the mirror image $\tilde{T}$ of each tree $T$, where $\card_{\tilde{T}}(b,a) = s(b) - \card_T(b,a)$. Then $T \meet R$ is the mirror image of $\tilde{T} \join \tilde{R}$
\end{remark}

\subsection{Cover relations}
The goal of this section is to describe the cover relations of the $s$-weak-order. We show that they correspond to a certain class of rotations on trees that we call \emph{$s$-tree rotations}. Such rotations can be performed along \emph{tree-ascents}, a notion that we now introduce in terms of trees and inversion sets, and which generalizes the classical notion of ascents of a permutation.

\begin{definition}
\label{def:tree-ascent}
Let $T$ be an $s$-decreasing tree of some weak composition $s$ of length $n$. We say that~$(a,c)$ with $a < c$ is a \defn{tree-ascent} of $T$ if
\begin{enumerate}[(i)]
\item $a$ is a descendant of $c$;
\label{cond:tree-ascent-desc}
\item $a$ does not belong to the right child of $c$;
\label{cond:tree-ascent-non-final}
\item if $a$ is a descendant of $b$ with $c > b > a$, then $a$ belongs to the right descendant of $b$;
\label{cond:tree-ascent-middle}
\item if it exists, the strict right child of $a$ is empty. 
\label{cond:tree-ascent-smaller}
\end{enumerate}
\end{definition}

The above definition can be visually interpreted from the tree: 
basically, $a$ is located at the \emph{right end} of a non right subtree of $c$, with the exception $s(a)=0$, when $a$ is allowed to have more children on its unique leg. 
Figure~\ref{fig:s_tree_ascents} shows a schematic illustration of these two possible occurrences of an $(a,c)$ ascent. 

\begin{figure}[htbp]
\begin{center}
\includegraphics[width=0.5\textwidth]{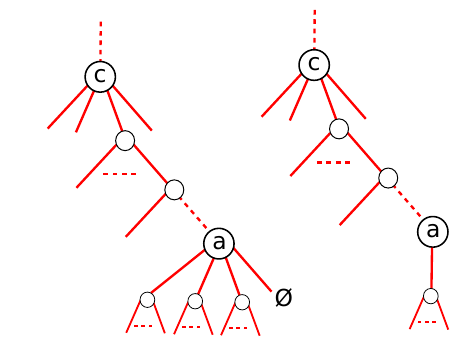}
\caption{Schematic illustration of two $s$-trees containing an ascent $(a,c)$.}
\label{fig:s_tree_ascents}
\end{center}
\end{figure}

\begin{remark}
\label{rem:tree-ascent-a}
For each value $a$, there is at most one tree-ascent $(a,c)$. In other words, the smaller value $a$ of a tree-ascent $(a,c)$ entirely characterizes the tree-ascent. More precisely, let $a$ be a node such that its strict right child is empty. Now let $c$ be the ascendant of $a$ of minimal value such that $\card(c,a) < s(c)$, if it exists. It is direct by Definition~\ref{def:tree-ascent} that $(a,c)$ is then the only possible tree-ascent $(a,*)$.
\end{remark}

\begin{example}
The tree-ascents of the tree on Figure~\ref{fig:ex-dtree} are: $(4,5)$, $(2,5)$, and $(1,5)$. Note that $(3,5)$ is not a tree-ascent because the right-most child of $3$ is not empty. On the other hand, $(2,5)$ is a tree ascent because even though its unique child is not empty, it is not a strict right child as $s(2) = 0$.
\end{example}

\begin{proposition}
\label{prop:tree-ascent-inversions}
A couple $(a,c)$ is a tree-ascent if and only it satisfies the following statements
\begin{enumerate}[(i)]
\item for all $d$ such that $c < d \leq n$, we have $\card(d,c) = \card(d,a)$;
\label{cond:inv-tree-ascent-desc}
\item $\card(c,a) < s(c)$;
\label{cond:inv-tree-ascent-non-final}
\item for all $b$ such that  $a < b < c$, then $\card(c,b) = \card(c,a)$ implies that $\card(b,a) = s(b)$;
\label{cond:inv-tree-ascent-middle}
\item if $s(a) > 0$, for all $a' < a$, then $\card(a,a') = s(a)$ implies that $\card(c,a') > \card(c,a)$.
\label{cond:inv-tree-ascent-smaller}
\end{enumerate}
\end{proposition}
\begin{proof}
Each condition of the proposition is a translation of the corresponding condition in Definition~\ref{def:tree-ascent}. First, suppose that $(a,c)$ is a tree-ascent in the sense of Definition~\ref{def:tree-ascent} and let us prove that the conditions of the proposition are satisfied.
\begin{enumerate}[(i)]
\item $a$ is a descendant of $c$, so in particular, for any $d > c$, we have $a$ and $c$ in the same subtree which implies the condition.
\item $a$ is not weakly right of $T_{s(c)}^c$: it does not belong to $T_{s(c)}^c$ and cannot be right of $c$ because it is a descendant of $c$. Then by definition of tree-inversions, $\card(c,a) < s(c)$.
\item Let us suppose that there is $a < b < c$ with $\card(c,b) = \card(c,a)$. If $b$ is left of $a$, then the condition is satisfied. If $a$ is a descendant of $b$, it belongs to its right child and again $\card(b,a) = s(b)$. The only case left is when $b$ is right of $a$. This means that $a$ and $b$ have a common ancestor $c'$ with $\card(c',a) < \card(c',b)$. In particular, $c' \neq c$. If $c' > c > b$, by planarity we have either $\card(c',c) \geq \card(c',b)$ or $\card(c,b) = s(c)$. Because $a$ is descendant of $c$, $\card(c',a) = \card(c',c)$ which contradicts the first case. The second case also leads to a contradiction because we have by hypothesis that $\card(c,b) = \card(c,a)$ and we have seen already $\card(c,a) < s(c)$. So  $b$ cannot be right of $a$ and the condition is satisfied.
\item We suppose $s(a) > 0$ and take $a' < a$ such that $\card(a,a') = s(a)$. This means that $a'$ is either in the right child of $a$ or right of $a$. We have $s(a) > 0$ so the right child of $a$ is a \emph{strict right child} and by hypothesis, it is empty. So there exists $c' > a$ with $\card(c',a) < \card(c',a')$. If $c' = c$, the condition is satisfied. If $c' > c$, we use the planarity condition on $c' > c > a'$. Because $a$ is a descendant of $c$, we know $\card(c',a) = \card(c',c)$ which forbids $\card(c',c) \geq \card(c',a)$. Then  $\card(c,a') = s(c)$ which satisfy the condition. If $c' < c$, then we have $a$ a descendant of $c'$ with $c > c' > a$ which implies that $a$ belongs to the right child of $c'$ and contradicts our previous statement. 
\end{enumerate}

To prove the reciprocity, we now suppose that $(a,c)$ satisfies all conditions of the proposition and we prove that it is indeed a tree-ascent.
\begin{enumerate}[(i)]
\item Condition~\ref{cond:tree-ascent-desc} of the definition implies that $a$ cannot be either left or right of $c$: it is then a descendant of $c$.
\item is true by definition of tree-inversions.
\item Suppose that we have $a$ a descendant of $b$ with $a < b < c$. Then $b$ cannot be either left or right of $c$, it is then a descendant of $c$ and moreover, $a$ and $b$ are in the same subtree of $c$. This implies $\card(c,a) = \card(c,b)$. Condition~\ref{cond:tree-ascent-middle} then tells us $\card(b,a) = s(b)$ which forces $a$ to be in the right subtree of $b$.
\item Let us suppose that there exists $a'$ in the strict right child of $a$. This means that $s(a) > 0$ (otherwise, it does not have a strict right child) and $\card(a,a') = s(a)$. Condition~\ref{cond:tree-ascent-smaller} then implies that $\card(c,a') > \card(c,a)$ which forbids $a'$ to be a descendant of $a$ (it is not in the same subtree of $c$) and leads to a contradiction.
\end{enumerate}
\end{proof}

\begin{definition}
Let $I$ be a multi inversion set. We say that we \defn{add} the inversion $(y,x)$ to $I$ and we write $I + (y,x)$ if we increase $\card(y,x)$ by one.
\end{definition}

\begin{lemma}
\label{lem_tree_rotation_inversions}
Let $T$ be an $s$-decreasing tree for some weak composition $s$, and $(a,c)$ a tree-ascent of~$T$. Then $\tc{\left( \inv(T) + (c,a) \right)}$ is a tree-inversion set. The inversions of its corresponding $s$-tree $T'$ are
\begin{equation}
\label{eq_tree_rotation_inversions}
\card_{T'}(c',a') =
\left\{
	\begin{array}{ll}
		\card_{T}(c',a')+1  & \mbox{if } c'=c \text{ and } a' \text{ is equal to } a \text{ or a non-left descendant of } a,\\
		\card_{T}(c',a') & \mbox{otherwise.} 
	\end{array}
\right.
\end{equation}
\end{lemma}

\begin{proof}
Let $(a,c)$ be a tree-ascent of $T$ and $K  = \inv(T) + (c,a)$. We will prove the result in two parts.

\begin{proofpart}[$\tc{K}$ is a tree-inversion set]
By Lemma~\ref{prop:tc-transitive} and Lemma~\ref{prop:tc-planarity}, it suffices to show that $K$ satisfies the planarity condition. 

Let $x < y < z$ be such that $\card_K(z,x) = i$. If $(z,x) \neq (c,a)$, then $\card_T(z,x) = \card_K(z,x)$ and by the planarity of $T$, we have $\card_T(z,y) \geq i$ or $\card_T(y,x) = s(y)$. Whichever condition is satisfied it would also be true in $K$ because the cardinality of an inversion can only increase. 

The remaining case to consider is then $a < b < c$ and $\card_K(c,a) = i$. By definition of $K$, we have $\card_T(c,a) = i-1$. By planarity, we have  $\card_T(c,b) \geq i-1$ or $\card_T(b,a) = s(b)$. If $\card_T(b,a) = s(b)$, we also have $\card_K(b,a) = s(b)$ and the planarity condition is satisfied. If $\card_T(c,b) > i-1$, then we have $\card_K(c,b) = \card_T(c,b) \geq i$ and the planarity condition is satisfied. There is left to check the case where $\card_T(c,b) = i - 1$. But then we have $\card_T(c,b) = \card_T(c,a)$ and by Condition~\eqref{cond:tree-ascent-middle} of the tree-ascent $(a,c)$ we have $\card_T(b,a) = s(b)$ and so also $\card_K(b,a) = s(b)$ which then satisfies the planarity condition.
\end{proofpart}

\begin{proofpart}[Equation~\eqref{eq_tree_rotation_inversions} defines $\tc{K}$, the smallest transitive multi inversion set containing $K$]
We start by proving that the multi inversion set defined by Equation~\eqref{eq_tree_rotation_inversions} is transitive.

Let $x<y<z$ and assume that $\card_{T'}(z,y)=i$. We need to show that 
\[
\card_{T'}(y,x)=0 \quad \text{or} \quad \card_{T'}(z,x)\geq i.
\]

\begin{proofcase}
{\bf $\mathbf{(z,y)=(c,a')}$ for some $a' = a$ or $a'$ a non-left descendant of $a$ in $T$.}
In this case $\card_{T'}(z,y)= \card_{T}(z,y)+1=i$, which implies $\card_{T}(y,x)=0$ or $\card_{T}(z,x)\geq i-1$ by transitivity of $T$.
If $\card_{T}(y,x)=0$ then $\card_{T'}(y,x)=0$ because $y\neq c$ and we are done. 
If $\card_{T}(z,x)> i-1$ then $\card_{T'}(z,x)\geq i$ and we are also done.
Now assume $\card_{T}(z,x)= i-1$ and $\card_{T}(y,x)>0$. If $y = a' = a$, this means $\card_T{(a,x)} >0$. If $y = a' \neq a$, by transitivity of $T$ for $x<y<a$, and given that $a'$ is a non-left descendant of $a$ in $T$, we get the following two inequalities:
\[
\card_T(a,x) \geq \card_T(a,y) = \card_T(a,a') > 0.
\]
So whether $a' = a$ or not, we obtain $\card_T(a,x)>0$. From this, we prove that $x$ is necessarily a non-left descendant of~$a$. Indeed, $x$ cannot be right of~$a$. Remember that $(a,c)$ is a tree-ascent which by Cond.~\ref{cond:tree-ascent-middle} of Def.~\ref{def:tree-ascent} means it belongs to the right descendant of any node $b$ between $a$ and $c$. Then $x$ being right of $a$ means $\card_T(z,x)=\card_T(c,x)>\card_T(c,a)=i-1$ which contradicts our assumption.
Therefore $\card_{T'}(z,x)=\card_{T}(z,x)+1=i-1+1=i\geq i$ as desired.
\end{proofcase}

\begin{proofcase}
{\bf $\mathbf{(z,y) \neq (c,a)}$ and $\mathbf{(z,y) \neq (c,a')}$ for any non-left descendant $a'$ of $a$ in $T$.}
In this case \mbox{$\card_{T'}(z,y)= \card_{T}(z,y)=i$}, which implies that $\card_{T}(y,x)=0$ or $\card_{T}(z,x)\geq i$ by transitivity of $T$.
If $\card_{T}(z,x)\geq i$ then $\card_{T'}(z,x)\geq i$ and we are done. If $\card_T(y,x)=0$ and $\card_{T'}(y,x)=0$ we are also done. Now assume $\card_T(y,x)=0$ and $\card_{T'}(y,x)>0$. This happens exactly when $(y,x)=(c,a')$ for some $a' = a$ or $a'$ a non-left descendant of $a$ in $T$. Since $z>y$ and $x$ is a descendant of $y$ then $\card_T(z,x)=\card_T(z,y)=i$, and therefore $\card_{T'}(z,x)\geq i$ holds.
\end{proofcase}

This finishes our proof that the multi inversion set defined by Equation~\eqref{eq_tree_rotation_inversions} is transitive. 
It clearly contains $K$, and is the smallest transitive multi inversion set containing it because of the transitivity condition. We conclude that it is equal to $\tc{K}$ by Lemma~\ref{prop:tc-transitive}. 
\end{proofpart}
The Lemma follows from the assertions in Part 1 and Part 2.
\end{proof}

\begin{definition}
If $(a,c)$ is a tree-ascent of an $s$-decreasing tree $T$, we call the $s$-decreasing tree corresponding to $\tc{\left( \inv(T) + (c,a) \right)}$ an \defn{$s$-tree rotation} of $T$.
\end{definition}

an $s$-tree rotation can be seen directly on the tree: it corresponds to moving a node to the right in a way that we make more precise now. When rotating the tree-ascent $(a,c)$, $a$ moves to the right along $c$ (indeed, we increase $\card(c,a)$). The node $a$ is then inserted on the next subtree of $c$, along the left-most path: its exact position depends on the relative value of $a$ compared with the other nodes on the path. A node can move only if its strict right child is empty (following the definition of tree-ascents) and it moves taking along all its middle children. Indeed, if $a'$ is a middle descendant of $a$, then $\card(a,a') > 0$ and so the cardinality of $(c,a')$ in $\tc{\left( \inv(T) + (c,a) \right)}$ is increased by transitivity. Note however that if $a'$ is a left descendant of $a$, it does not move along with $a$. Indeed, in this case $\card(a,a') = 0$ and so it is not increased by transitivity when $\card(c,a)$ increases. 

\begin{example}
We show all $6$ possible $s$-tree rotations of a tree on Figure~\ref{fig:tree_rotation}. To each tree-ascent $(a,c)$, corresponds an $s$-tree rotation. We color the node $a$ of the tree-ascent in red to show which tree-rotation we are performing. 

When rotating the tree-ascent $(5,8)$, the node $5$ is moved to the second subtree of $8$. Before the rotation, the left-most path of this second subtree was $7 - 6 -3$. As $5$ is smaller than $7$ and $6$ and we necessarily have $\card(7,5) = \card(6,5) = 0$, it becomes a left descendant of $7$ and $6$. As $5 > 3$, it becomes a parent of $3$. We necessarily have $\card(5,3) = s(5) = 1$ and so $3$ becomes the right child of $5$. This can be observed on all other rotations. Note in particular that after rotating a tree ascent $(a,c)$, the strict left child of $a$ is always empty. 

When rotating the tree-ascent $(7,8)$, you can observe that the nodes $1$, $2$, and $4$, which are in the middle subtree of $7$ are moved along. Indeed, we have $\card(7,4) = \card(7,2) = \card(7,1) = 1$. After the rotation, we have $\card(8,7) = 2$ and so we obtain by transitivity $\card(8,4) = \card(8,2) = \card(8,1) = 1$.
\end{example}

\begin{figure}[htbp]
\input{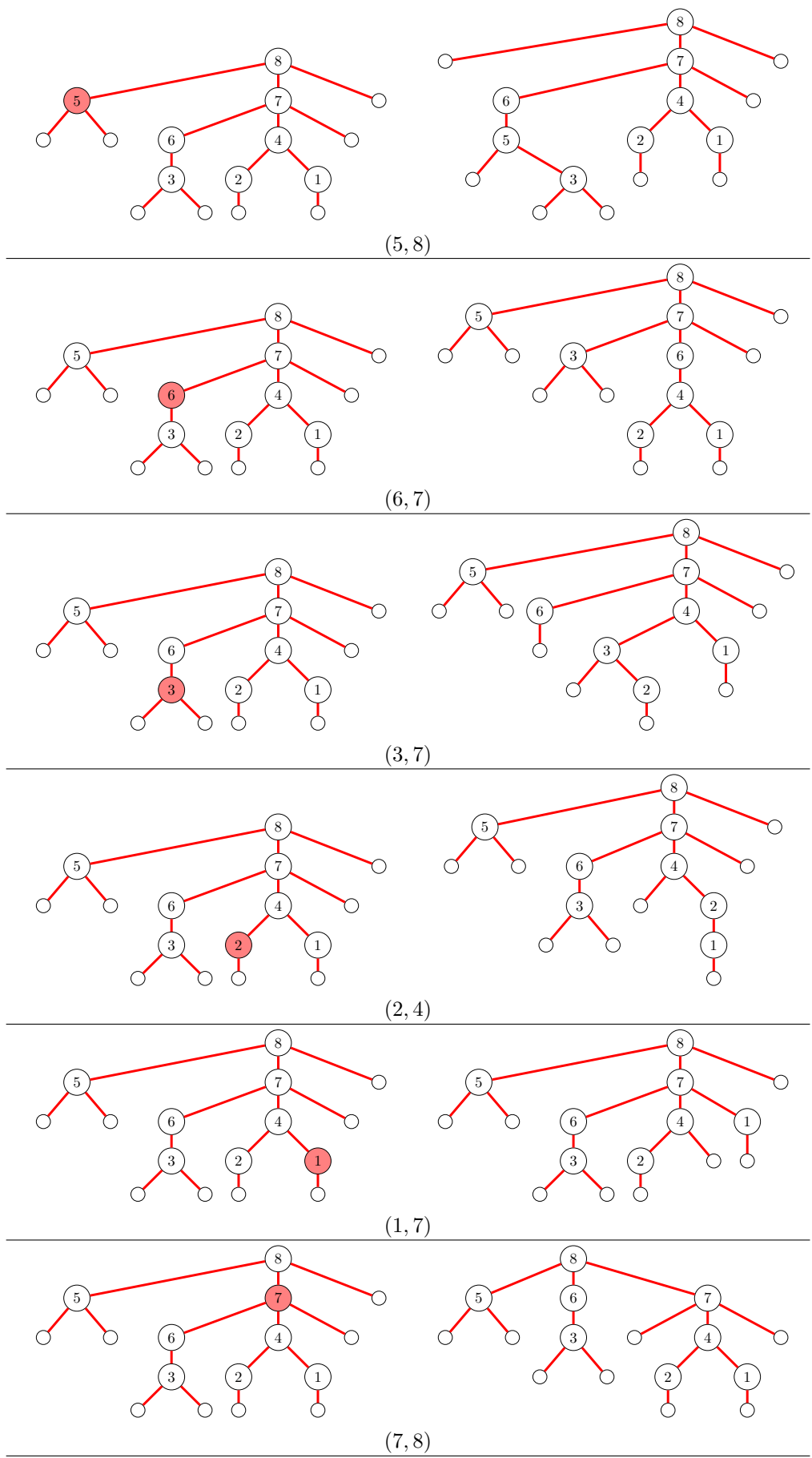}
\caption{Examples of $s$-tree rotations}
\label{fig:tree_rotation}
\end{figure}

The main result of this section is given by the following theorem.

\begin{theorem}
\label{thm:cover-relations}
The cover relations of the $s$-weak order are in correspondence with $s$-tree rotations. 
\end{theorem}

We need to show that an $s$-decreasing tree $T'$ covers $T$ in the $s$-weak order if and only if there is a tree-ascent~$(a,c)$ of $T$ such that $\inv(T') = \tc{\left( \inv(T) + (c,a) \right)}$.
The proof will follow from the following lemma.

\begin{lemma}
\label{lem:cover-relations-recip}
Let $T$ and $T'$ be two $s$-decreasing trees for some weak composition $s$ with $T \woless T'$. Then there is a tree-ascent $(a,c)$ of $T$ such that $\card_{T'}(c,a) > \card_T(c,a)$.
\end{lemma}

\begin{proof}
Let $V = \lbrace (c,a); \card_T(c,a) < \card_{T'}(c,a) \rbrace$, \emph{i.e.}, the tree-inversions of $T$ which cardinality have increased in $T'$. As $T \woless T'$, we know that $V \neq \emptyset$. Let us call $I$, $II$, $III$, and $IV$ the multi sets of inversions which satisfy respectively conditions \eqref{cond:tree-ascent-desc} to \eqref{cond:tree-ascent-smaller}  of Definition~\ref{def:tree-ascent} in $T$. Our goal is to prove that $V \cap I \cap II \cap III \cap IV \neq \emptyset$. First note that $V \subseteq II$: indeed, if the cardinality of $(c,a)$ increases, it can not already be maximal. 

\begin{proofpart}[$V \cap I \neq \emptyset$]
\label{part1_lem:cover-relations-recip}
Let $(c,a) \in V $ such that $c = \max \lbrace j; (j,i) \in V \rbrace$, we claim that $(c,a) \in I$. Suppose that there is $d > c$ such that $\card_T(d,c) \neq \card_T(d,a)$. If $\card_T(d,c) < \card_T(d,a)$, then the planarity condition on $T$ forces $\card_T(c,a) = s(c)$ which contradicts $(c,a) \in V$. We then have $\card_T(d,c) > \card_T(d,a)$. Because $d >c$, the maximality of $c$ implies that those cardinalities are unchanged in $T'$ and so $\card_{T'}(d,c) > \card_{T'}(d,a)$. In particular $\card_{T'}(d,c) > 0$. We also have $\card_{T'}(c,a) > 0$ because $(c,a) \in V$. By transitivity of $T'$, we obtain $\card_{T'}(d,a) \geq \card_{T'}(d,c)$ which contradicts our previous statement.
\end{proofpart}

\begin{proofpart}[$V \cap I \cap III \neq \emptyset$]
Let $(c,a) \in V \cap I$ such that $c-a$ is minimal. We claim that $(c,a) \in III$. Suppose there is $a < b < c$, with $\card_T(c,b) = \card_T(c,a)$ and $\card_T(b,a) < s(b)$. We choose the maximal $b$ that complies. As $(c,a) \in V$, we know that its cardinality increases in $T'$ and, in particular, $\card_{T'}(c,a) = v > 0$. By planarity, we then have either $\card_{T'}(c,b) \geq v$ or $\card_{T'}(b,a) = s(b)$. 

In the case where $\card_{T'}(c,b) \geq v$, we have $\card_T(c,b) = \card_T(c,a) < \card_{T'}(c,a) \leq \card_{T'}(c,b)$ which means $(c,b) \in V$. As $c - b < c -a$, by minimality of $(c,a)$, we have $(c,b) \notin V \cap I$ and so $(c,b) \notin I$. So there is $d > c$ with $\card_T(d,c) \neq \card_T(d,b)$. If $\card_T(d,c) < \card_T(d,b)$, by planarity we have $\card_T(c,b) = s(c)$ which contradicts $(c,b) \in V$. So $\card_T(d,c) > \card_T(d,b)$. But then we have $\card_T(d,a) = \card_T(d,c)$ because $(c,a) \in I$ and we can again apply planarity on $c > b > a$ which leads to the contradiction $\card_T(b,a) = s(b)$.

In the case where $\card_{T'}(b,a) = s(b)$. We have $\card_T(b,a) < \card_{T'}(b,a)$ and so $(b,a) \in V$. By minimality of $c-a$, we then have $(b,a) \notin I$. So there is $d > b$ with $\card_T(d,b) \neq \card_T(d,a)$. If $\card_T(d,b) < \card_T(d,a)$, we obtain the contradiction $\card_T(b,a) = s(b)$ by planarity. Then $\card_T(d,b) > \card_T(d,a)$. As $\card_T(c,b) = \card_T(c,a)$, we cannot have $d=c$. We cannot have $d > c$, otherwise $\card_T(d,a) = \card_T(d,c)$ and we can apply planarity on $d > c > b$ which gives $\card_T(c,b) = s(c) = \card_T(c,a)$ and contradicts $(c,a) \in V$. So $c > d> b > a$. As $\card_T(d,b) > 0$, we have by transitivity $\card_T(c,b) \geq \card_T(c,d)$ and so by our initial hypothesis $\card_T(c,a) \geq \card_T(c,d)$. If $\card_T(c,a) = \card_T(c,d)$, the maximality of $b$ forces $\card_T(d,a) = s(d)$. Similarly, if $\card_T(c,a) > \card_T(c,d)$ we obtain $\card_T(d,a) = s(d)$ by planarity on $c > d > a$. In both cases, this contradicts $\card_T(d,b) > \card_T(d,a)$.
\end{proofpart}

\begin{proofpart}[$V \cap I \cap III \cap IV \neq \emptyset$] Let $(c,a) \in V \cap I \cap III$ such that $c-a$ is maximal. We claim that $(c,a) \in IV$. By contradiction, let us suppose $s(a) > 0$ and that there exists $a' < a$ such that $\card_T(a,a') = s(a)$ and $\card_T(c,a') \leq \card_T(c,a)$. We choose the maximal value $a'$ that complies.

Because $(c,a) \in V$, we have $\card_{T'}(c,a) >  \card_T(c,a) \geq 0$. We also have $\card_{T'}(a,a') \geq \card_{T}(a,a') = s(a) > 0$. By transitivity of $T'$, we get $\card_{T'}(c,a') \geq \card_{T'}(c,a) > \card_T(c,a)$. So $(c,a') \in V$. By maximality of $c-a$, $(c,a') \notin V \cap I \cap III$, so either $(c,a') \notin I$ or $(c,a') \notin III$.
\begin{itemize}
\item If $(c,a') \notin I$, then there is $d > c$ such that $\card_T(d,c) \neq \card_T(d,a')$. Remember that $(c,a) \in I$ and so $\card_T(d,c) = \card_T(d,a)$. By transitivity of $T$, we then have $\card_T(d,a') \geq \card_T(d,a) = \card_T(d,c)$ and so by our hypothesis,  $\card_T(d,c) < \card_T(d,a')$. By planarity, we get that $\card_T(c,a') = s(c)$ and so $\card_T(c,a) \geq \card_T(c,a')$ is also equal to $s(c)$ which contradicts $(c,a) \in V$.
\item If $(c,a') \notin III$, there is $b$ with $a' < b < c$ and $\card_T(c,b) = \card_T(c,a')$ and $\card_T(b,a') < s(b)$. If $b < a$, by planarity on $a > b > a'$, we have $\card_T(a,b) \geq \card_T(a,a') = s(a)$. And we also have by hypothesis $\card_T(c,a) \geq \card_T(c,a') = \card_T(c,b)$. So $b$ satisfies our initial condition on $a'$ but $b > a'$ which contradicts the maximality of $a'$. If $b=a$, this contradicts our hypothesis $\card_T(a,a') = s(a)$. So $b > a$. We have $\card_T(c,b) = \card_T(c,a') \leq \card_T(c,a)$. By planarity on $c > b > a$, we also have $\card_T(c,b) \geq
 \card_T(c,a)$. Indeed $\card_T(b,a) = s(b)$ would imply the contradiction $\card_T(b,a') \geq s(b)$ by transitivity on $b > a > a'$. So $\card_T(c,b) = \card_T(c,a)$. But $(c,a) \in III$ by hypothesis which then implies $\card_T(b,a) = s(b)$ and then again by transitivity $\card_T(b,a') \geq s(b)$. 
\end{itemize}
This proves that $(c,a') \in I \cap III$ which contradicts the maximality of $c-a$ and proves our claim.
\end{proofpart}
The result follows from Parts 1,2 and 3 and the fact that $V \subseteq II$.
\end{proof}

\begin{proof}[Proof of Theorem \ref{thm:cover-relations}]
We will show that $T'$ covers $T$ in the $s$-weak order if and only if there is a tree-ascent~$(a,c)$ of $T$ such that $\inv(T') = \tc{\left( \inv(T) + (c,a) \right)}$.

Let $T$ and $T'$ be two $s$-decreasing trees such that $T'$ covers $T$. Then, $T \woless T'$ and by Lemma~\ref{lem:cover-relations-recip}, there is a tree-ascent $(a,c)$ in $T$ such that $\card_{T'}(c,a) > \card_T(c,a)$. In particular, $\inv(T) + (c,a) \subseteq \inv(T')$ and by transitivity $\tc{(\inv(T) + (c,a))} \subseteq \inv(T')$. By Lemma~\ref{lem_tree_rotation_inversions}, we know that $\tc{(\inv(T) + (c,a))}$ is a tree-inversion set, let $T''$ be its corresponding $s$-decreasing tree. We have $T \woless T'' \wole T'$. By hypothesis,~$T'$ covers $T$ and so $T' = T''$. Therefore $\inv(T') = \inv(T'') = \tc{\left( \inv(T) + (c,a) \right)}$ as desired.

To prove the other direction, let $(a,c)$ be a tree-ascent of an $s$-decreasing tree $T$, and let $T'$ be the rotated $s$-decreasing tree determined by $\inv(T') = \tc{(\inv(T) + (c,a))}$. 
We need to prove that $T'$ covers~$T$ in the $s$-weak order. 
Suppose that there is an $s$-decreasing tree $T''$ such that $T \woless T'' \wole T'$. By Lemma~\ref{lem:cover-relations-recip}, there is a tree-ascent $(a',c')$ of $T$ such that $\card_{T}(c',a') < \card_{T''}(c',a') \leq \card_{T'}(c',a')$.  
From Lemma~\ref{lem_tree_rotation_inversions} one can deduce that $c'=c$ and $a'=a$.
This implies that $\inv(T) + (c,a) \subseteq \inv(T'')$.  
By transitivity, it follows that $\inv(T') \subseteq \inv(T'')$. Therefore $T' = T''$, which implies that $T'$ covers $T$.
\end{proof}

\subsection{The $s$-weak order is polygonal, semidistributive and congruence uniform}
\label{subsec:cong-uniform}

The main purpose of this section is to present several lattice properties of the $s$-weak order. In particular, we will show that the $s$-weak order is congruence uniform, an important notion in lattice theory. 

The case $s=(1,\dots ,1)$ of our result recovers a result of Caspard in~\cite{caspard_latticepermutations_2000}, who  showed that the lattice of permutations (weak order of type $A$) is congruence uniform. 
The generalization of Caspard's result for the weak order of finite Coxeter groups was proven by Caspard, Le Conte de Poly-Barbut and Morvan in~\cite{CasparPBM2004}. 
The proof of congruence uniformity in~\cite{CasparPBM2004} is based on a more general notion of ``$\HH$ lattices'', which we will also prove here for the $s$-weak order.  
As a consequence we obtain that the $s$-weak order is a polygonal, semidistributive and congruence uniform lattice (Theorem~\ref{thm:hh}).

Before defining $\HH$ lattices, we need to introduce several notions. The first is a notion of polygonality.\footnote{As stated here, this notion is taken from~\cite{Reading2016}, but it is equivalent to conditions (1) and (2) in \cite[Definition~10]{CasparPBM2004}. These conditions regard what they call the ``hat" $(x_1,y,x_2)^\meet$ and ``antihat" $(y_1,x,y_2)^\join$. The $\HH$ terminology stands for {\bf Hat} and anti-{\bf H}at.}

\begin{definition}[{\cite[Definition 9-6.1]{Reading2016}}]
\label{def:polygonal}
A lattice $L$ is said to be \defn{polygonal} if the following two conditions hold:
\begin{enumerate}
\item If distinct elements $y_1$ and $y_2$ both cover an element $x$, then the interval $[x, y_1 \join y_2]$ is a polygon.
\label{cond:polygon-up}
\item If an element $y$ covers two distinct elements $x_1$ and $x_2$, then the interval $[x_1 \meet x_2 , y]$ is a polygon.
\label{cond:polygon-down}
\end{enumerate}
Where a \defn{polygon} is an interval $[x,y]$ formed by two finite maximal chains disjoined except at $x$ and $y$.
\end{definition}

You can check on Figure~\ref{fig:sweak_lattice} that the given examples of $s$-weak lattices are all polygonal, with only three possible polygons: squares, pentagons, and hexagons. 
Putting together some results of Lacina~\cite{SWeakSB}, we deduce that this is the case in general.

\begin{proposition}[from Lemmas 3.16 to 3.18 and Theorem 3.19 of \cite{SWeakSB}]
\label{prop:polygons}
The $s$-weak order is a polygonal lattice. 
More precisely, 
let $T$ be an $s$-decreasing tree such that $T$ is covered by $Z$ and $Q$ by rotating tree-ascents $(a,b)$ and $(c,d)$ respectively, with $a < c$. The interval $[T, Z \join Q]$ is either a square, a pentagon, or an hexagon depending on the following cases:
\begin{enumerate}
\item If $(a,b)$ is a tree-ascent of $Q$ and $(c,d)$ is a tree-ascent of $Z$, then $[T, Z \join Q]$ is a square such that the two maximal chains are given by
\begin{align*}
T & \xrightarrow{(a,b)} Z \xrightarrow{(c,d)} Z \join Q \\
T & \xrightarrow{(c,d)} Q \xrightarrow{(a,b)} Z \join Q
\end{align*}
as depicted in the first case of Figure~\ref{fig:polygons}.
\label{cond:polygons-square}
\item If $(a,b)$ is a tree-ascent of $Q$ but $(c,d)$ is not a tree-ascent of $Z$, then $b = c$ and $[T, Z \join Q]$ is a pentagon such that the two maximal chains are given by
\begin{align*}
T & \xrightarrow{(a,c)} Z \xrightarrow{(a,d)} Z' \xrightarrow{(c,d)} Z \join Q \\
T & \xrightarrow{(c,d)} Q \xrightarrow{(a,c)} Z \join Q
\end{align*}
as depicted in the second case of Figure~\ref{fig:polygons}.
\label{cond:polygons-pentagon-left}
\item If $(a,b)$ is not a tree-ascent of $Q$ and $(c,d)$ is a tree-ascent of $Z$, then $b = c$ and $[T, Z \join Q]$ is a pentagon such that the two maximal chains are given by
\begin{align*}
T & \xrightarrow{(a,c)} Z \xrightarrow{(c,d)} Z \join Q \\
T & \xrightarrow{(c,d)} Q \xrightarrow{(a,d)} Q' \xrightarrow{(a,c)} Z \join Q
\end{align*}
as depicted in the third case of Figure~\ref{fig:polygons}.
\label{cond:polygons-pentagon-right}
\item If $(a,b)$ is not a tree-ascent of $Q$ and $(c,d)$ is not a tree-ascent of $Z$, then $b = c$, $s(b) = 1=$, and $[T, Z \join Q]$ is an hexagon such that the two maximal chains are given by
\begin{align*}
T & \xrightarrow{(a,c)} Z \xrightarrow{(a,d)} Z' \xrightarrow{(c,d)} Z \join Q \\
T & \xrightarrow{(c,d)} Q \xrightarrow{(a,d)} Q' \xrightarrow{(a,c)} Z \join Q
\end{align*}
as depicted in the fourth case of Figure~\ref{fig:polygons}.
\label{cond:polygons-hexagon}
\end{enumerate}
\end{proposition}

\begin{figure}[ht]
\scalebox{.5}{\begin{tabular}{cccc}
\begin{tikzpicture}
\node (T) at (0,0) {$T$};
\node (Z) at (-2,2) {$Z$};
\node (Q) at (2,2) {$Q$};
\node (ZQ) at (0,4) {$Z \join Q$};

\draw (T) edge node[xshift=-.5cm] {$(a,b)$} (Z);
\draw (T) edge node[xshift=.5cm] {$(c,d)$} (Q);
\draw (Z) edge node[xshift=-.5cm] {$(c,d)$} (ZQ);
\draw (Q) edge node[xshift=.5cm] {$(a,b)$} (ZQ);
\end{tikzpicture}
&
\begin{tikzpicture}
\node (T) at (0,0) {$T$};
\node (Z) at (-2,2) {$Z$};
\node (Z2) at (-2,4) {$Z'$};
\node (Q) at (2,4) {$Q$};
\node (ZQ) at (0,6) {$Z \join Q$};

\draw (T) edge node[xshift=-.5cm] {$(a,c)$} (Z);
\draw (Z) edge node[xshift=-.5cm] {$(a,d)$} (Z2);
\draw (T) edge node[xshift=.5cm] {$(c,d)$} (Q);
\draw (Z2) edge node[xshift=-.5cm] {$(c,d)$} (ZQ);
\draw (Q) edge node[xshift=.5cm] {$(a,c)$} (ZQ);
\end{tikzpicture}
&
\begin{tikzpicture}
\node (T) at (0,0) {$T$};
\node (Z) at (-2,2) {$Z$};
\node (Q) at (2,2) {$Q$};
\node (Q2) at (2,4) {$Q'$};
\node (ZQ) at (0,6) {$Z \join Q$};

\draw (T) edge node[xshift=-.5cm] {$(a,c)$} (Z);
\draw (T) edge node[xshift=.5cm] {$(c,d)$} (Q);
\draw (Q) edge node[xshift=.5cm] {$(a,d)$} (Q2);
\draw (Q2) edge node[xshift=.5cm] {$(a,c)$} (ZQ);
\draw (Z) edge node[xshift=-.5cm] {$(c,d)$} (ZQ);
\end{tikzpicture}
&
\begin{tikzpicture}
\node (T) at (0,0) {$T$};
\node (Z) at (-2,2) {$Z$};
\node (Z2) at (-2,4) {$Z'$};
\node (Q) at (2,2) {$Q$};
\node (Q2) at (2,4) {$Q'$};
\node (ZQ) at (0,6) {$Z \join Q$};

\draw (T) edge node[xshift=-.5cm] {$(a,c)$} (Z);
\draw (Z) edge node[xshift=-.5cm] {$(a,d)$} (Z2);
\draw (Q) edge node[xshift=.5cm] {$(a,d)$} (Q2);
\draw (T) edge node[xshift=.5cm] {$(c,d)$} (Q);
\draw (Z2) edge node[xshift=-.5cm] {$(c,d)$} (ZQ);
\draw (Q2) edge node[xshift=.5cm] {$(a,c)$} (ZQ);
\end{tikzpicture}
\\
\begin{tikzpicture}[yscale=1.5]
\node (T) at (0,0) {\scalebox{.6}{\input{figures/dtrees/T_012_two_facets1}}};
\node (Z) at (-2,2) {\scalebox{.6}{\input{figures/dtrees/T_012_two_facets2}}};
\node (Q) at (2,2) {\scalebox{.6}{\input{figures/dtrees/T_012_two_facets3}}};
\node (ZQ) at (0,4){\scalebox{.6}{\input{figures/dtrees/T_012_two_facets4}}};

\draw (T) -- (Z);
\draw (T) -- (Q);
\draw (Z) -- (ZQ);
\draw (Q) -- (ZQ);
\end{tikzpicture}
&
\begin{tikzpicture}[yscale=1.5]
\node (T) at (0,0) {\scalebox{.6}{\input{figures/dtrees/T_022_two_facets4}}};
\node (Z) at (-2,2) {\scalebox{.6}{\input{figures/dtrees/T_022_two_facets6}}};
\node (Z2) at (-2,4) {\scalebox{.6}{\input{figures/dtrees/T_022_two_facets7}}};
\node (Q) at (2,4) {\scalebox{.6}{\input{figures/dtrees/T_022_two_facets4}}};
\node (ZQ) at (0,6) {\scalebox{.6}{\input{figures/dtrees/T_022_two_facets8}}};

\draw (T) -- (Z);
\draw (Z) -- (Z2);
\draw (T) -- (Q);
\draw (Z2) -- (ZQ);
\draw (Q) -- (ZQ);
\end{tikzpicture}
&
\begin{tikzpicture}[yscale=1.5]
\node (T) at (0,0) {\scalebox{.6}{\input{figures/dtrees/T_022_two_facets1}}};
\node (Z) at (-2,2) {\scalebox{.6}{\input{figures/dtrees/T_022_two_facets4}}};
\node (Q) at (2,2) {\scalebox{.6}{\input{figures/dtrees/T_022_two_facets2}}};
\node (Q2) at (2,4) {\scalebox{.6}{\input{figures/dtrees/T_022_two_facets3}}};
\node (ZQ) at (0,6){\scalebox{.6}{\input{figures/dtrees/T_022_two_facets5}}};

\draw (T) -- (Z);
\draw (T) -- (Q);
\draw (Q) -- (Q2);
\draw (Q2) -- (ZQ);
\draw (Z) --  (ZQ);
\end{tikzpicture}
&
\begin{tikzpicture}[yscale=1.5]
\node (T) at (0,0) {\scalebox{.6}{\input{figures/dtrees/T_012_two_facets5}}};
\node (Z) at (-2,2) {\scalebox{.6}{\input{figures/dtrees/T_012_two_facets6}}};
\node (Z2) at (-2,4) {\scalebox{.6}{\input{figures/dtrees/T_012_two_facets7}}};
\node (Q) at (2,2) {\scalebox{.6}{\input{figures/dtrees/T_012_two_facets1}}};
\node (Q2) at (2,4) {\scalebox{.6}{\input{figures/dtrees/T_012_two_facets2}}};
\node (ZQ) at (0,6) {\scalebox{.6}{\input{figures/dtrees/T_012_two_facets8}}};

\draw (T) -- (Z);
\draw (Z) -- (Z2);
\draw (Q) -- (Q2);
\draw (T) --  (Q);
\draw (Z2) -- (ZQ);
\draw (Q2) -- (ZQ);
\end{tikzpicture}
\end{tabular}}
\caption{The four possible intervals between $T$ and $Z \join Q$ where $T \wocover Z$ and $T \wocover Q$}
\label{fig:polygons}
\end{figure}

Using this nice characterization of the polygons, we can prove that the $s$-weak order is a semidistributive lattice. 
As we will see,  semidistributivity is a necessary condition to be congruence uniform, and a necessary condition to be an $\HH$ lattice.

\begin{definition}
A lattice is said to be \defn{meet-semidistributive} if for all elements $x,y,z$, we have
\begin{align}
z \meet x = z \meet y \Rightarrow z \meet \left( x \join y \right) = z \meet x = z \meet y.
\end{align}

Similarly, it is \defn{join-semidistributive} if
\begin{align}
z \join x = z \join y \Rightarrow z \join \left( x \meet y \right) = z \join x = z \join y.
\end{align}

A lattice is \defn{semidistributive} if it is both meet-semidistributive and join-semidistributive.
\end{definition}

A direct proof that the lattice of permutations (weak order of type $A$) is semidistributive can be found in~~\cite{DuquenneCherfouh94}. The general result of semidistributivity for the weak order of finite Coxeter groups can be found in~\cite{PolyBarbut94}. 

\begin{proposition}
\label{prop:semidistributive}
The $s$-weak order is a semidistributive lattice.
\end{proposition}

\begin{proof}
It suffices to prove join-semidistributivity. The meet-semidistributivity is obtained by symmetry of the lattice. We use a criterion from~\cite[Lemma 9.2-6]{Reading2016} which states that join-semidistributivity is equivalent to: for all $x,y,z$ such that $y$ and $z$ are covered by a common element, and $x \join y = x \join z$, then $x \join (y \meet z) = x \join z = x \join y$.

We write $T = y \meet z$ and $T' = x \join y = x \join z$. We want to prove that $T \join x = T'$. As $y$ and $z$ are covered by a common element, they fall under one of the four possibilities described in Proposition~\ref{prop:polygons}, \ie $y \join z$ is the top of a polygon (square, petagon, or hexagon) which initial element is $T$. Let us suppose that $y$ is on the chain which initial rotation is through the tree ascent $(a,b)$ and $z$ is on the chain which initial rotation is through the tree ascent $(c,d)$. First observe that as $T \wole y$, then $T \join x \wole y \join x$. Now, the different cases of Proposition~\ref{prop:polygons} give different expressions for the inversions of $T'$,
\begin{enumerate}
\item $\inv(T') = \tc{(\inv(x) \cup (\inv(T) + (b,a)))} = \tc{(\inv(x) \cup (\inv(T) + (d,c)))}$;
\item $\inv(T') = \tc{(\inv(x) \cup (\inv(T) + (c,a) + (d,a)))} = \tc{(\inv(x) \cup (\inv(T) + (d,c)))}$;
\item $\inv(T') = \tc{(\inv(x) \cup (\inv(T) + (c,a)))} = \tc{(\inv(x) \cup (\inv(T) + (d,c) + (d,a)))}$;
\item $\inv(T') = \tc{(\inv(x) \cup (\inv(T) + (c,a) + (d,a)))} = \tc{(\inv(x) \cup (\inv(T) + (d,c) + (d,a)))}$.
\end{enumerate}

Let us look at case~\eqref{cond:polygons-square} where the polygon is a square. We need to prove that $\card_{T \join x}(d,c) \geq \card_T(d,c) + 1$. This would imply that $(\inv(T) + (d,c)) \subseteq \inv(T \join x)$ and as by definition $\inv(x) \subseteq \inv(T \join x)$, by transitivity we would have $\inv(T') \subseteq \inv(T \join x)$. We know that $\card_{T'}(d,c) \geq \card_{z}(d,c) = \card_T(d,c) +1$. This means that there is a transitivity path between $d$ and $c$ in $(\inv(x) \cup (\inv(T) + (b,a)))$. As $c > a$, this transitivity path also exists directly in $\inv(x) \cup \inv(T)$ and so $\card_{T \join x}(d,c) \geq \card_T(d,c) + 1$. This proof also works for case~\eqref{cond:polygons-pentagon-left}.

For case~\eqref{cond:polygons-pentagon-right}, we need to prove that $\card_{T \join x}(c,a) \geq \card_T(c,a) + 1$. Again, we know that $\card_{T'}(c,a) \geq \card_y(c,a) = \card_T(c,a) + 1$. This means that there is a transitivity path in $(\inv(x) \cup (\inv(T) + (d,c) + (d,a)))$ between $c$ and $a$. As $c < d$, this transitivity path also exists in $\inv(x) \cup \inv(T)$ and so $\card_{T \join x}(c,a) \geq \card_T(c,a)+1$.

In case~\eqref{cond:polygons-hexagon}, a similar reasoning leads to $\card_{T \join x}(c,a) \geq \card_T(c,a) + 1$ and $\card_{T \join x}(d,c) \geq \card_T(d,c) +1$. By transitivity, we obtain $\card_{T \join x}(d,a) \geq \card_T(d,a) + 1$. This implies $\inv(T') \subseteq \inv(T \join x)$.
\end{proof}

Now, we are finally ready to show that the $s$-weak order is an $\HH$ lattice. We recall the definition from~\cite{CasparPBM2004}.

\begin{definition}[{\cite[Definition~10]{CasparPBM2004}}]
Given a lattice $L$, we denote by $E(L)$ the set of covering relations of $L$. 
We say that $L$ is an \defn{$\HH$-lattice} if it is finite, semidistributive, polygonal\footnote{Note that the notion of polygonal lattice had not yet been introduced in~\cite{CasparPBM2004} but in their definition of $\HH$-lattices, they describe an equivalent notion.} and there exist a labeling function $$\ell : E(L) \longrightarrow \mathcal{L}$$ where $\mathcal{L}$ is a set of labels, and a ranking function $$r : \mathcal{L} \longrightarrow \NN$$ satisfying the following conditions on every polygon $[x,y]$ of $L$.

We denote by $x_1$, $x_2$ the two elements covering $x$, and by $y_1$, $y_2$ the two elements covered by $y$, such that $x_1$ and $y_1$ (resp. $x_2$ and $y_2$) belong to the same maximal chain. The labeling $\ell$ and rank function $r$ must satisfy:

\begin{enumerate}
\item $\ell(x,x_1) = \ell(y_2, y)$ and $\ell(x,x_2) = \ell(y_1,y)$;
\label{cond:HH-labels}
\item for $t_1, \dots, t_k$ the labels of a maximal chain in a polygon, then 
\label{cond:HH-ranks}
\begin{align*}
r(t_1),r(t_k) < r(t_2), r(t_{k-1}) < \dots < r(t_{\frac{k+1}{2}}) &\text{if }k \text{ is odd,} \\
r(t_1),r(t_k) < r(t_2), r(t_{k-1}) < \dots < r(t_{\frac{k}{2}}),r(t_{\frac{k}{2}+1}) &\text{if }k \text{ is even.} 
\end{align*}
\end{enumerate}
\end{definition}
\medskip

For example, the labels on the four polygons shown on top of Figure~\ref{fig:polygons} satisfy Condition~\eqref{cond:HH-labels}. We will show that they give a general rule to label all cover relations of the $s$-weak order and define an appropriate rank function. 

\begin{theorem}[{\cite[Corollary~1]{CasparPBM2004}}]
Every $\HH$-lattice is congruence uniform.
\end{theorem}

This motivates the main result of this section stated below.
Similarly as in our proof of the semidistributivity property, our proof of the $\HH$ property relies on the exhaustive description of the polygons appearing in the $s$-weak order provided by Lacina in~\cite{SWeakSB}.

\begin{theorem}
\label{thm:hh}
The $s$-weak order is an $\HH$ lattice. Hence, it is polygonal, semidistributive and congruence uniform.
\end{theorem}

\begin{proof}
By Proposition~\ref{prop:polygons}, the $s$-weak lattice is polygonal. By Proposition~\ref{prop:semidistributive}, it is also semidistributive. Let $T$ be an $s$-decreasing tree covered by a tree $T'$ through the tree rotation $(a,b)$. We define the labeling function $\mathcal{L}$ by $\mathcal{L}(T,T') = (a,b)$ and the rank function $r$ by $r(a,b) = b - a$. We need to check that $\mathcal{L}$ and $r$ satisfy the conditions of $\HH$ lattices.

Looking at the four possible polygons of Proposition~\ref{prop:polygons}, it is immediate to see that $\mathcal{L}$ satisfies the condition. As for the ranking $r$, observe that the maximal chains in the polygons can be of length either~$2$ or~$3$. There is nothing to check for the length $2$ ones. For length $3$, if the chain is labeled by $(t_1,t_2,t_3)$ we need to have $r(t_1) < r(t_2)$ and $r(t_3) < r(t_2)$. In all cases, we have $t_1 = (a,c)$ and $t_3 = (c,d)$ or the other way around, and $t_2 = (a,d)$. As $d > c > a$, we have indeed that $r(a,c) = c - a < d - a = r(a,d)$ and $r(c,d) = d - c < d - a = r(a,d)$. 
\end{proof}

\begin{remark}
A lattice is said to be \defn{congruence uniform} if it is semidistributive and it can be obtained from the one element set by a sequence of \emph{doublings} of convex sets~\cite{day_splitting_1977}; we refer to~\cite{reading_latticeproperties_2003} for more details. 
Equivalently, a lattice is congurence uniform if it can be obtained from the one element lattice by a sequence of doublings of intervals~\cite{day_splitting_1977}. 
In forthcoming work, we will provide explicit sequence of doublings giving rise to the $s$-weak order and to the $s$-Tamari lattice defined below, thus giving an alternative proof of the congruence uniform property.
\end{remark}

\section{The $s$-Tamari lattice}

In this section we study properties of a sublattice of the $s$-weak order that we call the $s$-Tamari lattice. 
We show that if $s$ contains no zeros, then it can also be obtained as a quotient lattice of the $s$-weak oder. 
For a general signature $s$, we provide a simple description of the cover relations in terms of certain rotation on trees, 
and show that the $s$-Tamari lattice is isomorphic to a well chosen $\nu$-Tamari lattice, a poset introduced by Pr\'eville-Ratelle and Viennot in~\cite{PrevilleRatelleViennot}. 

\subsection{The sublattice of $s$-Tamari trees}
\begin{definition}
\label{def:s-tam-tree}
An $s$-decreasing tree $T$ is called an \defn{$s$-Tamari tree} if for any $a<b<c$ the number of~$(c,a)$ inversions is less than or equal to the number of~$(c,b)$ inversions: 
\[
\card_T(c,a)\leq \card_T(c,b).
\]
In terms of the tree this means that the node labels in $T^c_i$ are smaller than all the labels in $T^c_j$ for $i<j$.
The multi set of inversions of an $s$-Tamari tree is called an \defn{$s$-Tamari inversion set}.

The \defn{$s$-Tamari poset} is the restriction of the $s$-weak order to the set of $s$-Tamari trees. 
\end{definition}

In the case of permutations ($s = (1,1,\dots)$), $s$-Tamari trees are in correspondence with $231$-avoiding permutations. 
If the signature $s$ has no zeros, then $s$-decreasing trees correspond to $121$-avoiding \mbox{$s$-permutations}, 
and $s$-Tamari trees to the subset of those that are $231$-avoiding. 
Figure~\ref{fig:s-Tamari_trees023} illustrates an example of the lattice of $s$-decreasing trees and the sub-collection of $s$-Tamari trees for $s=(0,2,3)$. As we will see now, this sub-collection has the structure of a sublattice. 
The following result generalizes a result by Bj\"{o}rner and Wachs in~\cite[Section~9]{bjorner_wachs_shellableII_1997}, which slightly rephrased shows that the classical Tamari lattice can be realized as the sublattice of $231$-avoiding permutations of the classical weak order.

\begin{figure}[htbp]
\begin{center}
\begin{tabular}{cc}
\scalebox{1.5}{\input{figures/sweak_s023_stam.tex}}
&
\includegraphics[width=0.5\textwidth]{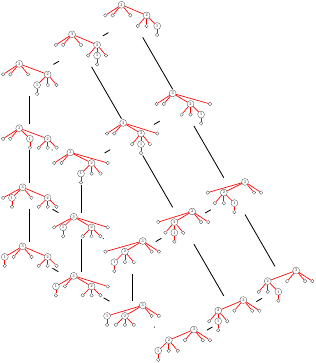}
\end{tabular}
\end{center}
\caption{$s$-decreasing trees and $s$-Tamari trees for $s=(0,2,3)$. On the left picture, the non $s$-Tamari trees are in black.}
\label{fig:s-Tamari_trees023}
\end{figure}

\begin{theorem}
\label{thm:stam-sublattice}
The $s$-Tamari poset is a sublattice of the $s$-week order. In particular, it is a lattice.
\end{theorem}

\begin{proof}
We need to show that the collection of $s$-Tamari trees is closed under the meet and join operations.

Let $T$ and $R$ be two $s$-Tamari trees. We will show that their join $T \join R$ is an $s$-Tamari tree, or in other words, that $\card_{T\join R}(c,a)\leq \card_{T\join R}(c,b)$ for any $a<b<c$.
It suffices to show this in the case where~$b=a+1$.
For this, we will use the following alternative description of the inversions of $T\join R$, which follows from Theorem~\ref{thm:lattice} and Definition~\ref{def:tc-tree-inversions}. 
Let us call a $(T,R)$-transitive path between $c>a$ to any list $c=b_1>b_2>\dots b_k=a$ such that $\card_T(b_i,b_{i+1})>0$ or $\card_R(b_i,b_{i+1})>0$ for all $1\leq i<k$. 
The inversion $\card_{T\join R}(c,a)$ is the highest possible value of $\max\{\card_T(b_1,b_2),\card_R(b_1,b_2)\}$ over all such $(T,R)$-transitive paths.
If $a<a+1<c$ we need to show 
\[
\card_{T\join R}(c,a) \leq \card_{T\join R}(c,a+1).
\]
Without loss of generality we can assume that $c=b_1>b_2>\dots b_k=a$ is a $(T,R)$-transitive path such that $\card_{T\join R}(c,a)=\card_T(b_1,b_2)$.
If $b_{k-1}=a+1$ then $c=b_1>b_2>\dots b_{k+1}=a+1$ is a $(T,R)$-transitive path between $c>a+1$ and therefore  
\[
\card_{T\join R}(c,a+1) \geq \card_T(b_1,b_2) =\card_{T\join R}(c,a),
\]
which finishes our claim.
In the other case, $b_{k-1}>a+1>a$ and since $T$ and $R$ are $s$-Tamari trees we have 
\[
\card_T(b_{k-1},a+1) \geq \card_T(b_{k-1},a)
\quad\quad \text{and} \quad\quad
\card_R(b_{k-1},a+1) \geq \card_R(b_{k-1},a).
\]
Since one of these two numbers is bigger than zero, then $c=b_1>b_2>\dots>b_{k-1}>a+1$ is also a $(T,R)$-transitive path between $c>a+1$. Thus, 
\[
\card_{T\join R}(c,a+1) \geq \card_T(b_1,b_2) = \card_{T\join R} (c,a),
\]
which finishes our proof for the join.

In remains to show that the meet $T\meet R$ is also an $s$-Tamari tree. Note that
Theorem~\ref{thm:lattice} does not give an explicit construction of the meet between two trees. 
However, by Remark~\ref{rem:meet}, the meet can be obtained by working on their mirror images. 
More precisely, if $\tilde T$ and $\tilde R$ are the mirror images of~$T$ and~$R$, respectively 
(this means $\card_{\tilde{T}}(b,a) = s(b) - \card_T(b,a)$ and the same for $R$),
then $T\meet R$ is the mirror image of $\tilde T \join \tilde R$.  
Therefore, proving that $T\meet R$ is an $s$-Tamari tree is equivalent to the following statement. If~$\tilde T$ and~$\tilde R$ are two trees such that $\card(c,b) \leq \card(c,a)$ for all $a < b <c$, then $\tilde T \join \tilde R$ also satisfies this property. 

We use a similar technique as for the join by proving the desired property for $a$ and $a+1$.
Consider a $(\tilde T,\tilde R)$-transitivity path $c=b_1>b_2>\dots>b_k=a+1$ between $c$ and $a+1$ where the maximal initial cardinality is attained: 
\[
\card_{\tilde T \join \tilde R}(c,a+1) =
\max\{\card_{\tilde T}(b_1,b_2),\card_{\tilde R}(b_1,b_2)\}.
\]
Since $a<b_2<b_1$, then by hypothesis 
\[
\card_{\tilde T}(b_1,b_2) \leq \card_{\tilde T}(b_1,a)
\]
\[
\card_{\tilde R}(b_1,b_2) \leq \card_{\tilde R}(b_1,a)
\]
Taking the maximum on both sides we get 
\begin{align*}
\card_{\tilde T \join \tilde R}(c,a+1) &\leq
\max\{ \card_{\tilde T}(b_1,a), \card_{\tilde R}(b_1,a) \} \\
& \leq 
 \card_{\tilde T\join \tilde R}(b_1,a) \\
 & = 
 \card_{\tilde T\join \tilde R}(c,a)
\end{align*}
This proves the desired property for $a$ and $a+1$, and finishes our proof.  
\end{proof}

\subsection{The $s$-Tamari lattice as a quotient lattice of the $s$-weak order}\label{sec_sTam_quotient}

The goal of this section is to describe the $s$-Tamari lattice defined above as a quotient lattice of the $s$-weak order. This can be done only when $s$ does not contain any zero (except on the first position which is irrelevant, or in the trivial case where all entries are equal to zero). Otherwise, we quickly run into counter examples: as shown on Figure~\ref{fig:quotient-counter-ex}, for $s = (0,0,1)$ the $s$-weak order is a square whereas the $s$-Tamari lattice is a chain of $3$ elements. The only non trivial quotient lattice of the square is a chain of $2$ elements, so the $s$-Tamari can not be obtained as a quotient lattice.

\begin{figure}[ht]
\begin{tabular}{ccc}
\scalebox{2}{\begin{tikzpicture}[every node/.style={inner sep = -.5pt}]
\node(tree0) at (0.000000000000000,0.000000000000000) {\scalebox{0.2}{
{ \newcommand{\nodea}{\node[draw,circle] (a) {$3$}
;}\newcommand{\nodeb}{\node[draw,circle] (b) {$2$}
;}\newcommand{\nodec}{\node[draw,circle] (c) {$1$}
;}\newcommand{\noded}{\node[draw,circle] (d) {$ $}
;}\newcommand{\nodee}{\node[draw,circle] (e) {$ $}
;}\begin{tikzpicture}[every node/.style={inner sep = 3pt}]
\matrix[column sep=.3cm, row sep=.3cm,ampersand replacement=\&]{
         \& \nodea  \&         \\ 
 \nodeb  \&         \& \nodee  \\ 
 \nodec  \&         \&         \\ 
 \noded  \&         \&         \\
};

\path[ultra thick, red] (c) edge (d)
	(b) edge (c)
	(a) edge (b) edge (e);
\end{tikzpicture}}
}};
\node(tree1) at (0.000000000000000,1.00000000000000) {\scalebox{0.2}{
{ \newcommand{\nodea}{\node[draw,circle] (a) {$3$}
;}\newcommand{\nodeb}{\node[draw,circle] (b) {$2$}
;}\newcommand{\nodec}{\node[draw,circle] (c) {$ $}
;}\newcommand{\noded}{\node[draw,circle] (d) {$1$}
;}\newcommand{\nodee}{\node[draw,circle] (e) {$ $}
;}\begin{tikzpicture}[every node/.style={inner sep = 3pt}]
\matrix[column sep=.3cm, row sep=.3cm,ampersand replacement=\&]{
         \& \nodea  \&         \\ 
 \nodeb  \&         \& \noded  \\ 
 \nodec  \&         \& \nodee  \\
};

\path[ultra thick, red] (b) edge (c)
	(d) edge (e)
	(a) edge (b) edge (d);
\end{tikzpicture}}
}};
\node[draw=red,circle,inner sep=-4pt](tree2) at (-0.866025403784439,0.500000000000000) {\scalebox{0.2}{
{ \newcommand{\nodea}{\node[draw,circle] (a) {$3$}
;}\newcommand{\nodeb}{\node[draw,circle] (b) {$1$}
;}\newcommand{\nodec}{\node[draw,circle] (c) {$ $}
;}\newcommand{\noded}{\node[draw,circle] (d) {$2$}
;}\newcommand{\nodee}{\node[draw,circle] (e) {$ $}
;}\begin{tikzpicture}[every node/.style={inner sep = 3pt}]
\matrix[column sep=.3cm, row sep=.3cm,ampersand replacement=\&]{
         \& \nodea  \&         \\ 
 \nodeb  \&         \& \noded  \\ 
 \nodec  \&         \& \nodee  \\
};

\path[ultra thick, red] (b) edge (c)
	(d) edge (e)
	(a) edge (b) edge (d);
\end{tikzpicture}}
}};
\node[draw=red,circle,inner sep=-4pt](tree3) at (-0.866025403784439,1.50000000000000) {\scalebox{0.2}{
{ \newcommand{\nodea}{\node[draw,circle] (a) {$3$}
;}\newcommand{\nodeb}{\node[draw,circle] (b) {$ $}
;}\newcommand{\nodec}{\node[draw,circle] (c) {$2$}
;}\newcommand{\noded}{\node[draw,circle] (d) {$1$}
;}\newcommand{\nodee}{\node[draw,circle] (e) {$ $}
;}\begin{tikzpicture}[every node/.style={inner sep = 3pt}]
\matrix[column sep=.3cm, row sep=.3cm,ampersand replacement=\&]{
         \& \nodea  \&         \\ 
 \nodeb  \&         \& \nodec  \\ 
         \&         \& \noded  \\ 
         \&         \& \nodee  \\
};

\path[ultra thick, red] (d) edge (e)
	(c) edge (d)
	(a) edge (b) edge (c);
\end{tikzpicture}}
}};

\draw (tree0) -- (tree1);
\draw (tree0) -- (tree2);
\draw (tree1) -- (tree3);
\draw (tree2) -- (tree3);

\draw[red] (0,0.5) ellipse (.4cm and .9cm);
\end{tikzpicture}} &
\scalebox{2}{\begin{tikzpicture}[every node/.style={inner sep = -.5pt}]
\node(tree0) at (0.000000000000000,0.000000000000000) {\scalebox{0.2}{
{ \newcommand{\nodea}{\node[draw,circle] (a) {$3$}
;}\newcommand{\nodeb}{\node[draw,circle] (b) {$2$}
;}\newcommand{\nodec}{\node[draw,circle] (c) {$1$}
;}\newcommand{\noded}{\node[draw,circle] (d) {$ $}
;}\newcommand{\nodee}{\node[draw,circle] (e) {$ $}
;}\begin{tikzpicture}[every node/.style={inner sep = 3pt}]
\matrix[column sep=.3cm, row sep=.3cm,ampersand replacement=\&]{
         \& \nodea  \&         \\ 
 \nodeb  \&         \& \nodee  \\ 
 \nodec  \&         \&         \\ 
 \noded  \&         \&         \\
};

\path[ultra thick, red] (c) edge (d)
	(b) edge (c)
	(a) edge (b) edge (e);
\end{tikzpicture}}
}};
\node(tree1) at (0.000000000000000,1.00000000000000) {\scalebox{0.2}{
{ \newcommand{\nodea}{\node[draw,circle] (a) {$3$}
;}\newcommand{\nodeb}{\node[draw,circle] (b) {$2$}
;}\newcommand{\nodec}{\node[draw,circle] (c) {$ $}
;}\newcommand{\noded}{\node[draw,circle] (d) {$1$}
;}\newcommand{\nodee}{\node[draw,circle] (e) {$ $}
;}\begin{tikzpicture}[every node/.style={inner sep = 3pt}]
\matrix[column sep=.3cm, row sep=.3cm,ampersand replacement=\&]{
         \& \nodea  \&         \\ 
 \nodeb  \&         \& \noded  \\ 
 \nodec  \&         \& \nodee  \\
};

\path[ultra thick, red] (b) edge (c)
	(d) edge (e)
	(a) edge (b) edge (d);
\end{tikzpicture}}
}};
\node(tree2) at (-0.866025403784439,0.500000000000000) {\scalebox{0.2}{
{ \newcommand{\nodea}{\node[draw,circle] (a) {$3$}
;}\newcommand{\nodeb}{\node[draw,circle] (b) {$1$}
;}\newcommand{\nodec}{\node[draw,circle] (c) {$ $}
;}\newcommand{\noded}{\node[draw,circle] (d) {$2$}
;}\newcommand{\nodee}{\node[draw,circle] (e) {$ $}
;}\begin{tikzpicture}[every node/.style={inner sep = 3pt}]
\matrix[column sep=.3cm, row sep=.3cm,ampersand replacement=\&]{
         \& \nodea  \&         \\ 
 \nodeb  \&         \& \noded  \\ 
 \nodec  \&         \& \nodee  \\
};

\path[ultra thick, red] (b) edge (c)
	(d) edge (e)
	(a) edge (b) edge (d);
\end{tikzpicture}}
}};
\node(tree3) at (-0.866025403784439,1.50000000000000) {\scalebox{0.2}{
{ \newcommand{\nodea}{\node[draw,circle] (a) {$3$}
;}\newcommand{\nodeb}{\node[draw,circle] (b) {$ $}
;}\newcommand{\nodec}{\node[draw,circle] (c) {$2$}
;}\newcommand{\noded}{\node[draw,circle] (d) {$1$}
;}\newcommand{\nodee}{\node[draw,circle] (e) {$ $}
;}\begin{tikzpicture}[every node/.style={inner sep = 3pt}]
\matrix[column sep=.3cm, row sep=.3cm,ampersand replacement=\&]{
         \& \nodea  \&         \\ 
 \nodeb  \&         \& \nodec  \\ 
         \&         \& \noded  \\ 
         \&         \& \nodee  \\
};

\path[ultra thick, red] (d) edge (e)
	(c) edge (d)
	(a) edge (b) edge (c);
\end{tikzpicture}}
}};

\draw (tree0) -- (tree1);
\draw (tree0) -- (tree2);
\draw (tree1) -- (tree3);
\draw (tree2) -- (tree3);

\draw[red] (0,0.5) ellipse (.4cm and .9cm);
\draw[red] (-0.866025403784439,1) ellipse (.4cm and .9cm);
\end{tikzpicture}} &
\qquad
\includegraphics[scale=2]{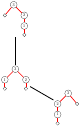} \\
$\pidown$ & $\piup$ & $s$-Tamari
\end{tabular}
\caption{The $s$-weak order with the equivalence relations $\pidown$ and $\piup$ and the $s$-Tamari lattice for $s = (0,0,1)$}
\label{fig:quotient-counter-ex}
\end{figure}
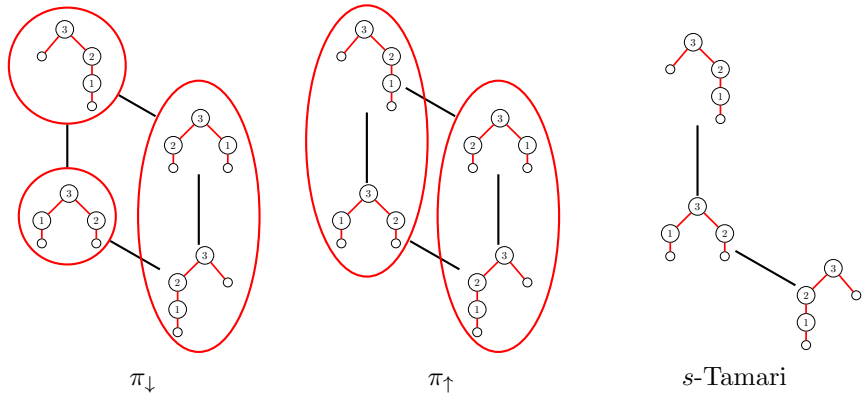

In the remaining of this section, we therefore work with a composition $s$ (without zeros). 
We briefly recall the definitions of lattice congruences and quotient lattices, and refer to~\cite{Reading-latticeCongruences,Reading-cambrianLattices} for more details.

\begin{definition}[Lattice congruence and quotient lattice, {\cite{Reading-latticeCongruences}}]
\label{def:lattice-congruence}
An equivalence relation $\alpha$ on a finite lattice $L$ is a \defn{lattice congruence} if the following conditions hold:
\begin{enumerate}[(i)]
\item Each equivalence class is an interval in $L$.
\item The map $\pidown^\alpha$ mapping each element to the bottom element of its equivalence class is order-preserving.
\item The map $\piup^\alpha$ mapping each element to the top element of its equivalence class is order-preserving.
\end{enumerate}
In particular, $\alpha$ is compatible with meets and joins: if $x,x'$ are in the same class and $y,y'$ are in the same class, then $x\meet y$ and $x'\meet y'$ (resp. $x\join y$ and $x'\join y'$) are in the same class.

The \defn{quotient} $L/\alpha$ is the poset on the equivalence classes of $\alpha$, 
where the order relation is defined by~$X\wole Y$ in $L/\alpha$ if there are representatives $x\in X$ and $y\in Y$ such that $x\wole y$ in $L$. This quotient inherits a lattice structure where the meet $X\meet Y$ (resp. the join $X\join Y$) of two congruence classes $X$ and $Y$ is the congruence class of $x \meet y$ (resp. of $x \join y$), for arbitrary representatives $x \in X$ and $y \in Y$.

Note that the quotient lattice $L/\alpha$ is isomorphic to the subposet of $L$ induced by $\pidown^\alpha(L)$ (or equivalently by $\piup^\alpha(L)$). This subposet is not necessarily a sublattice of $L$ in general.
\end{definition}

In our case, it is actually more convenient to directly define the equivalence relation in terms of the projection~$\pidown$.

\begin{proposition}
\label{prop:pidown}
For an $s$-decreasing tree $T$, the multi inversion set defined by 
\[
\card_{Q}(c,a) = \min\left( \card_T(c,b); a \leq b < c \right)
\]
for $c > a$, is an $s$-tree-inversion set. 
Abusing of notation, we define $\pidown(T) = Q$ to be the corresponding $s$-decreasing tree. 
Then, $Q$ is an $s$-Tamari tree and $Q \wole T$ in the $s$-weak order .
\end{proposition}

An example is illustrated in Figure~\ref{fig:pidownup}. In this example, we obtain $\card_{Q}(5,2) = 0$ because $\card_T(5,3) = 0$. Similarly, we get $\card_Q(5,1) = 0$ and $\card_Q(4,1) = 0$.

\begin{proof}
First we show that the multi inversion set is an $s$-tree-inversion set as defined in Definition~\ref{def:tree-inversion-set}. For this we need to verify the transitivity and planarity conditions. 

Let us check the transitivity. Let $a < b < c$ such that 
\begin{equation}\label{eq_proof_projectingdown_one}
\card_{Q}(c,b) = 
\min(\card_T(c,b'); b \leq b' < c) = i.
\end{equation}
Now 
\begin{equation}\label{eq_proof_projectingdown_two}
\card_{Q}(c,a) = \min(\card_T(c,a'); a \leq a' < c).
\end{equation}
We consider two possible cases.
If the minimal in Equation~\eqref{eq_proof_projectingdown_two} is obtained only with $a' < b$,
then we have that $\card_T(c,a') < \card_T(c,b)$ which implies~\mbox{$\card_T(b,a') = 0$} by transitivity on $T$. This gives that $\card_{Q}(b,a) = \min(\card_{T}(b,a'); a \leq a' < b) = 0$, and so the transitivity condition for $Q$ is satisfied in this case. 
For the second case, assume that the minimal in Equation~\eqref{eq_proof_projectingdown_two} is obtained for some $a' \geq b$, then $\card_{Q}(c,a) = \min(\card_T(c,a'); b \leq a' < c) = \card_{Q}(c,b)$. So the multi inversion is transitive in this case as well. 

The planarity condition follows from $\card_{Q}(c,a) \leq \card_{Q}(c,b)$, which is obtained directly by comparing Equations~\eqref{eq_proof_projectingdown_one} and~\eqref{eq_proof_projectingdown_two}. Since the transitivity and planarity conditions are satisfied, the multi inversion set is an $s$-tree inversion set. Thus, it is the tree inversion set of an $s$-decreasing tree $Q$ by Proposition~\ref{prop:tree-inversions-bij}.  

The fact that $\card_{Q}(c,a) \leq \card_{Q}(c,b)$ also proves that $Q$ is an $s$-Tamari tree. Furthermore, since $\card_{Q}(c,a) \leq \card_T(c,a)$ by definition, then $Q \wole T$.
\end{proof}

\begin{figure}[ht]
\begin{tabular}{cc}
\multicolumn{2}{c}{$T$} \\
\multicolumn{2}{c}{
{ \newcommand{\nodea}{\node[draw,circle] (a) {$5$}
;}\newcommand{\nodeb}{\node[draw,circle] (b) {$3$}
;}\newcommand{\nodec}{\node[draw,circle] (c) {$$}
;}\newcommand{\noded}{\node[draw,circle] (d) {$$}
;}\newcommand{\nodee}{\node[draw,circle] (e) {$4$}
;}\newcommand{\nodef}{\node[draw,circle] (f) {$2$}
;}\newcommand{\nodeg}{\node[draw,circle] (g) {$$}
;}\newcommand{\nodeh}{\node[draw,circle] (h) {$$}
;}\newcommand{\nodei}{\node[draw,circle] (i) {$1$}
;}\newcommand{\nodej}{\node[draw,circle] (j) {$$}
;}\newcommand{\nodeba}{\node[draw,circle] (ba) {$$}
;}\begin{tikzpicture}[auto]
\matrix[column sep=.3cm, row sep=.3cm,ampersand replacement=\&]{
         \&         \&         \&         \&         \&         \& \nodea  \&         \&         \\ 
         \& \nodeb  \&         \&         \&         \&         \& \nodee  \&         \& \nodeba \\ 
 \nodec  \&         \& \noded  \&         \& \nodef  \&         \&         \& \nodei  \&         \\ 
         \&         \&         \& \nodeg  \&         \& \nodeh  \&         \& \nodej  \&         \\
};

\path[ultra thick, red] (b) edge (c) edge (d)
	(f) edge (g) edge (h)
	(i) edge (j)
	(e) edge (f) edge (i)
	(a) edge (b) edge (e) edge (ba);
\end{tikzpicture}}
}
\\
$Q = \pidown(T)$ & $R = \piup(T)$ \\
{ \newcommand{\nodea}{\node[draw,circle] (a) {$5$}
;}\newcommand{\nodeb}{\node[draw,circle] (b) {$3$}
;}\newcommand{\nodec}{\node[draw,circle] (c) {$$}
;}\newcommand{\noded}{\node[draw,circle] (d) {$2$}
;}\newcommand{\nodee}{\node[draw,circle] (e) {$$}
;}\newcommand{\nodef}{\node[draw,circle] (f) {$1$}
;}\newcommand{\nodeg}{\node[draw,circle] (g) {$$}
;}\newcommand{\nodeh}{\node[draw,circle] (h) {$4$}
;}\newcommand{\nodei}{\node[draw,circle] (i) {$$}
;}\newcommand{\nodej}{\node[draw,circle] (j) {$$}
;}\newcommand{\nodeba}{\node[draw,circle] (ba) {$$}
;}\begin{tikzpicture}[auto]
\matrix[column sep=.3cm, row sep=.3cm,ampersand replacement=\&]{
         \&         \&         \&         \&         \&         \& \nodea  \&         \&         \\ 
         \& \nodeb  \&         \&         \&         \&         \& \nodeh  \&         \& \nodeba \\ 
 \nodec  \&         \&         \& \noded  \&         \& \nodei  \&         \& \nodej  \&         \\ 
         \&         \& \nodee  \&         \& \nodef  \&         \&         \&         \&         \\ 
         \&         \&         \&         \& \nodeg  \&         \&         \&         \&         \\
};

\path[ultra thick, red] (f) edge (g)
	(d) edge (e) edge (f)
	(b) edge (c) edge (d)
	(h) edge (i) edge (j)
	(a) edge (b) edge (h) edge (ba);
\end{tikzpicture}}
&
{ \newcommand{\nodea}{\node[draw,circle] (a) {$5$}
;}\newcommand{\nodeb}{\node[draw,circle] (b) {$3$}
;}\newcommand{\nodec}{\node[draw,circle] (c) {$$}
;}\newcommand{\noded}{\node[draw,circle] (d) {$$}
;}\newcommand{\nodee}{\node[draw,circle] (e) {$4$}
;}\newcommand{\nodef}{\node[draw,circle] (f) {$$}
;}\newcommand{\nodeg}{\node[draw,circle] (g) {$$}
;}\newcommand{\nodeh}{\node[draw,circle] (h) {$2$}
;}\newcommand{\nodei}{\node[draw,circle] (i) {$$}
;}\newcommand{\nodej}{\node[draw,circle] (j) {$1$}
;}\newcommand{\nodeba}{\node[draw,circle] (ba) {$$}
;}\begin{tikzpicture}[auto]
\matrix[column sep=.3cm, row sep=.3cm,ampersand replacement=\&]{
         \&         \&         \&         \& \nodea  \&         \&         \&         \&         \\ 
         \& \nodeb  \&         \&         \& \nodee  \&         \&         \& \nodeh  \&         \\ 
 \nodec  \&         \& \noded  \& \nodef  \&         \& \nodeg  \& \nodei  \&         \& \nodej  \\ 
         \&         \&         \&         \&         \&         \&         \&         \& \nodeba \\
};

\path[ultra thick, red] (b) edge (c) edge (d)
	(e) edge (f) edge (g)
	(j) edge (ba)
	(h) edge (i) edge (j)
	(a) edge (b) edge (e) edge (h);
\end{tikzpicture}}
\end{tabular}
\caption{A tree with its two projections $\pidown$ and $\piup$.}
\label{fig:pidownup}
\end{figure}
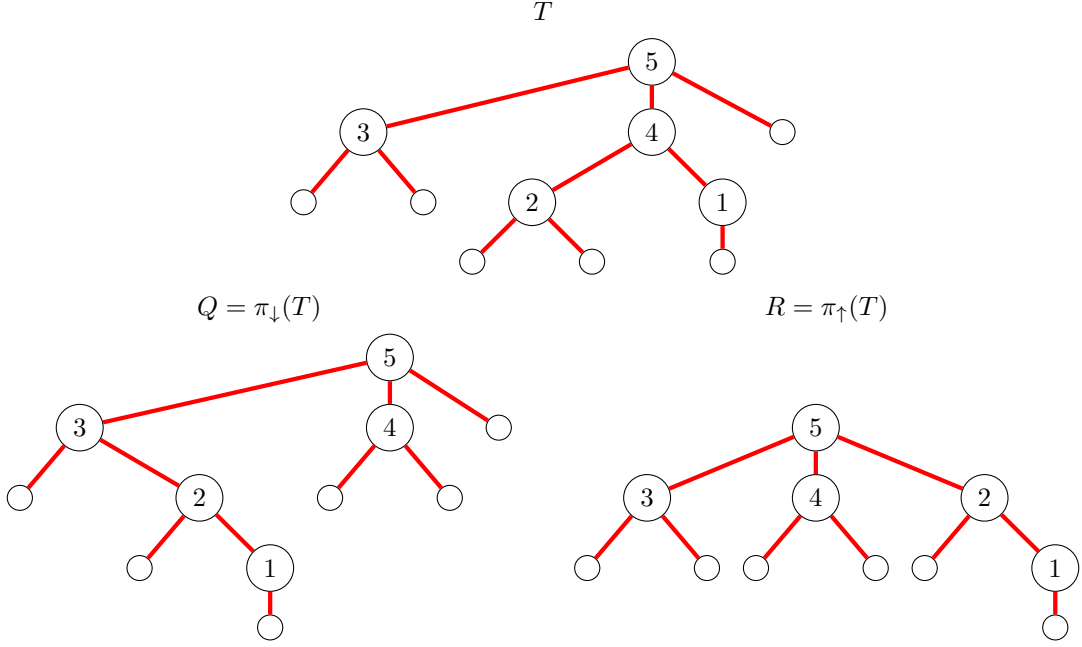

\begin{lemma}
Let $T$ be an $s$-decreasing tree, then $\pidown(\pidown(T)) = \pidown(T)$. In other words, $\pidown$ is indeed a projection. In particular, if $T$ is an $s$-Tamari tree, $\pidown(T) = T$.
\end{lemma}

\begin{proof}
We proved in Proposition~\ref{prop:pidown} that the image of $\pidown$ is an $s$-Tamari tree. 
Therefore, it suffices to show that $\pidown(T) = T$ when $T$ is an $s$-Tamari tree.
This follows immediately from Definition~\ref{def:s-tam-tree}. 
\end{proof}

\begin{definition}
Two $s$-decreasing trees $T$ and $T'$ are \defn{$\pidown$-equivalent}, in which case we write~$T \epidown T'$, if and only if $\pidown(T) = \pidown(T')$. The classes of this equivalent relation are called the \defn{$s$-Tamari classes}.
\end{definition}

The relation $\epidown$ is trivially an equivalence relation, and our goal is to prove that it is a lattice congruence. Some examples are shown on Figure~\ref{fig:s-tamari-class}. You can check in particular that each class forms an interval.

\begin{figure}[ht]
\begin{tabular}{cc}
\includegraphics[scale=.5]{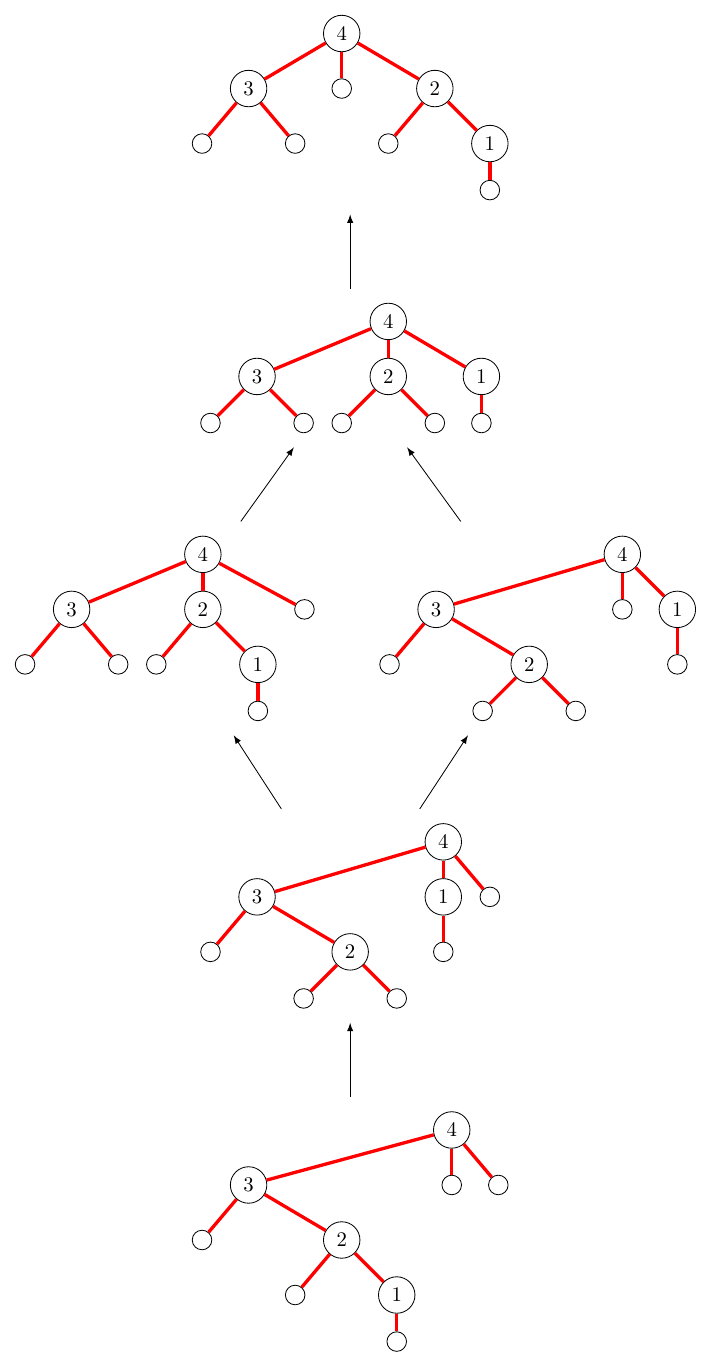}
&
\scalebox{1.5}{\input{figures/sweak_quotient.tex}}
\end{tabular}
\caption{An $s$-Tamari class on the left and all classes for $s = (0,2,2)$ on the right.}
\label{fig:s-tamari-class}
\end{figure}

The following definition will allow us to present an alternative description of the $s$-Tamari classes.

\begin{definition}
\label{def:s-tam-congruence-rotation}
Let $T$ be a $s$-decreasing tree for a composition $s$ (without zeros), and let $(a,c)$ be a tree-ascent of $T$. We say that $(a,c)$ is an \defn{$s$-congruence} tree-ascent if and only if there exists $b$ with $a < b < c$ and $\card_T(b,a) = s(b)$. The corresponding $s$-tree rotation is called an \defn{$s$-congruence rotation}.
\end{definition} 

The following proposition shows that the $s$-tree rotations that preserve the class are exactly the $s$-congruence rotations. 

\begin{proposition}
\label{prop:s-tam-congruence-rotation}
Let $s$ be a composition (without zeros) and $T \wole T'$ be a cover relation of the $s$-weak order. 
Then, $T \epidown T'$ if and only if $T'$ is obtained by applying an $s$-congruence rotation on $T$. Moreover, if the cover relation is not an $s$-congruence rotation, then $\pidown(T) \woless \pidown(T')$.
\end{proposition}

You can check on Figure~\ref{fig:s-tamari-class} that each cover relation satisfies this condition.

\begin{proof}
Let $T$ be an $s$-decreasing tree with a tree-ascent $(a,c)$, and $T'$ the result of the corresponding $s$-tree-rotation. We write $Q = \pidown(T)$ and $Q' = \pidown(T')$. Let $A$ be the set containing the node $a$ and all its middle descendants. In particular $a' \leq a$ for all $a' \in A$. Remember by Lemma~\ref{lem_tree_rotation_inversions} that the only increased cardinalities in $T'$ are those of $(c,a')$ where $a' \in A$. 

Assume that $(a,c)$ is an $s$-congruence tree-ascent. This means that there exists $b$ with $a < b < c$ and  $\card_T(b,a) = s(b)$. Because $s(b) > 0$, by transitivity this gives $\card_T(c,a') \geq \card_T(c,b)$ for all $a' \in A$ (all elements of $A$ are right of~$b$). Thus, the minimal cardinality computed by $\pidown$ is not obtained through the elements of $A$, \emph{i.e.} for all $a'' \leq a$ (whether they are in $A$ or not), $\card_Q(c,a'') = \min(\card_T(c,b'); b' \notin A; a'' \leq b' \leq c)$. This is also true in $T'$ and we get $\card_{Q'}(c,a'') = \min(\card_{T'}(c,b'); b' \notin A; a'' \leq b' \leq c)$. These two expressions are equal because for all $b' \notin A$, we have $\card_T(c,b') = \card_{T'}(c,b')$. As these are the only changes between $T$ and $T'$, we obtain that $Q = Q'$.

Now suppose on the other hand that $(a,c)$ is not an $s$-congruence tree-ascent. This means that there is no $b$ with $a < b < c$ such that $\card_T(b,a) = s(b)$. In other words, for all $b$ with $a < b < c$, $a$ is left of $b$ and we have $\card_T(c,a) \leq \card_T(c,b)$ (this is true only because $s(b) > 0)$. This implies that $\card_Q(c,a) = \card_T(c,a)$. In $T'$, we have $\card_{T'}(c,a') = \card_T(c,a') + 1$ for all $a' \in A$. This includes $a$ and only elements smaller than~$a$. In particular, $\card_{T'}(c,b) = \card_T(c,b)$ and $\card_{T'}(b,a) = \card_T(b,a)$ for all $b$ with $a < b < c$. We get that $\card_{Q'}(c,a) = \card_{T'}(c,a) = \card_Q(c,a) + 1$. So~$Q \neq Q'$. Furthermore, the cardinality of $(c,a)$ has strictly increased, and for all $a'' \leq a$, we have 
$$\card_{Q'}(c,a'') = \min( \card_{T'}(c,b'); a'' \leq b' < c) \geq \min( \card_{T}(c,b'); a'' \leq b' < c) = \card_Q(c,a'').$$ 
Moreover, all other cardinalities are unchanged, and so $Q \woless Q'$.
\end{proof}

\begin{remark}
Note that this proof does not work when $s(b)$ can be equal to zero, more precisely whenever~$s$ contains a zero outside its extremal positions. In this case, there will be $s$-congruence rotations between elements of different $\epidown$ classes, as in Figure~\ref{fig:quotient-counter-ex}: the tree-ascent $(1,3)_0$ is always an $s$-congruence tree-ascent because $\card(2,1) = 0 = s(2)$ even when $1$ is left of $2$. 
\end{remark}

\begin{corollary}
\label{cor:pidown-order-pres}
Let $s$ be a composition (without zeros), then $\pidown$ is order preserving.
\end{corollary}

\begin{proof}
Proposition~\ref{prop:s-tam-congruence-rotation} states that if $T \wole T'$ is a cover relation of the $s$-weak order, then $\pidown(T) \wole \pidown(T')$ (the equality depending on whether or not the rotation is an $s$-congruence rotation). As this is true on cover relations, it can be extended to any $T \wole T'$. 
\end{proof}

We say that a subset $X$ of the $s$-weak order is \defn{closed by interval} if $T,T''\in X$ implies that $T'\in X$ for every $T \wole T' \wole T''$.

\begin{corollary}
\label{cor:s-tam-interval-closes}
Let $s$ be a composition (without zeros), the $s$-Tamari classes are closed by interval. 
\end{corollary}

\begin{proof}
Suppose that we have three $s$-decreasing trees $T \wole T' \wole T''$ such that $\pidown(T) = \pidown(T'') = Q$. Since $\pidown$ is order preserving, we get $Q \wole \pidown(T') \wole Q$, and so $\pidown(T') = Q$.
\end{proof}

\begin{corollary}
\label{cor:s-tam-equiv-equiv}
Let $s$ be a composition (without zeros), and $\equiv$ be the equivalence relation on $s$-decreasing trees defined by: $T \equiv T'$ if and only if $T$ and $T'$ can be connected by a sequence of $s$-congruence rotations. Then $\equiv$ is the same equivalence relation than $\epidown$. 
\end{corollary}

\begin{proof}
By Proposition~\ref{prop:s-tam-congruence-rotation},  $s$-congruence rotations preserve the classes of $\epidown$. 
As a consequence, the equivalence~$T \equiv T'$ implies $T \epidown T'$. 
Now assume that $T \epidown T'$, we need to show that $T \equiv T'$. 
Consider the $s$-Tamari tree $Q=\pidown(T)=\pidown(T')$ obtained by projecting down $T$ and $T'$. This is the minimal element in the class of $\epidown$ that contains $T$ and $T'$. 
By Corollary~\ref{cor:s-tam-interval-closes}, the intervals $[Q,T]$ and~$[Q,T']$ also contained in this class, 
and each covering relation inside one of these intervals corresponds to an $s$-congruence rotation by Proposition~\ref{prop:s-tam-congruence-rotation}.
Thus, $Q$ is connected to both $T$ and $T'$ by a sequence of $s$-congruence rotations, and so  $T \equiv T'$. 
\end{proof}

\begin{remark}
Note that in the case of permutations ($s = (1,1,\dots)$), the equivalence relation $\equiv$ is a slight deformation of the \defn{sylvester congruence} defined in~\cite[Definition 8]{HNT05}. More precisely, it is the transitive closure of the relation $\dots b \dots ac \dots \equiv \dots b \dots ca \dots$ with $a < b < c$. This small transformation corresponds to an $s$-congruence rotation in our setting.
\end{remark}

We have almost all the ingredients in order to prove that the equivalence relation $\epidown$ (or equivalently~$\equiv$) is a lattice congruence. We still need to prove that the induced $s$-Tamari classes have a unique maximal element, and that the projection $\piup$ is order preserving. 
We start by describing the maximal elements of the classes.

\begin{proposition}
Let $s$ be a composition (without zeros) and $T$ be an $s$-decreasing tree such that if~$\card_T(b,a) = s(b)$ then $\card_T(c,a) = s(c)$ for all $c > b$. Then, $T$ is maximal in its $s$-Tamari class and we call it an \defn{$s$-maximal-Tamari tree}.
\end{proposition}

In the case of permutations, $s$-maximal-Tamari trees correspond to $213$ avoiding permutations. You can check on Figure~\ref{fig:s-tamari-class} that the trees satisfying this condition are exactly the maximal elements of the classes.

\begin{proof}
This is also a consequence of Proposition~\ref{prop:s-tam-congruence-rotation}. Suppose that we have an $s$-maximal-Tamari tree~$T$ with a tree ascent $(a,c)$. Then, by definition of a tree-ascent, $\card_T(c,a) < s(c)$. As $T$ is an $s$-maximal-Tamari tree, this means that there is no $b$ with $a < b <c$ such that $\card_T(b,a) = s(b)$, and so $(a,c)$ is not an $s$-congruence rotation.
\end{proof}

\begin{proposition}\label{prop:piup}
For an $s$-decreasing tree $T$, the multi inversion set defined by 
\[
\card_R(c,a) =
\begin{cases}
s(c), & \text{if there exists } b \text{ with } a < b < c \text{ and } \card_T(b,a) = s(b)\\
\card_T(c,a) & \text{otherwise}
\end{cases}
\]
for $c > a$, is an $s$-tree-inversion set. 
Abusing of notation, we define $\piup(T) = R$ to be the corresponding $s$-decreasing tree. 
Then, $R$ is an $s$-maximal-Tamari tree and $T \wole R$ in the $s$-weak order .
\end{proposition}

\begin{proof}
First we show that the multi inversion set is an $s$-tree-inversion set as defined in Definition~\ref{def:tree-inversion-set}.
Again, we need to verify the transitivity and planarity conditions. 

\smallskip
Let us start with transitivity. Suppose that we have $a < b < c$ with $\card_R(c,b) = i$ and $\card_R(b,a) > 0$. We need to show that  $\card_R(c,a) \geq i$. We consider two possible cases. 

Case 1: Assume that $\card_T(b,a) = 0$. Then, there exists $b'$ with $a < b' < b$ and $\card_T(b',a) = s(b')$. By definition, this also implies that $\card_R(c,a) = s(c) \geq \card_R(c,b)=i$. 

Case 2: Assume that $\card_T(b,a) > 0$. If $\card_T(c,b) = i$, then transitivity in $T$ would imply $\card_T(c,a) \geq i$, which itself implies $\card_R(c,a) \geq i$ as desired. 
On the other hand, if $\card_T(c,b) \neq i$, then there exists $b'$ with $b < b' < c$ and $\card_T(b',b) = s(b')$. As we have assumed $\card_T(b,a) > 0$, this gives by transitivity that $\card_T(b',a) = s(b')$, and so $\card_R(c,a) = s(c)\geq i$ as desired. 

\smallskip
Now let us prove planarity. Let $a < b < c$ with $\card_R(c,a) = i$. 
We need to show that $\card_R(b,a) = s(b)$ or $\card_R(c,b) \geq i$. We consider again two cases.

Case 1: If $\card_T(c,a) = i$, then by planarity in $T$, either $\card_T(b,a) = s(b)$ or $\card_T(c,b) \geq i$. Both properties remain true for $R$ and so, planarity holds in this case. 

Case 2: If $\card_T(c,a) \neq i$, then there exists $b'$ with $a < b' < c$ and $\card_T(b',a) = s(b')$ (and $\card_R(c,a) = s(c)$).
If $b' \leq b$, this implies $\card_R(b,a) = s(b)$. 
If $b' > b$, this implies by construction $\card_R(c,b) = s(c)$. In both cases, the planarity condition holds.

\smallskip
We have shown that the multi inversion set $R$ satisfies the transitivity and planarity conditions, and therefore it is an $s$-tree inversion set. By Proposition~\ref{prop:tree-inversions-bij}, it is the tree inversion set of an $s$-decreasing tree, which we also denote by $R$ for simplicity.  

By definition of $R$, it is straight forward to check that if~$\card_R(b,a) = s(b)$ then $\card_R(c,a) = s(c)$ for all~$c > b$. Therefore, $R$ is an $s$-maximal-Tamari tree. The fact that $T \wole R$ is also clear from the definition, as we only increase cardinalitites.  
\end{proof}

\begin{lemma}
Let $T$ be an $s$-decreasing tree, then $\piup(\piup(T)) = \piup(T)$. In other words, $\piup$ is indeed a projection. In particular, if $T$ is an $s$-maximal-Tamari tree, $\piup(T) = T$.
\end{lemma}

\begin{proof}
We proved in Proposition~\ref{prop:piup} that the image of $\piup$ is an $s$-maximal-Tamari tree. 
Therefore, it suffices to show that $\piup(T) = T$ when $T$ is an $s$-maximal-Tamari tree.
This follows immediately from the definition. 
\end{proof}

Like $\pidown$, $\piup$ defines a \defn{$\piup$-equivalence relation} given by $T \epiup T'$ if and only if $\piup(T) = \piup(T')$. 

\begin{proposition}
\label{prop:s-tam-up-congruence-rotation}
Let $T \wole T'$ be a cover relation of the $s$-weak order. Then $T \epiup T'$ if and only if $T'$ is obtained by applying an $s$-congruence rotation on $T$. Moreover, if the cover relation is not an $s$-congruence rotation, then $\piup(T) \woless \piup(T')$.
\end{proposition}

\begin{proof}
Let $T$ be an $s$-decreasing tree with a tree-ascent $(a,c)$ and $T'$ the result of the corresponding $s$-tree-rotation. We write $R = \piup(T)$ and $R' = \piup(T')$. 
As in the proof of Proposition~\ref{prop:s-tam-congruence-rotation},
we denote by~$A$ the set containing the node $a$ and all its middle descendants, and recall that the only increased cardinalities in $T'$ are those of $(c,a')$ where $a' \in A$. 

Assume that $(a,c)$ is an $s$-congruence tree-ascent. This means that there exists $b$ with $a < b < c$ such that $\card_T(b,a) = s(b)$. Therefore, for all $a' \in A$, we also have $\card_T(b,a') = s(b)$. This implies that for all $a' \in A$, and all $c' \geq c$, we have $\card_R(c',a') = s(c)$. This is also true in $T'$, and all the other cardinalities stay unchanged, which gives~$R' = R$.

Now if $(a,c)$ is not an $s$-congruence tree-ascent, there exists no $b$ with $a < b < c$ and $\card_T(b,a) = s(b)$. This gives $\card_R(c,a) = \card_T(c,a)$. For the same reason and because no $(b,a)$ cardinality is modified by the rotation, we obtain $\card_{R'}(c,a) = \card_{T'}(c,a) = \card_R(c,a) + 1$. For all other $a' < c'$, we have $\card_R(c',a') \leq \card_{R'}(c',a')$ because the cardinality can only increase. We then obtain that $R \woless R'$.
\end{proof}

\begin{corollary}
\label{cor:piup}
Let $s$ be a composition (without zeros), then: 
\begin{enumerate}
\item $\piup$ is order preserving.
\item The classes of the $\epiup$ equivalence are closed by interval.
\item The three equivalence relations $\epiup$,  $\equiv$, and $\epidown$ are the same. 
\end{enumerate}
\end{corollary}

\begin{proof}
The proofs of the three items in this corollary are exactly the same as the proofs of the three Corollaries~\ref{cor:pidown-order-pres},~\ref{cor:s-tam-interval-closes}, and~\ref{cor:s-tam-equiv-equiv}, respectively, using Proposition~\ref{prop:s-tam-up-congruence-rotation} instead of Proposition~\ref{prop:s-tam-congruence-rotation}. 
\end{proof}

\begin{remark}
Note that the proofs of~Propositions~\ref{prop:piup} and~\ref{prop:s-tam-up-congruence-rotation} do not require $s(b) > 0$. Indeed, even when $s$ contains zeros, the projection $\piup$ is order preserving and $\epiup$ is the same as $\equiv$. Nevertheless, it is not the same as $\epidown$ (because Proposition~\ref{prop:s-tam-congruence-rotation} requires $s(b) > 0$). This can be seen on Figure~\ref{fig:quotient-counter-ex} where the $\epidown$-classes and $\epiup$-classes are circled. In this case, you can see that $\epiup$ is a lattice congruence but the quotient lattice is not the $s$-Tamari lattice. On the other hand, on Figure~\ref{fig:quotient-counter-ex}, $\epidown$ is not a lattice congruence: indeed the projection sending elements to the top of their $\epidown$ class is not order preserving.  
\end{remark}

\begin{theorem}\label{thm_sTam_quotientLattice}
For $s$ a composition (without zeros), 
the equivalence relation $\epiup$ (or equivalently any of $\equiv$ or $\epidown$) is a lattice congruence. 
Moreover, the $s$-Tamari lattice is the quotient lattice of the $s$-weak order determined by $\epidown$.
\end{theorem}

\begin{proof}
We have defined the $\epidown$ equivalence as the fibers of the projection $\pidown$ onto $s$-Tamari trees. As~$\pidown(T) \wole T$, the classes have a unique minimal element. Similarly, the classes of $\epiup$ have a unique maximal element. By Corollaries~\ref{cor:s-tam-interval-closes} and~\ref{cor:piup}, they are the same classes and are closed by interval. So, the classes are intervals themselves. Moreover, Corollaries~\ref{cor:pidown-order-pres} and~\ref{cor:piup} prove that $\pidown$ and $\piup$ are order preserving. As a consequence $\epidown$ is a lattice congruence as in Definition~\ref{def:lattice-congruence}.  

The quotient lattice determined by $\epidown$ is then isomorphic to the subposet of either the minimal class elements (the $s$-Tamari trees) or the maximal ones (the $s$-maximal-Tamari trees). The first is by definition the $s$-Tamari lattice.
\end{proof}

\begin{remark}
As a consequence, we get that the $s$-Tamari trees form a join-sublattice which we also proved directly in Theorem~\ref{thm:stam-sublattice}. But note that Theorem~\ref{thm:stam-sublattice} does not derives entirely from Theorem~\ref{thm_sTam_quotientLattice}: the fact that $s$-Tamari trees form a meet-sublattice is not a consequence of~Theorem~\ref{thm_sTam_quotientLattice} and, besides, Theorem~\ref{thm:stam-sublattice} is true even when $s$ contains zeros.
\end{remark}

\subsection{Cover relations}
Similarly as in the lattice of $s$-decreasing trees, the cover relations on the sublattice of $s$-Tamari trees can be described in terms of certain rotations on trees which we now describe. 

\begin{definition}
\label{def:Tamari-ascent}
Let $T$ be an $s$-Tamari tree of some weak composition $s$. We say that~$(a,c)$ with $a < c$ is a \defn{Tamari-ascent} of $T$ if $a$ is a non-right child of $c$. 
\end{definition}

Note that $a$ needs to be a direct child of $c$ (not a descendant). Figure~\ref{fig:stam-cover} shows an example with $4$ Tamari-ascents.

\begin{lemma}\label{lem_sTamari_rotation}
Let $T$ be an $s$-Tamari tree for some weak composition $s$, and $(a,c)$ a Tamari-ascent of~$T$. 
Then $\tc{\left( \inv(T) + (c,a) \right)}$ is an $s$-Tamari inversion set. The inversions of its corresponding $s$-Tamari tree~$T'$ are
\begin{equation}
\label{eq_sTamari_rotation_inversions}
\card_{T'}(c',a') =
\left\{
	\begin{array}{ll}
		\card_{T}(c',a')+1  & \mbox{if } c'=c \text{ and } a' \text{ is equal to } a \text{ or a non-left descendant of } a, \\
		\card_{T}(c',a') & \mbox{otherwise. } 
	\end{array}
\right.
\end{equation}
\end{lemma}

In other words, a node moves to the right with its right children and middle children.

\begin{proof}
Let $(c,a)$ be a pair of nodes in $T$ such that $a$ is a non-right child of $c$, and denote $K  = \inv(T) + (c,a)$ for simplicity. The proof follows the same steps as in the proof of Lemma~\ref{lem_tree_rotation_inversions}, the details are similar for Part 1 and exactly the same for Part 2. The last Part 3 checks that $T'$ is an $s$-Tamari tree.

\begin{proofpart}[$\tc{K}$ is a tree-inversion set]
By Lemma~\ref{prop:tc-transitive} and Lemma~\ref{prop:tc-planarity}, it suffices to show that $K$ satisfies the planarity condition. 
For this we need to show that for any $x<y<z$, $\card_K(z,x)=i$ implies $\card_K(y,x)=s(y)$ or $\card_K(z,y)\geq i$. We consider two possible cases. 

If $(z,x)\neq (c,a)$ then $\card_T(z,x)=\card_K(z,x)=i$, which implies $\card_T(y,x)=s(y)$ or $\card_T(z,y)\geq i$. Whichever condition holds is also true for $K$.

Now consider the case $a<b<c$. Here we have $\card_K(c,a)=\card_T(c,a)+1=i$, which implies $\card_T(b,a)=\card_K(b,a)=s(b)$ or $\card_T(c,b)=\card_K(c,b)\geq i-1$.
Note that $\card_T(c,b)$ can not be equal to $i-1$, otherwise $a<b$ would both belong to the same subtree $T_{i-1}^c$ of $T$. But this contradicts the fact that $T$ is a decreasing tree because $a$ is a child of $c$ and therefore the root of $T_{i-1}^c$ by assumption.
\end{proofpart}

\begin{proofpart}[Equation~\eqref{eq_sTamari_rotation_inversions} defines $\tc{K}$, the smallest transitive multi inversion set containing $K$]

We start by proving that the multi inversion set defined by Equation~\eqref{eq_sTamari_rotation_inversions} is transitive.

Let $x<y<z$ and assume that $\card_{T'}(z,y)=i$. We need to show that 
\[
\card_{T'}(y,x)=0 \quad \text{or} \quad \card_{T'}(z,x)\geq i.
\]

\begin{paragraph}{\bf{Case 1}}
{\bf $\mathbf{(z,y)=(c,a')}$ for some $a'$ a non-left descendant of $a$ in $T$.}
In this case $\card_{T'}(z,y)= \card_{T}(z,y)+1=i$, which implies $\card_{T}(y,x)=0$ or $\card_{T}(z,x)\geq i-1$ by transitivity of $T$.
If $\card_{T}(y,x)=0$ then $\card_{T'}(y,x)=0$ because $y\neq c$ and we are done. 
If $\card_{T}(z,x)> i-1$ then $\card_{T'}(z,x)\geq i$ and we are also done.
Now assume $\card_{T}(z,x)= i-1$ and $\card_{T}(y,x)>0$. By transitivity of $T$ for $x<y<a$, and given that $a'$ is a non-left descendant of $a$ in $T$, we get the following two inequalities:
\[
\card_T(a,x) \geq \card_T(a,y) = \card_T(a,a') > 0.
\]
Since $a$ is a child of $c$ in $T$ and $\card_T(a,x)>0$, then $x$ is necessarily a non-left descendant of~$a$, otherwise $\card_T(z,x)=\card_T(c,x)>\card_T(c,a)=i-1$ which contradicts our assumption.
Therefore $\card_{T'}(z,x)=\card_{T}(z,x)+1=i-1+1=i\geq i$ as desired.
\end{paragraph}

\medskip
\begin{paragraph}{\bf{Case 2}}
{\bf $\mathbf{(z,y) \neq (c,a')}$ for any non-left descendant $a'$ of $a$ in $T$.}
In this case \mbox{$\card_{T'}(z,y)= \card_{T}(z,y)=i$}, which implies that $\card_{T}(y,x)=0$ or $\card_{T}(z,x)\geq i$ by transitivity of $T$.
If $\card_{T}(z,x)\geq i$ then $\card_{T'}(z,x)\geq i$ and we are done. If $\card_T(y,x)=0$ and $\card_{T'}(y,x)=0$ we are also done. Now assume $\card_T(y,x)=0$ and $\card_{T'}(y,x)>0$. This happens exactly when $(y,x)=(c,a')$ for some $a'$ a non-left descendant of $a$ in $T$. Since $z>y$ and $x$ is a descendant of $y$ then $\card_T(z,x)=\card_T(z,y)=i$, and therefore $\card_{T'}(z,x)\geq i$ holds.
\end{paragraph}

\medskip
This finishes our proof that the multi inversion set defined by Equation~\eqref{eq_sTamari_rotation_inversions} is transitive. 
It clearly contains $K$, and is the smallest transitive multi inversion set containing it because of the transitivity condition. We conclude that it is equal to $\tc{K}$ by Lemma~\ref{prop:tc-transitive}. 
\end{proofpart}

\begin{proofpart}[$\tc{K}$ is an $s$-Tamari inversion set]
We need to check that for any $a<b<c$ we have that $\card_{T'}(c,a)\leq \card_{T'}(c,b)$. This follows by inspection from Equation~\eqref{eq_sTamari_rotation_inversions}.
\end{proofpart}
The lemma follows froms Parts 1, 2, and 3.
\end{proof}

\begin{definition}
If $(a,c)$ is a Tamari-ascent of an $s$-Tamari tree $T$, we call the $s$-Tamari tree corresponding to $\tc{\left( \inv(T) + (c,a) \right)}$ a \defn{$s$-Tamari rotation} of $T$.
\end{definition}

See Figure~\ref{fig:stam-cover} for examples of $s$-Tamari rotations.

\begin{figure}[ht]
\input{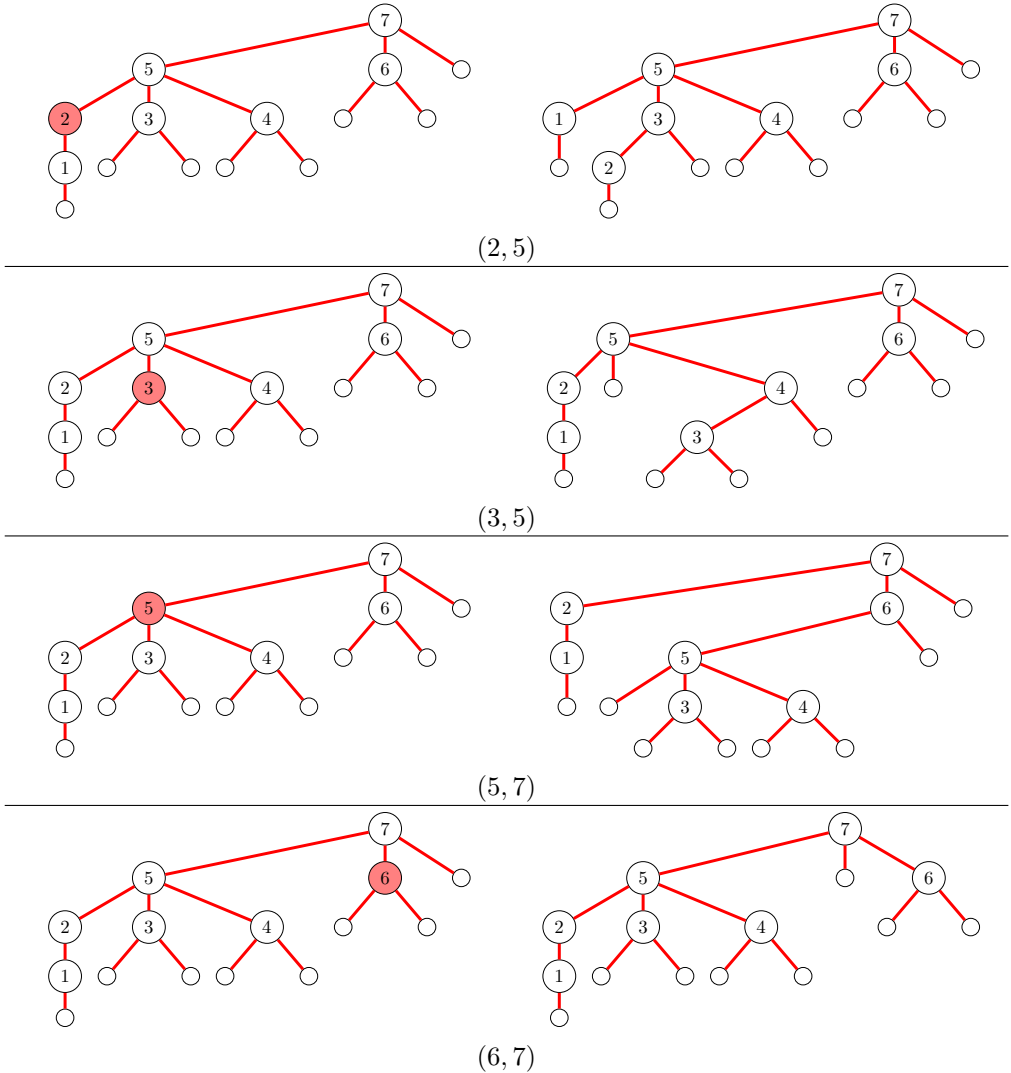}
\caption{Example of $s$-Tamari rotations}
\label{fig:stam-cover}
\end{figure}

\begin{theorem}\label{thm_sTamari_covering_relations}
The cover relations of the $s$-Tamari lattice are in correspondence with $s$-Tamari rotations.
\end{theorem}

We need to show that an $s$-Tamari tree $T'$ covers $T$ in $s$-Tamari order if and only if there is a Tamari ascent $(a,c)$ of $T$ such that $\inv(T')= \tc{\left( \inv(T) + (c,a) \right)}$. 
The proof of this fact will follow from the following lemma. 

\begin{lemma}\label{lem_Tamari_ascent_exist}
Let $T$ and $T'$ be two $s$-Tamari trees for some weak composition $s$, with $T\wole T'$.
Then, there exists a $s$-Tamari ascent $(a,c)$ of $T$ such that $\card_{T'}(c,a)> \card_{T}(c,a)$.
\end{lemma}

\begin{proof}
Let $V = \lbrace (c,a); \card_T(c,a) < \card_{T'}(c,a) \rbrace$, \emph{i.e.}, the tree-inversions of $T$ which cardinality have increased in $T'$. As $T \woless T'$, we know that $V \neq \emptyset$. 

As proved in Part~\ref{part1_lem:cover-relations-recip} of the proof of Lemma~\ref{lem:cover-relations-recip}, 
there is a $(c,a)\in V$ such that $a$ is a descendant of~$c$ in $T$. 
Let $a'$ be the child of $c$ that belongs to the unique path from $c$ to $a$ in the tree $T$. 
Note that~$(c,a')\in V$, otherwise
\[
\card_{T'}(c,a') = \card_T(c,a') = \card_T(c,a) < \card_{T'}(c,a),
\] 
which contradicts the fact that $T'$ is an $s$-Tamari tree. 

Since $\card_{T'}(c,a')>\card_T(c,a')$ then $a'$ is a non right child of $T$. 
Therefore, $(a'c)\in V$ is an $s$-Tamari ascent of $T$. 
\end{proof}

\begin{proof}[Proof of Theorem~\ref{thm_sTamari_covering_relations}]
We will show that $T'$ covers $T$ in $s$-Tamari order if and only if there is a Tamari ascent $(a,c)$ of $T$ such that $\inv(T')= \tc{\left( \inv(T) + (c,a) \right)}$. 

Let $T$ and $T'$ be two $s$-Tamari trees such that $T'$ covers $T$ in $s$-Tamari order. 
In particular, $T \wole T'$ and Lemma~\ref{lem_Tamari_ascent_exist} guaranties that there is a Tamari ascent $(a,c)$ of $T$ such that $\card_{T'}(c,a)>\card_T(c,a)$.
In particular, 
$\inv(T)+(c,a) \subseteq \inv(T')$
and by transitivity 
$\tc{\left( \inv(T) + (c,a) \right)} \subseteq \inv(T')$.
By Lemma~\ref{lem_sTamari_rotation}, we know that $\tc{\left( \inv(T) + (c,a) \right)}$ is an $s$-Tamari inversion set; Let $T''$ be its corresponding $s$-Tamari tree. 
We have that $T\woless T'' \wole T'$. By hyphothesis, $T'$ covers $T$, and so $T''=T'$. Therefore, $\inv(T')= \inv(T'')=\tc{\left( \inv(T) + (c,a) \right)}$ as desired.     

To prove the other direction, let $(a,c)$ be a Tamari ascent of an $s$-Tamari tree $T$, and let $T'$ be the rotated $s$-Tamari tree determined by $\inv(T')=\tc{\left( \inv(T) + (c,a) \right)}$. 
We need to prove that $T'$ covers $T$ in $s$-Tamari order. 
Suppose that there is an $s$-Tamari tree $T''$ such that $T \woless T'' \wole T'$.
By Lemma~\ref{lem_Tamari_ascent_exist}, there is a Tamari-ascent $(a',c')$ of $T$ such that $\card_{T}(c',a') < \card_{T''}(c',a') \leq \card_{T'}(c',a')$.  
From Lemma~\ref{lem_sTamari_rotation} we deduce that $c'=c$ and $a'=a$. 
This implies that $\inv(T) + (c,a) \subseteq \inv(T'')$.  
By transitivity, it follows that $\inv(T') \subseteq \inv(T'')$. Therefore $T' = T''$, which implies that $T'$ covers $T$.\end{proof}

\subsection{The $\nu$-Tamari lattice}
The main purpose of this section is to show that the $s$-Tamari lattice is closely related to the $\nu$-Tamari lattice introduced by Pr\'eville-Ratelle and Viennot 
in~\cite{PrevilleRatelleViennot}. The elements of the $\nu$-Tamari lattice are given by lattice paths that lie weakly above a fixed lattice path~$\nu$, and the covering relation can be explicitly described from such paths. 
An alternative description in terms of certain combinatorial objects called $\nu$-trees was presented in~\cite{ceballos_vtamarivtrees_2018}. 
We briefly recall these objects and refer to~\cite{ceballos_vtamarivtrees_2018} for more details. 

\subsubsection{The $\nu$-Tamari lattice in terms of $\nu$-paths}

Let $\nu$ be a lattice path on the plane consisting of a finite number of north and east unit steps, which are represented by the letters $N$ and $E$. For instance, the bottom path on the right of Figure~\ref{fig:bij-stam-nupaths} is $\nu=NEENEEN$.

\begin{definition}
A \defn{$\nu$-path} $\mu$ is a lattice path using north and east steps, that has the same endpoints as $\nu$ and is weakly above $\nu$. 
\end{definition}

The set of $\nu$-paths has a natural order structure that we now define.
For a lattice point $ p $ on a $\nu$-path~$\mu$, the \defn{horizontal distance} $ \horiz_\nu(p) $ is defined as the maximum number of horizontal steps that can be added to the right of $ p $ without crossing $\nu$.
We say that $p$ is a \defn{valley} of $\mu$ if it is preceded by an east step $E$ and followed by a north step $N$. 
Given a valley $p$ of $\mu$, consider the next lattice point $q$ of $\mu$ such that $\horiz_\nu(q) =  \horiz_\nu(p)$. Let $ \mu_{[p,q]} $ be the subpath of $\mu$ that starts at $p$ and ends at $q$, and let $\mu'$ be the path obtained from $\mu$ by switching the $E$ preceding $p$ and~$\mu_{[p,q]}$.
The path $\mu'$ is called the \defn{$\nu$-rotation} of $\mu$ at the valley $ p $, and we write $\mu \lessdot_\nu \mu'$. 

\begin{definition}
The \defn{$\nu$-Tamari lattice} is the poset on $\nu$-paths whose cover relations are $\nu$-rotations. 
\end{definition}

In~\cite{PrevilleRatelleViennot}, Pr\'eville-Ratelle and Viennot showed that this poset is a lattice.
The right hand side of Figure~\ref{fig:bij-stam-nupaths} shows an example of the $\nu$-Tamari lattice, for $\nu=NEENEEN$

\subsubsection{The $\nu$-Tamari lattice in terms of $\nu$-trees}\label{sec_nu_trees}
The $\nu$-Tamari lattice can be alternatively described in terms of $\nu$-trees~\cite{ceballos_vtamarivtrees_2018}.
We denote by $A_\nu$ the set of lattice points weakly above $\nu$ inside the smallest rectangle containing $\nu$. We say that two points $p,q \in A_\nu$ are \defn{$\nu$-incompatible} if and only if $p$ is southwest or northeast to $q$ and the south-east corner of the smallest rectangle containing $p$ and $q$ is weakly above~$\nu$. Otherwise, we say that $p$ and $q$ are \defn{$\nu$-compatible}.

\begin{definition}
A $\nu$-tree is a maximal collection of pairwise $\nu$-compatible elements in $A_\nu$.
\end{definition}

Although $\nu$-trees are just sets of lattice points above $\nu$, they can be regarded as planar binary trees in the classical sense, by connecting consecutive nodes (or lattice points) in the same row or column~\cite{ceballos_vtamarivtrees_2018}. 
The result is a planar binary tree with a root a the top left corner of~$A_\nu$.

A node $q\in T$ that lies below another node $p\in \vT$ and on the left of another node $r\in \vT$ is called a \defn{$\nu$-ascent} of $T$. The node $q$ may be exchanged by a node $q'$ located exactly on the right of $p$ and above~$r$, creating a new $\nu$-Tamari tree $\vT'$. This operation of exchanging $q$ by $q'$ is called a \defn{$\nu$-tree rotation}.

\begin{definition}
The \defn{rotation lattice of $\nu$-trees} is the poset on $\nu$-trees whose cover relations are given by $\nu$-tree rotations.  
\end{definition}

The right hand side of Figure~\ref{fig:bij-stam-nutrees} shows an example of the rotation lattice of $\nu$-trees, for $\nu=NEENEEN$. One can see that it is isomorphic to the $\nu$-Tamari lattice illustrated in Figure~\ref{fig:bij-stam-nupaths}. 

\begin{theorem}[{\cite{ceballos_vtamarivtrees_2018}}]
The $\nu$-Tamari lattice is isomorphic to the rotation lattice of $\nu$-trees.
\end{theorem}
 
A bijection between these two lattices is given by the \defn{right flushing bijection} from~\cite{ceballos_vtamarivtrees_2018}. This bijection maps a $\nu$-path $\mu$ to the unique $\nu$-tree $T$ such that $\mu$ and $T$ have the same number of lattice points at each height. 
The $\nu$-tree $T$ can be constructed recursively as follows.
Denote by $h_i$ the number of lattice points of $\mu$ at height $i$. Construct $T$ by adding $h_i$ nodes at height~$i$ from bottom to top, from right to left, avoiding forbidden positions. The forbidden positions are those above a node that is not the left most node in a row (these come from the initial points of the east steps in the path $\mu$). 
We refer to~\cite{ceballos_vtamarivtrees_2018} for details on why this works.
In particular, note that the number of lattice points in every $\nu$-path is equal to the number of nodes in every $\nu$-tree (equals to the number of steps of $\nu$ plus one).

\subsection{The $s$-Tamari lattice and the $\nu$-Tamari lattice are isomorphic}
\label{sec_sTamarivTamari}

Now that we have all the required definitions, we are ready to show that the $s$-Tamari lattice introduced in this paper is isomorphic to a well chosen $\nu$-Tamari lattice. 

\subsubsection{Isomorphism with the $\nu$-Tamari lattice}
Let $s=(s(1),\dots,s(n))$ be a weak composition and~$\nu(s)$ be the lattice path $\nu(s):=NE^{s(1)}\dots NE^{s(n)}$. 
We denote by $\reverse s$ the reversed sequence $(s(n),\dots,s(2),s(1))$. 
We now present a bijection between $s$-Tamari trees and $\nu(\reverse s)$-paths which determines an isomorphism of their corresponding lattices. For simplicity, we fix $\nu=\nu(\reverse s)$ throughout this section. As a warm up, note that the number of nodes in every $s$-tree is equal to the number of nodes in every $\nu$-tree for this choice of $\nu$.

Let $T$ be an $s$-Tamari tree. We say that its nodes (internal nodes and leaves) are traversed in \emph{reversed preorder} if we visit the root $r$ first and then visit in reversed preorder each of the trees $T^r_i$ such that~$T^r_i$ is visited before~$T^r_j$ for $i>j$.  
We label the internal nodes of $T$ with $N$ and the leaves with $E$, see Figure~\ref{fig:bij-stam-nupaths}. Then, we read the labels using the reverse preorder transversal of $T$, and omit the last appearance of $E$. 
We denote by $\streestovpaths(T)$ te resulting path.   

\begin{figure}[htbp]
\begin{center}
\begin{tabular}{cc}
\includegraphics[height= 8.5cm]{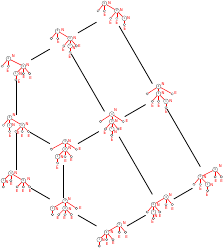}
&
\includegraphics[height= 8.5cm]{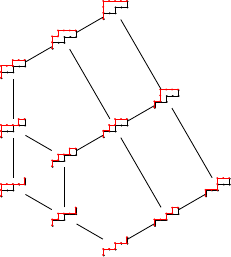}
\end{tabular}
\caption{Bijection between the $s$-Tamari lattice and the $\nu(\protect\reverse{s})$-Tamari lattice for $s=(0,2,2)$.}
\label{fig:bij-stam-nupaths}
\end{center}
\end{figure}

For instance, if $T_0$ is the bottom $s$-Tamari tree on the left hand side of Figure~\ref{fig:bij-stam-nupaths}, then we have that~$\streestovpaths(T_0)=NEENEEN$. If $T_1$ is the top $s$-Tamari tree in the same figure, then $\streestovpaths(T_1)=NNNEEEE$. These two paths are the bottom and top elements of the $\nu$-Tamari lattice for $\nu=\nu(\reverse s)$. In fact, we have the following theorem.

\begin{theorem}\label{thm_sTam_vTam_paths}
The map $\streestovpaths$ is an isomorphism between the $s$-Tamari lattice and the $\nu(\reverse s)$-Tamari lattice.
\end{theorem}

\begin{remark}
The proof of this theorem follows essentially from the work on signature Catalan combinatorics in~\cite{ceballos_signature_2018}.  In that paper, the authors use the preorder transversal on trees in order to describe many bijections in the zoo of signature Catalan objects indexed by a signature $s=(s(1),\dots , s(n))$. 

A main family in this zoo is the family of what they call ``\emph{$s$-trees}".  
These are planar rooted trees $T$ with ``signature" $s$, where the signature is defined as 
\[
\operatorname{signat} (T)=
(\operatorname{deg}(v_1),\dots , \operatorname{deg}(v_n)),
\]
where $v_1,\dots,v_n$ are the internal nodes of $T$ read in preorder, and $\operatorname{deg}(v)$ is the number of children of $v$. 
Note that here there is a small difference with our paper, because our signature counts the number of children of the internal nodes minus one. 

Our $s$-Tamari trees are exactly the mirror images (flip horizontally) of the $(\reverse{s+\bf{1}})$-trees in~\cite{ceballos_signature_2018}.\footnote{Here $(s+\bf{1})$ means that we add one to each of the entries of $s$.} 
Indeed, if we read the number of children of the internal nodes of an $s$-Tamari tree in reversed preorder, then we get the sequence $(s(n)+1,\dots , s(2)+1,s(1)+1)$. So, taking its mirror image we get a $(\reverse{s+\bf{1}})$-tree. 
Viceversa, if we start with the mirror image of an $(\reverse{s+\bf{1}})$-tree and label the internal nodes with the numbers $n,\dots , 2,1$ in reversed preorder, we get an $s$-decreasing tree that is an~$s$-Tamari tree. 

The bijection $\streestovpaths$ in Theorem~\ref{thm_sTam_vTam_paths} is related to the bijection between $s$-trees and $s$-Dyck paths in~\cite{ceballos_signature_2018}, and our Theorem will be implied by some of their observations. 
\end{remark}

\begin{proof}[Proof of Theorem~\ref{thm_sTam_vTam_paths}]
For simplicity, let $\nu=\nu(\reverse s)$.
Let $T$ be an $s$-Tamari tree, and let $a,c$ be two nodes of $T$ such that $a$ is a child of $c$.
We denote by $\streestovpaths(a),\streestovpaths(c)$ the corresponding lattice points in the path $\streestovpaths(T)$.
More precisely, the root of $T$ corresponds to the initial lattice point of $\streestovpaths(T)$; in general, the nodes of $T$, read in reverse preorder, correspond one-to-one to the lattice points of $\streestovpaths(T)$ when read from the start to the end. Note that we are abusing of notation here for simplicity, because $\streestovpaths(a)$ depends also on $T$ and not only on $a$.

By \cite[Lemma~3.3]{ceballos_signature_2018}, it follows that 
\begin{align}\label{eq_horiz_increase}
\horiz_\nu(\streestovpaths(a)) = \horiz_\nu(\streestovpaths(c)) + \card_T (c,a).
\end{align}
In other words, the $i$-th child of $c$ increases the horizontal distance by $i$. The reader is invited to check this property in Figure~\ref{fig:bij-stam-nupaths}.

This implies, in particular, that the horizontal distance of each lattice point in the path $\streestovpaths(T)$ is non-negative. This shows that $\streestovpaths(T)$ is a $\nu$-path (a path weakly above $\nu$).
Moreover, $\streestovpaths$ is a bijection between the set of $s$-Tamari trees and the set of $\nu$-paths, as shown in~\cite[Theorem~3.6]{ceballos_signature_2018}.
Our main contribution is to show that the cover relations in the $s$-Tamari lattice are transformed to cover relations in the $\nu$-Tamari lattice. This was not shown in~\cite{ceballos_signature_2018}. 

Assume that $(a,c)$ is a Tamari ascent of $T$ (that is, $a$ is a non-right child of $c$). And let $T'$ be the $s$-Tamari tree obtained by applying the $s$-Tamari rotation at $(a,c)$ in $T$. In other words, $T \wole T'$ is a cover relation of the $s$-Tamari lattice. 
We want to show that $\streestovpaths(T')$ covers $\streestovpaths(T)$ in the $\nu$-Tamari lattice.

First, observe that $\streestovpaths(a)$ is a valley of $\streestovpaths(T)$. 
Indeed, $a$ is an internal node of $T$ and so $\streestovpaths(a)$ is followed by a north step, and since $a$ is a non-right child of $c$ then $\streestovpaths(a)$ is necessarily preceded by an east step. 
The next lattice point in $\streestovpaths(T)$ that has the same horizontal distance as $\streestovpaths(a)$ is $\streestovpaths(a')$, where $a'$ is the left child of $a$ in $T$. 
The path between $\streestovpaths(a)$ is $\streestovpaths(a')$ in $\streestovpaths(T)$ consists of the lattice points $\streestovpaths(x)$ for $x\in A$, where $A$ is the set consisting of $a$ and all its descendants that are not in its left leg.
 Performing a $\nu$-rotation at the valley $\streestovpaths(a)$ of $\streestovpaths(T)$ corresponds exactly to performing the $s$-Tamari rotation of $T$ at the Tamari-ascent~$(a,c)$.
 
We have shown that $s$-Tamari rotations are transformed to $\nu$-Tamari rotations under the bijection $\streestovpaths$. It remains to show that every $\nu$-Tamari rotation arises this way. 
For this, note that if $\streestovpaths(a)$ is a valley then it is followed by a north step, which implies that $a$ is an internal node of $T$. Since $\streestovpaths(a)$ is preceded by an east step, then necessarily $a$ is the non-right child of some node $c$. Using the same argument as above, the $\nu$-rotation at the valley $\streestovpaths(a)$ in $\streestovpaths(T)$ corresponds to the $s$-Tamari rotation of $T$ at the Tamari-ascent~$(a,c)$.
\end{proof}

\subsubsection{Isomorphism with the rotation lattice of $\nu$-trees}
Composing the bijection $\streestovpaths$ with the right-flushing bijection described in Section~\ref{sec_nu_trees}, we obtain a bijection $\streestovtrees$ from the set of $s$-Tamari trees to the set of~$\nu(\reverse s)$-trees,  which gives an isomorphism of the corresponding lattices. The image $\streestovtrees(T)$ of an $s$-Tamari tree $T$ can be described purely in terms of $T$ as follows. 

Let $T$ be an $s$-Tamari tree. We use the reverse preorder transversal to label the nodes of $T$ with the numbers $0,1,\dots ,n$, such that a node $x$ has label $i$ if the number of internal nodes traversed striktly before $x$ is equal to $i$. We denote by $\streestovtrees(T)$ the unique $\nu$-tree containing as many nodes at height~$i$ as there are nodes in $T$ with label $i$. Such $\nu$-tree can be uniquely constructed using the right flushing algorithm described in Section~\ref{sec_nu_trees}. 
See Figure~\ref{fig:bij-stam-nutrees} for an illustration.

\begin{corollary}\label{thm_sTam_vTam_trees}
The map $\streestovtrees$ is an isomorphism between the $s$-Tamari lattice and the rotation lattice of~$\nu(\reverse s)$-trees.
\end{corollary}

\begin{proof}
For simplicity, let $\nu=\nu(\reverse s)$.
The number of lattice points at height $i$ in the $\nu$-path $\streestovpaths(T)$ is equal to the number of nodes labeled~$i$ in $T$ (because every north step/internal node of $T$ increases the height by one). 
Therefore, $\streestovtrees(T)$ is the right flushing bijection applied to the $\nu$-path~$\streestovpaths(T)$.
The result follows.
\end{proof}

\begin{figure}[htbp]
\begin{center}
\begin{tabular}{cc}
\includegraphics[height= 8.5cm]{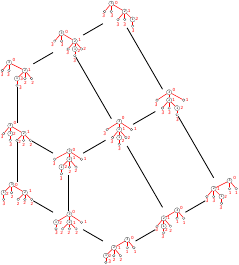}
&
\includegraphics[height= 8.5cm]{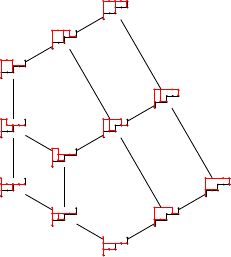}
\end{tabular}
\caption{Bijection between the $s$-Tamari lattice and the rotation lattice of~\mbox{$\nu(\protect\reverse{s})$-trees} for $s=(0,2,2)$.}
\label{fig:bij-stam-nutrees}
\end{center}
\end{figure}

\subsection{The $s$-Tamari lattice is polygonal, semidistributive and congruence uniform}

In contrast to the lattice properties that we proved for the $s$-weak order, it is straight forward to show that the $s$-Tamari lattice is polygonal, semidistributive and congruence uniform. 

For the congruence uniform property, recall that 
if a lattice is congruence uniform then every sublattice of it is congruence uniform as well~\cite[Theorem (4.3)]{Day79}.  
Since the $s$-Tamari lattice is a sublattice of the $s$-weak order, which is congruence uniform (Theorem~\ref{thm:hh}), it follows that the $s$-Tamari lattice is congruence uniform, hence semidistributive. 
Furthermore, the congruence uniform property of the $s$-Tamari lattice also follows directly from their equivalence with the $\nu$-Tamari lattice (Theorem~\ref{thm_sTam_vTam_paths}). In fact, as proved in~\cite{PrevilleRatelleViennot}, every $\nu$-Tamari lattice is an interval of a larger Tamari lattice (in particular, a sublattice). Since Tamari lattices are known to be congruence uniform, then $\nu$-Tamari lattices are congruence uniform as well.

The polygonality property of the $s$-Tamari lattice also follows from its connection to the $\nu$-Tamari lattice. Indeed, if a lattice is polygonal then all of its intervals are polygonal as well (since the conditions for polygonality are conditions about intervals). Therefore, since the Tamari lattice is known to be polygonal, then the $\nu$-Tamari lattice (and thus the $s$-Tamari lattice) is polygonal as well.


Furthermore, all these properties can be proven directly by showing that the $s$-Tamari lattice is an $\HH$-lattice using similar techniques as in Section~\ref{subsec:cong-uniform}. As far as we know, the $\HH$-lattice property was not even known for the classical Tamari lattice.

\begin{theorem}
The $s$-Tamari lattice is an $\HH$ lattice. Hence, it is polygonal, semidistributive and congruence uniform.
\end{theorem}

\begin{proof}
We omit the full proof as it is very similar to the proof of Theorem~\ref{thm:hh}, and also relies on the results of~\cite{SWeakSB}. Basically, it is shown in~\cite{SWeakSB} that using $s$-Tamari ascents instead of tree-ascents, the $s$-Tamari lattice has two possible polygons: a square or a pentagon corresponding respectively to cases~\eqref{cond:polygons-square} and~\eqref{cond:polygons-pentagon-right} of Proposition~\ref{prop:polygons}. Using a similar labeling as in the $s$-weak order case (labels are given by Tamari ascents instead of tree-ascents), we obtain the $\HH$ property directly.
\end{proof}

\section*{Acknowledgements}

All computations and computer exploration needed for the paper have been done using SageMath~\cite{SageMath2022}. We would also like to thank Nathan Reading and Henry M\"uhle for answering our questions and providing references concerning congruence uniformity.

\bibliographystyle{alphaurl}
\bibliography{../draft}

\begin{thebibliography}{BMCPR13}

\bibitem[AS06]{AguiarSottile_StructureLodayRonco}
Marcelo Aguiar and Frank Sottile.
\newblock Structure of the {L}oday-{R}onco {H}opf algebra of trees.
\newblock {\em J. Algebra}, 295(2):473--511, 2006.
\newblock \href {https://doi.org/10.1016/j.jalgebra.2005.06.021}
  {\path{doi:10.1016/j.jalgebra.2005.06.021}}.

\bibitem[BCP22]{HopfDreams_2022}
Nantel Bergeron, Cesar Ceballos, and Vincent Pilaud.
\newblock Hopf dreams and diagonal harmonics.
\newblock {\em J. Lond. Math. Soc. (2)}, 105(3):1546--1600, 2022.
\newblock \href {https://doi.org/10.1112/jlms.12541}
  {\path{doi:10.1112/jlms.12541}}.

\bibitem[BMCPR13]{bousquet_representation_2013}
Mireille Bousquet-M\'{e}lou, Guillaume Chapuy, and Louis-Fran\c{c}ois
  Pr\'{e}ville-Ratelle.
\newblock The representation of the symmetric group on {$m$}-{T}amari
  intervals.
\newblock {\em Adv. Math.}, 247:309--342, 2013.
\newblock \href {https://doi.org/10.1016/j.aim.2013.07.014}
  {\path{doi:10.1016/j.aim.2013.07.014}}.

\bibitem[BMFPR11]{bousquet_number_2011}
Mireille Bousquet-M\'{e}lou, \'{E}ric Fusy, and Louis-Fran\c{c}ois
  Pr\'{e}ville-Ratelle.
\newblock The number of intervals in the {$m$}-{T}amari lattices.
\newblock {\em Electron. J. Combin.}, 18(2):Paper 31, 26, 2011.
\newblock \href {https://doi.org/10.37236/2027} {\path{doi:10.37236/2027}}.

\bibitem[BP12]{BPR12}
F.~Bergeron and L.-F. {Pr{\'e}ville-Ratelle}.
\newblock Higher trivariate diagonal harmonics via generalized {{Tamari}}
  posets.
\newblock {\em J. Comb.}, 3(3):317--341, 2012.
\newblock \href {https://doi.org/10.4310/JOC.2012.v3.n3.a4}
  {\path{doi:10.4310/JOC.2012.v3.n3.a4}}.

\bibitem[BW97]{bjorner_wachs_shellableII_1997}
Anders Bj\"{o}rner and Michelle~L. Wachs.
\newblock Shellable nonpure complexes and posets. {II}.
\newblock {\em Trans. Amer. Math. Soc.}, 349(10):3945--3975, 1997.
\newblock \href {https://doi.org/10.1090/S0002-9947-97-01838-2}
  {\path{doi:10.1090/S0002-9947-97-01838-2}}.

\bibitem[Cas00]{caspard_latticepermutations_2000}
Nathalie Caspard.
\newblock The lattice of permutations is bounded.
\newblock {\em Internat. J. Algebra Comput.}, 10(4):481--489, 2000.
\newblock \href {https://doi.org/10.1142/S0218196700000182}
  {\path{doi:10.1142/S0218196700000182}}.

\bibitem[CdPBM04]{CasparPBM2004}
Nathalie Caspard, Claude Le~Conte de~Poly-Barbut, and Michel Morvan.
\newblock Cayley lattices of finite coxeter groups are bounded.
\newblock {\em Advances in Applied Mathematics}, 33(1):71--94, 2004.
\newblock \href {https://doi.org/https://doi.org/10.1016/j.aam.2003.09.002}
  {\path{doi:https://doi.org/10.1016/j.aam.2003.09.002}}.

\bibitem[CGD19]{ceballos_signature_2018}
Cesar Ceballos and Rafael~S. Gonz\'alez~D'Le\'on.
\newblock Signature {Catalan} combinatorics.
\newblock {\em Journal of Combinatorics}, 10(4):725--773, 2019.
\newblock \href {https://doi.org/https://doi.org/10.4310/JOC.2019.v10.n4.a6}
  {\path{doi:https://doi.org/10.4310/JOC.2019.v10.n4.a6}}.

\bibitem[Cha07]{chapoton_tamari_intervals_2005}
F.~Chapoton.
\newblock Sur le nombre d'intervalles dans les treillis de {T}amari.
\newblock {\em S\'{e}m. Lothar. Combin.}, 55:Art. B55f, 18, 2005/07.

\bibitem[CP19]{FPSAC2019}
C.~Ceballos and V.~Pons.
\newblock {The s-weak order and s-permutahedra}.
\newblock In {\em {31th International Conference on Formal Power Series and
  Algebraic Combinatorics (FPSAC 2019)}}, volume 82B, page Art. 76.
  {S\'eminaire Lotharingien de Combinatoire}, 2019.

\bibitem[CPS20]{ceballos_vtamarivtrees_2018}
Cesar Ceballos, Arnau Padrol, and Camilo Sarmiento.
\newblock The {$\nu$}-{T}amari lattice via {$\nu$}-trees, {$\nu$}-bracket
  vectors, and subword complexes.
\newblock {\em Electron. J. Combin.}, 27(1):Paper No. 1.14, 31, 2020.
\newblock \href {https://doi.org/https://doi.org/10.37236/8000}
  {\path{doi:https://doi.org/10.37236/8000}}.

\bibitem[DA94]{DuquenneCherfouh94}
V.~Duquenne and Cherfouh A.
\newblock On permutation lattices.
\newblock {\em Mathematical Social Sciences}, 27(1):73 -- 89, 1994.
\newblock \href
  {https://doi.org/http://dx.doi.org/10.1016/0165-4896(94)00733-0}
  {\path{doi:http://dx.doi.org/10.1016/0165-4896(94)00733-0}}.

\bibitem[Day77]{day_splitting_1977}
Alan Day.
\newblock Splitting lattices generate all lattices.
\newblock {\em Algebra Universalis}, 7(2):163--169, 1977.
\newblock \href {https://doi.org/10.1007/BF02485425}
  {\path{doi:10.1007/BF02485425}}.

\bibitem[Day79]{Day79}
Alan Day.
\newblock Characterizations of finite lattices that are bounded-homomorphic
  images or sublattices of free lattices.
\newblock {\em Canadian Journal of Mathematics}, 31(1):69--78, 1979.
\newblock \href {https://doi.org/10.4153/CJM-1979-008-x}
  {\path{doi:10.4153/CJM-1979-008-x}}.

\bibitem[FPR17]{fang_enumeration_2017}
Wenjie Fang and Louis-Fran\c{c}ois Pr\'{e}ville-Ratelle.
\newblock The enumeration of generalized {T}amari intervals.
\newblock {\em European J. Combin.}, 61:69--84, 2017.
\newblock \href {https://doi.org/10.1016/j.ejc.2016.10.003}
  {\path{doi:10.1016/j.ejc.2016.10.003}}.

\bibitem[FZ02]{FominZelevinsky-ClusterAlgebrasI}
Sergey Fomin and Andrei Zelevinsky.
\newblock Cluster algebras. {I}. {F}oundations.
\newblock {\em J. Amer. Math. Soc.}, 15(2):497--529 (electronic), 2002.

\bibitem[FZ03]{FominZelevinsky-ClusterAlgebrasII}
Sergey Fomin and Andrei Zelevinsky.
\newblock Cluster algebras. {II}. {F}inite type classification.
\newblock {\em Invent. Math.}, 154(1):63--121, 2003.

\bibitem[GH93]{garsia_graded_1993}
Adriano~M. Garsia and Mark Haiman.
\newblock A graded representation model for {Macdonald}'s polynomials.
\newblock {\em Proceedings of the National Academy of Sciences of the United
  States of America}, 90(8):3607--3610, 1993.
\newblock URL: \url{http://www.ams.org/mathscinet-getitem?mr=1214091}, \href
  {https://doi.org/10.1073/pnas.90.8.3607} {\path{doi:10.1073/pnas.90.8.3607}}.

\bibitem[GH96]{garsia_natural_1996}
Adriano~M. Garsia and Mark Haiman.
\newblock Some natural bigraded {$S\_n$}-{Modules}.
\newblock {\em The Electronic Journal of Combinatorics}, 3(2):R24, January
  1996.
\newblock URL:
  \url{http://www.combinatorics.org/ojs/index.php/eljc/article/view/v3i2r24}.

\bibitem[Hai94]{haiman_conjectures_1994}
Mark~D. Haiman.
\newblock Conjectures on the quotient ring by diagonal invariants.
\newblock {\em J. Algebraic Combin.}, 3(1):17--76, 1994.
\newblock \href {https://doi.org/10.1023/A:1022450120589}
  {\path{doi:10.1023/A:1022450120589}}.

\bibitem[Hai01]{haiman_hilbert_2001}
Mark Haiman.
\newblock Hilbert schemes, polygraphs and the {Macdonald} positivity
  conjecture.
\newblock {\em Journal of the American Mathematical Society}, 14(4):941--1006,
  2001.
\newblock URL: \url{http://www.ams.org/jams/2001-14-04/S0894-0347-01-00373-3/},
  \href {https://doi.org/10.1090/S0894-0347-01-00373-3}
  {\path{doi:10.1090/S0894-0347-01-00373-3}}.

\bibitem[HL07]{HL07}
Christophe Hohlweg and Carsten E. M.~C. Lange.
\newblock Realizations of the associahedron and cyclohedron.
\newblock {\em Discrete Comput. Geom.}, 37(4):517--543, 2007.
\newblock \href {https://doi.org/10.1007/s00454-007-1319-6}
  {\path{doi:10.1007/s00454-007-1319-6}}.

\bibitem[HLT11]{HLT11}
Christophe Hohlweg, Carsten E. M.~C. Lange, and Hugh Thomas.
\newblock Permutahedra and generalized associahedra.
\newblock {\em Advances in Mathematics}, 226(1):608--640, January 2011.
\newblock \href {https://doi.org/10.1016/j.aim.2010.07.005}
  {\path{doi:10.1016/j.aim.2010.07.005}}.

\bibitem[HNT05]{HNT05}
F.~Hivert, J.-C. Novelli, and J.-Y. Thibon.
\newblock The algebra of binary search trees.
\newblock {\em Theoret. Comput. Sci.}, 339(1):129--165, 2005.
\newblock \href {https://doi.org/10.1016/j.tcs.2005.01.012}
  {\path{doi:10.1016/j.tcs.2005.01.012}}.

\bibitem[Lac22]{SWeakSB}
S.~Lacina.
\newblock {Poset Topology of s-Weak Order via SB-Labelings}.
\newblock {\em Journal of Combinatorics}, 13(3):357 -- 395, 2022.
\newblock \href {https://doi.org/10.4310/JOC.2022.v13.n3.a3}
  {\path{doi:10.4310/JOC.2022.v13.n3.a3}}.

\bibitem[LCdPB94]{PolyBarbut94}
C.~Le~Conte~de Poly-Barbut.
\newblock Sur les treillis de coxeter finis.
\newblock {\em Math\'ematiques et Sciences humaines}, 125:41--57, 1994.

\bibitem[LR98]{LodayRonco98_binarytrees}
Jean-Louis Loday and Mar\'{\i}a~O. Ronco.
\newblock Hopf algebra of the planar binary trees.
\newblock {\em Adv. Math.}, 139(2):293--309, 1998.
\newblock \href {https://doi.org/10.1006/aima.1998.1759}
  {\path{doi:10.1006/aima.1998.1759}}.

\bibitem[LR02]{LodayRonco02_permutations}
Jean-Louis Loday and Mar\'{\i}a~O. Ronco.
\newblock Order structure on the algebra of permutations and of planar binary
  trees.
\newblock {\em J. Algebraic Combin.}, 15(3):253--270, 2002.
\newblock \href {https://doi.org/10.1023/A:1015064508594}
  {\path{doi:10.1023/A:1015064508594}}.

\bibitem[MHPS12]{hoissen12associahedra}
Folkert M\"{u}ller-Hoissen, Jean~Marcel Pallo, and Jim Stasheff, editors.
\newblock {\em Associahedra, {T}amari lattices and related structures}, volume
  299 of {\em Progress in Mathematics}.
\newblock Birkh\"{a}user/Springer, Basel, 2012.
\newblock Tamari memorial Festschrift.
\newblock \href {https://doi.org/10.1007/978-3-0348-0405-9}
  {\path{doi:10.1007/978-3-0348-0405-9}}.

\bibitem[NT20]{novelli_hopf_2014}
Jean-Christophe Novelli and Jean-Yves Thibon.
\newblock Hopf algebras of m-permutations, (m+1)-ary trees, and m-parking
  functions.
\newblock {\em Advances in Applied Mathematics}, 117:102019, June 2020.
\newblock \href {https://doi.org/10.1016/j.aam.2020.102019}
  {\path{doi:10.1016/j.aam.2020.102019}}.

\bibitem[Pon15]{pons_lattice_2015}
Viviane Pons.
\newblock A lattice on decreasing trees : the metasylvester lattice.
\newblock In {\em 27th {International} {Conference} on {Formal} {Power}
  {Series} and {Algebraic} {Combinatorics} ({FPSAC} 2015)}, Discrete {Math}.
  {Theor}. {Comput}. {Sci}. {Proc}., {AS}, pages 381--392. Assoc. Discrete
  Math. Theor. Comput. Sci., Nancy, 2015.

\bibitem[Pon19]{BlogViv}
Viviane Pons.
\newblock On giving my talk, my first baby battle, 2019.
\newblock {Blog} {Post}.
\newblock URL: \url{http://openpyviv.com/2019/07/05/talk/}.

\bibitem[Pon22]{SageDemoI}
V.~Pons.
\newblock Sagemath code and demo for s-weak order and s-permutahedra, 2022.
\newblock \href {https://doi.org/10.5281/zenodo.7459784}
  {\path{doi:10.5281/zenodo.7459784}}.

\bibitem[PP18]{PP18}
V.~Pilaud and V.~Pons.
\newblock Permutrees.
\newblock {\em Algebraic Combinatorics}, 1(2):173--224, 2018.
\newblock \href {https://doi.org/10.5802/alco.1} {\path{doi:10.5802/alco.1}}.

\bibitem[PRV17]{PrevilleRatelleViennot}
Louis-Fran\c{c}ois Pr\'{e}ville-Ratelle and Xavier Viennot.
\newblock The enumeration of generalized {T}amari intervals.
\newblock {\em Trans. Amer. Math. Soc.}, 369(7):5219--5239, 2017.
\newblock \href {https://doi.org/10.1090/tran/7004}
  {\path{doi:10.1090/tran/7004}}.

\bibitem[PS19]{PS19}
Vincent Pilaud and Francisco Santos.
\newblock Quotientopes.
\newblock {\em Bulletin of the London Mathematical Society}, 51(3):406--420,
  2019.
\newblock \href {https://doi.org/10.1112/blms.12231}
  {\path{doi:10.1112/blms.12231}}.

\bibitem[PSZ23]{PSZ23}
Vincent Pilaud, Francisco Santos, and G{\"u}nter~M. Ziegler.
\newblock Celebrating {{Loday}}'s {{Associahedron}}.
\newblock May 2023.
\newblock \href {http://arxiv.org/abs/2305.08471} {\path{arXiv:2305.08471}}.

\bibitem[Rea03]{reading_latticeproperties_2003}
Nathan Reading.
\newblock Lattice and order properties of the poset of regions in a hyperplane
  arrangement.
\newblock {\em Algebra Universalis}, 50(2):179--205, 2003.
\newblock \href {https://doi.org/10.1007/s00012-003-1834-0}
  {\path{doi:10.1007/s00012-003-1834-0}}.

\bibitem[Rea04]{Reading-latticeCongruences}
Nathan Reading.
\newblock Lattice congruences of the weak order.
\newblock {\em Order}, 21(4):315--344 (2005), 2004.
\newblock \href {https://doi.org/10.1007/s11083-005-4803-8}
  {\path{doi:10.1007/s11083-005-4803-8}}.

\bibitem[Rea06]{Reading-cambrianLattices}
Nathan Reading.
\newblock Cambrian lattices.
\newblock {\em Adv. Math.}, 205(2):313--353, 2006.
\newblock \href {https://doi.org/10.1016/j.aim.2005.07.010}
  {\path{doi:10.1016/j.aim.2005.07.010}}.

\bibitem[Rea16]{Reading2016}
N.~Reading.
\newblock {\em Lattice Theory of the Poset of Regions}, pages 399--487.
\newblock Springer International Publishing, Cham, 2016.
\newblock \href {https://doi.org/10.1007/978-3-319-44236-5_9}
  {\path{doi:10.1007/978-3-319-44236-5_9}}.

\bibitem[{The}22]{SageMath2022}
{The Sage Developers}.
\newblock {\em {{SageMath}}, the {{Sage Mathematics Software System}}
  ({{Version}} 9.5)}, 2022.

\end{thebibliography}
\end{document}